\documentclass{article}

\pdfoutput=1
     \PassOptionsToPackage{numbers, compress}{natbib}

    \usepackage[final]{neurips_2022}

\usepackage[utf8]{inputenc} 
\usepackage[T1]{fontenc}    
\usepackage{hyperref}\hypersetup{
    colorlinks=true,
    linkcolor=black,
    filecolor=blue,      
    urlcolor=black,
    citecolor=black,
}   
\usepackage{url}            
\usepackage{booktabs}       
\usepackage{amsfonts}       
\usepackage{nicefrac}       
\usepackage{microtype}      
\usepackage{xcolor}         

\usepackage{galois}
\usepackage{pifont}
\usepackage[flushleft]{threeparttable}

\usepackage{enumerate}
\usepackage{graphicx}
\usepackage{colortbl}
\usepackage{amsmath}\allowdisplaybreaks
\usepackage{epstopdf}
\usepackage{mathrsfs}
\usepackage{amsthm,amssymb,todonotes}
\usepackage{appendix}
\usepackage{cases}
\usepackage{bm}

\usepackage{amsmath}\allowdisplaybreaks      
\usepackage{float}
\usepackage{subfigure,enumitem}

\usepackage{makecell}

\usepackage{wrapfig}

\providecommand{\argmin}{\mathop{\rm argmin}}

\usepackage{algorithm}
\usepackage[noend]{algpseudocode} 

\def\<{\left\langle} 
\def\>{\right\rangle}

\newtheorem{theorem}{Theorem}
\newtheorem{lemma}{Lemma}
\newtheorem{definition}{Definition}
\newtheorem{proposition}{Proposition}
\newtheorem{assumption}{Assumption}

\newtheorem{remark}{Remark}

\newcommand{\E}{\mathbb{E}}
\newcommand{\R}{\mathbb{R}}
\newcommand{\I}{\mathbf{I}}
\newcommand{\A}{\mathbf{A}}
\newcommand{\B}{\mathbf{B}}
\newcommand{\T}{\mathsf{T}}
\newcommand{\M}{\mathbf{M}}

\DeclareMathOperator{\poly}{poly}

\title{Zeroth-Order Negative Curvature Finding: Escaping Saddle Points  without Gradients}

\author{%
  Hualin Zhang$^{1}$~~~ Huan Xiong$^{2,3}$~~~  Bin Gu$^{1,3}$\\
  $^1$ Nanjing University of Information Science \& Technology\\
  $^2$ Harbin Institute of Technology\\
  $^3$  Mohamed bin Zayed University of Artificial Intelligence \\
  \texttt{\{zhanghualin98,huan.xiong.math,jsgubin\}@gmail.com}
 }

\begin{document}

\maketitle

\begin{abstract}
We consider escaping saddle points of nonconvex problems where only the function evaluations can be accessed. Although a variety of works have been proposed, the majority of them require either second or first-order information, and only a few of them have exploited zeroth-order methods, particularly the technique of negative curvature finding with zeroth-order methods which has been proven to be the most efficient method for escaping saddle points. To fill this gap,  in this paper, we propose two zeroth-order negative curvature finding frameworks that can replace Hessian-vector product computations without increasing the iteration complexity. We apply the proposed frameworks to ZO-GD, ZO-SGD, ZO-SCSG, ZO-SPIDER and prove that these ZO algorithms can converge to $(\epsilon,\delta)$-approximate second-order stationary points with less query complexity compared with prior zeroth-order works for finding local minima.

\end{abstract}

\section{Introduction}


Nonconvex optimization has received wide attention in recent years due to its popularity in modern machine learning (ML) and deep learning (DL) tasks. Specifically, in this paper, we study the following unconstrained optimization problem:
\begin{equation}
\label{eq: Problem}
    \min_{x\in \mathbb{R}^d} f(x) := \frac{1}{n}\sum_{i=1}^n f_i(x),
\end{equation}
where both $f_i(\cdot)$ and $f(\cdot)$ can be nonconvex. 
In general, finding the global optima of nonconvex functions is NP-hard. Fortunately, finding local optima is an alternative because it has been shown in theory and practice that local optima have comparable performance capabilities to global optima in many machine learning problems \cite{ge2015escaping,ge2016matrix,kawaguchi2016deep,ge2017optimization,ge2018learning,JMLR:v19:16-465,kawaguchi2019every}. Gradient-based methods have been shown to be able to find an $\epsilon$-approximate first-order stationary point ($\|\nabla f(x)\|\le \epsilon$) efficiently, both in the deterministic setting (\emph{e.g.}, gradient descent \cite{nesterov2003introductory}; accelerated gradient descent \cite{carmon2017convex,li2022restarted}) and stochastic setting (\emph{e.g.}, stochastic gradient descent \cite{nesterov2003introductory,reddi2016stochastic}; SCSG \cite{lei2017non}; SPIDER \cite{fang2018spider}). However, in nonconvex settings, first-order stationary points can be local minima, global minima, or even saddle points. Converging to saddle points will lead to highly suboptimal solutions \cite{jain2017global,sun2018geometric} and destroy the model's performance. Thus, escaping saddle points has recently become an important research topic in nonconvex optimization.

Several classical results have shown that, for $\rho$-Hessian Lipschitz functions (see Definition~\ref{def: Lipschitz}), using the second-order information like computing the Hessian \cite{nesterov2006cubic} or Hessian-vector products \cite{agarwal2017finding,carmon2018accelerated,allen2018natasha}, one can find an $\epsilon$-approximate second-order stationary point (SOSP, $\|\nabla f(x)\|\le \epsilon$ and $\nabla^2 f(x) \succeq -\sqrt{\rho \epsilon} \I$). However, when the dimension of $x$ is large, even once access to the Hessian is computationally infeasible. A recent line of work shows that, by adding uniform random perturbations,  first-order (FO) methods can efficiently escape saddle points and converge to SOSP. In the deterministic setting, \cite{jin2017escape} proposed the perturbed gradient descent (PGD) algorithm with gradient query complexity $\tilde{\mathcal{O}}(\log^4 d/\epsilon^2)$ by adding uniform random perturbation into the standard gradient descent algorithm. This complexity is later improved to $\tilde{\mathcal{O}}(\log^6 d / \epsilon^{1.75})$ by the perturbed accelerated gradient descent \cite{jin2018accelerated} which replaces the gradient descent step in PGD by Nesterov's accelerated gradient descent. 

\vspace{-10pt}
\begin{table*}[htb]
\label{tab: zo algorithms for finding SOSP} 
\caption{A summary of the results of finding $(\epsilon, \delta)$-approximate SOSPs (see Definition~\ref{def: stationary point}) by the zeroth-order algorithms. (CoordGE, GaussGE, and RandGE are abbreviations of ``coordinate-wise gradient estimator'', ``Gaussian random gradient estimator'' and ``uniform random gradient estimator'', respectively. RP, RS, and CR are abbreviations of ``random perturbation'', ``random search'' and ``cubic regularization'', respectively.)} 
	\begin{center}
    \setlength{\tabcolsep}{1mm}
    \begin{threeparttable}
	\begin{tabular}{lcccl}
		\toprule
		Algorithm   & Setting & ZO Oracle &  Main Techniques  &   Function Queries \\
		\midrule
		    ZPSGD \cite{jin2018local} & Deterministic & GaussGE + Noise & RP  &  $\tilde{\mathcal{O}}\left(\frac{d^2}{\epsilon^5} \right)$ \tnote{$\dag$}
			\\
			PAGD  \cite{vlatakis2019efficiently} & Deterministic & CoordGE  & RP & $\mathcal{O}\left( \frac{d \log^4 d}{\epsilon^2} \right)$ \tnote{$\dag$}
			\\
			RSPI \cite{lucchi2021second} & Deterministic & CoordGE & RS + NCF & $\mathcal{O}(\frac{d \log d}{\epsilon^{8/3}})$ \tnote{$\ddag$}
			\\
			\textbf{Theorem.~\ref{thm: ZO-GD}} & Deterministic & CoordGE & NCF &  $\mathcal{O}\left( \frac{d}{\epsilon^2} + \frac{d \log d}{\delta^{3.5}} \right)$ \\
		\midrule
			ZO-SCRN  \cite{balasubramanian2022zeroth} & Stochastic  & GaussGE & CR &  $\tilde{\mathcal{O}}\left( \frac{d}{\epsilon^{3.5}} +\frac{d^4}{\epsilon^{2.5}}  \right)$ \tnote{$\dag$}
			\\
			\textbf{Theorem.~\ref{thm: ZO-SGD}} & Stochastic & CoordGE & NCF  & $\tilde{\mathcal{O}}\left( \frac{d}{\epsilon^4} + \frac{d}{\epsilon^2 \delta^3} + \frac{d}{\delta^5}  \right)$
			\\
			 \textbf{Theorem.~\ref{thm: ZO-SCSG-NCF}} & Stochastic & CoordGE + (RandGE) & NCF & $\tilde{\mathcal{O}}\left( \frac{d}{\epsilon^{10/3}} + \frac{d}{\epsilon^2 \delta^3} + \frac{d}{\delta^5}  \right)$
			\\
			 \textbf{Theorem.~\ref{thm: ZO-SPIDER-NCF}} & Stochastic & CoordGE & NCF &  $\tilde{\mathcal{O}}\left( \frac{d}{\epsilon^{3}} + \frac{d}{\epsilon^2 \delta^2} + \frac{d}{\delta^5} \right)$
			\\
		\bottomrule
	\end{tabular}
	\begin{tablenotes}
        \item $\dag$ guarantees $(\epsilon, \mathcal{O}(\sqrt{\epsilon}))$-approximate SOSP, and
        $\ddag$ guarantees  $(\epsilon, \epsilon^{2/3})$-approximate SOSP.
    \end{tablenotes}
	\end{threeparttable}
	\end{center}
	\label{table_c}
\end{table*}
\vspace{-10pt}

Another line of work for escaping saddle points is to utilize the negative curvature finding (NCF), which can be combined with $\epsilon$-approximate first-order stationary point (FOSP) finding algorithms to find an ($\epsilon, \delta$)-approximate SOSP. The main task of NCF is to calculate the approximate smallest eigenvector of the Hessian for a given point. Classical methods for solving NCF like the power method and Oja's method need the computation of Hessian-vector products. Based on the fact the Hessian-vector product can be approximated by the finite difference between two gradients, \cite{xu2017neon+,allen2018neon2} proposed the FO NCF frameworks Neon+ and Neon2, respectively. In general, adding perturbations in the  negative curvature direction can escape saddle points more efficiently than adding random perturbations by a factor of $\tilde{\mathcal{O}}(\poly(\log d) )$ in theory. Specifically, in the deterministic setting, CDHS \cite{carmon2018accelerated} combined with Neon2 can find an ($\epsilon, \delta$)-approximate SOSP in gradient query complexity $\tilde{\mathcal{O}}(\log d / \epsilon^{1.75})$. Recently, the same result was achieved by a  simple single-loop algorithm \cite{zhang2021escape}, which combined the techniques of perturbed accelerated gradient descent and accelerated negative curvature finding. In the online stochastic setting, the best gradient query complexity result $\tilde{\mathcal{O}}(1/\epsilon^3)$ is achieved by SPIDER-SFO$^{+}$ \cite{fang2018spider}, which combined the near-optimal $\epsilon$-approximate FOSP finding algorithm SPIDER and the NCF framework Neon2 to find an $(\epsilon, \delta)$-approximate SOSP.

However, the gradient information is not always accessible. Many machine learning and deep learning  applications often encounter situations where the calculation of  explicit gradients is expensive or even infeasible, such as black-box adversarial attack on deep neural networks \cite{papernot2017practical,madry2018towards,chen2017zoo,bhagoji2018practical,tu2019autozoom} and policy search in reinforcement learning \cite{salimans2017evolution,choromanski2018structured,jing2021asynchronous}. Thus, zeroth-order (ZO) optimization, which uses function values to estimate the explicit  gradients as an important gradient-based black-box method, is one of the best options for solving this type of ML/DL problem. A considerable body of work has shown that ZO algorithms based on gradient estimation have comparable convergence rates to their gradient-based counterparts. Although many gradient estimation-based ZO algorithms have been proposed in recent years, most of them focus on the performance of converging to FOSPs \cite{nesterov2017random,ghadimi2013stochastic,ji2019improved,fang2018spider}, and only a few of them on SOSPs \cite{jin2018local,vlatakis2019efficiently,lucchi2021second,balasubramanian2022zeroth}. 

As mentioned above, although there have been several  works of finding local minima via ZO methods, they utilized the techniques of random perturbations \cite{jin2018local,vlatakis2019efficiently}, random search \cite{lucchi2021second}, and cubic regularization  \cite{balasubramanian2022zeroth}, as shown in Table~\ref{tab: zo algorithms for finding SOSP}, which are not the most efficient ones of escaping saddle points as discussed before. 
Specifically, in the deterministic setting, \cite{jin2018local} proposed the ZO perturbed stochastic gradient (ZPSGD) method, which uses a batch of Gaussian smoothing based stochastic ZO gradient estimators and adds a random perturbation in each iteration. As a result, ZPSGD can find an $\epsilon$-approximate SOSP using $\tilde{\mathcal
 O} \left(d^2 / \epsilon^5 \right)$ function queries. \cite{vlatakis2019efficiently} proposed the perturbed approximate gradient descent (PAGD) method which iteratively conducts the gradient descent steps by utilizing the forward difference version of the coordinate-wise gradient estimators until it reaches a point with a small gradient. Then, PAGD adds a uniform perturbation and continues the gradient descent steps. The total function queries of PAGD to find an $\epsilon$-approximate SOSP is $\tilde{\mathcal{O}} \left(d \log^4 d / \epsilon^2\right)$. Recently, \cite{lucchi2021second} proposed the random search power iteration (RSPI) method, which  alternately performs random search steps and power iteration steps. The power iteration step contains an inexact power iteration subroutine using only the ZO oracle to conduct the NCF, and the core idea is to use a finite difference approach to approximate the Hessian-vector product.
In the stochastic setting, \cite{balasubramanian2022zeroth} proposed a zeroth-order stochastic cubic regularization newton (ZO-SCRN) method with function query complexity $\tilde{\mathcal{O}}\left(d/\epsilon^{7/2}\right)$ using Gaussian sampling-based gradient estimator and Hessian estimator. Unfortunately, each iteration of ZO-SCRN needs to solve a cubic minimization subproblem, which does not have a closed-form solution. Typically, inexact solvers for solving the cubic minimization subproblem need additional computations of the Hessian-vector product \cite{agarwal2017finding} or the gradient \cite{carmon2016gradient}.

Thus, it is then natural to explore faster ZO negative curvature finding based algorithms to make escaping saddle points more efficient. To the best of our knowledge, negative curvature finding algorithms with access only to ZO oracle is still a vacancy in the stochastic setting. Inspired by the fact that the gradient can be approximated by the finite difference of function queries with high accuracy, a natural question is: \textit{Can we turn FO NCF methods (especially the state-of-the-art Neon2) into ZO methods without increasing the iteration complexity and turn ZO algorithms of finding FOSPs into the ones of finding SOSPs?}

\textbf{Contributions.} We summarize our main contributions as follows:
\begin{itemize}[leftmargin=0.2in]
    \item We give an affirmative answer to the above question. We propose two ZO negative curvature finding frameworks, which use only function queries and can detect whether there is a negative curvature direction at a given point $x$ on a smooth, Hessian-Lipschitz function $f: \R^d \to \R$ in offline deterministic and online stochastic settings, respectively.
    \item We apply the proposed frameworks to four ZO algorithms and prove that these ZO algorithms can converge to ($\epsilon, \delta$)-approximate SOSPs, which are ZO-GD, ZO-SGD, ZO-SCSG, and ZO-SPIDER. 
    \item In the deterministic setting, compared with the classical setting where $\delta = \mathcal{O}(\sqrt{\epsilon})$ \cite{jin2017escape,jin2018accelerated,jin2018local,vlatakis2019efficiently}, or the special case $\delta = \epsilon^{2/3}$ \cite{lucchi2021second}, our Theorem~\ref{thm: ZO-GD} is always not worse than other algorithms in Table~\ref{tab: zo algorithms for finding SOSP}. In the online stochastic setting,
    all of our algorithms don't need to solve the cubic subproblem as in ZO-SCRN and our Theorem~\ref{thm: ZO-SPIDER-NCF} improves the best function query complexity by a factor of $\tilde{\mathcal{O}}(1/\sqrt{\epsilon})$.
\end{itemize}

\section{Preliminaries}
Throughout this paper, we use $\|\cdot\|$ to denote the Euclidean norm of a vector and the spectral norm of a matrix. We use $\tilde{\mathcal{O}}(\cdot)$ to hide the poly-logarithmic terms. For a given set $\mathcal{S}$ drawn from $[n]:= \{1,2, \dots, n\}$, define $f_{\mathcal{S}}(\cdot):= \frac{1}{|\mathcal{S}|} \sum_{i \in \mathcal{S}} f_i(\cdot)$. 

\begin{definition}
\label{def: Lipschitz}
    For a twice differentiable nonconvex function $f$: $\mathbb{R}^d \to \mathbb{R}$,
    \begin{itemize}[leftmargin=0.2in]
        \item $f$ is $\ell$-Lipschitz smooth if 
        $\forall x,y \in \mathbb{R}^d, \|\nabla f(x) - \nabla f(y)\| \le \ell \|x-y\| $.
        \item $f$ is $\rho$-Hessian Lipschitz if
        $ \forall x,y \in \mathbb{R}^d, \|\nabla^2 f(x) - \nabla^2 f(y)\| \le \rho \|x-y\| $.
    \end{itemize}
\end{definition}

\begin{definition}
\label{def: stationary point}
For a twice differentiable nonconvex function $f$: $\mathbb{R}^d \to \mathbb{R}$, we say
\begin{itemize}[leftmargin=0.2in]
    \item $x\in \mathbb{R}^d$ is an $\epsilon$-approximate first-order stationary point if $\|\nabla f(x)\| \le \epsilon$.
    \item $x\in \mathbb{R}^d$ is an $(\epsilon,\delta)$-approximate second-order stationary point if $\|\nabla f(x)\| \le \epsilon, \nabla^2 f(x) \succeq -\delta \I$.
\end{itemize}
    
\end{definition}
We need the following assumptions which are standard in the literature of finding SOSPs \cite{allen2018neon2,fang2018spider,zhang2021escape}.
\begin{assumption}
\label{assum: basic}
    We assume that $f(\cdot)$ in \eqref{eq: Problem} satisfies:
    \begin{itemize}[leftmargin=0.2in]
        \item $\Delta_f := f(x_0) - f(x^*) < \infty$ where $x^* := \argmin_{x} f(x)$.
        \item Each component function $f_i(x)$ is $\ell$-Lipschitz smooth and $\rho$-Hessian Lipschitz.
        \item (For online case only) The variance of the stochastic gradient is bounded: $\forall x \in \mathbb{R}^d$, $\mathbb{E} \|\nabla f_i(x) -\nabla f(x)\|^2 \le \sigma^2$.
    \end{itemize}
\end{assumption}

We'll also need the following more stringent assumption to get high-probability convergence results of ZO-SPIDER.

\begin{assumption}
\label{assum: basic-high-probability}
We assume that Assumption~\ref{assum: basic} holds, and in addition, the gradient of each component function $f_i (x)$ satisfies $\forall i, x \in \mathbb{R}^d$, $ \|\nabla f_i(x) -\nabla f(x)\|^2 \le \sigma^2$.

\end{assumption}

\subsection{ZO Gradient Estimators}
Given a smooth, Hessian Lipschitz function $f$, a central difference version of the deterministic coordinate-wise gradient estimator is defined by
\begin{equation}
\label{eq: CoordGradEst}
\tag{CoordGradEst}
    \hat{\nabla}_{coord} f(x) = \sum_{i=1}^{d} \frac{f(x + \mu e_i) -f(x-\mu e_i)}{2 \mu} e_i,
\end{equation}
where $e_i$ denotes a standard basis vector with $1$ at its $i$-th coordinate and 0 otherwise; $\mu$ is the smoothing parameter, which is a sufficient small positive constant. A central difference version of the random gradient estimator is defined by
\begin{equation}
\label{eq: RandGradEst}
\tag{RandGradEst}
    \hat{\nabla}_{rand} f(x) = d \frac{f(x+\mu u) - f(x-\mu u)}{2\mu} u,
\end{equation}
where $u \in \R^d$ is a random direction drawn from a uniform distribution over the unit sphere; $\mu$ is the smoothing parameter, which is a sufficient small positive constant.

\begin{remark}
\textbf{Deterministic vs. Random}: \ref{eq: CoordGradEst} needs $d$ times more function queries than \ref{eq: RandGradEst}. However, as will be discussed in section~\ref{section: Applications of Zeroth-Order Negative Curvature Finding}, it has a lower approximation error and thus can reduce the iteration complexity. \textbf{Central Difference vs. Forward Difference} (please refer to Appendix A.1):
Under the assumption of Hessian Lipschitz, a smaller approximation error bound can be obtained by the central difference version of both \ref{eq: CoordGradEst} and \ref{eq: RandGradEst}. 
\end{remark}

\subsection{ZO Hessian-Vector Product Estimator}
By the definition of derivative: $\nabla^2 f(x)\cdot v = \lim_{\mu \to 0} \frac{\nabla f(x + \mu v) - \nabla f(x)}{\mu}$, we have $\nabla^2 f(x)\cdot v$ can be approximated by the difference of two gradients $\nabla f(x+v) - \nabla f(x)$ for some $v$ with small magnitude. On the other hand, $\nabla f(x+v), \nabla f(x)$ can be approximated by $\hat{\nabla }_{coord} f(x+v), \hat{\nabla}_{coord}f(x)$ with high accuracy, respectively. Then the coordinate-wise Hessian-vector product estimator is defined by:
\begin{equation}
    \label{eq: ZO-Hessian-Vector-Estimator}
    \mathcal{H}_f(x) v \triangleq \sum_{i=1}^d \frac{f(x+ v + \mu e_i) - f(x + v-\mu e_i) + f(x -\mu e_i) -f(x+ \mu e_i)}{2\mu}e_i.
\end{equation}
Note that we do not need to know the explicit representation of $\mathcal{H}_f(x)$. It is merely used as a notation for a virtual matrix and can be viewed as the Hessian $\nabla^2 f(x_0)$ with minor perturbations.  As stated in the following lemma, the approximation error is efficiently upper bounded.

\begin{lemma}
\label{lemma: error-bound-of-zo-hessian-vetor-product}
Assume that $f$ is $\rho$-Hessian Lipschitz, then for any smoothing parameter $\mu$ and $x\in \R^d$, we have
\begin{equation}
    \| \mathcal{H}_f(x) v - \nabla^2 f(x) v \|  \le \rho \left( \|v\|^2/2 + \sqrt{d}\mu^2/3 \right).
\end{equation}
    
\end{lemma}

The ZO Hessian-vector product estimator was previously studied in \cite{ye2018hessian,lucchi2021second}, but we provide a tighter bound than that in Lemma 6 in \cite{lucchi2021second}. This is because we utilize properties of the central difference version of the coordinate-wise gradient estimator under the Hessian Lipschitz assumption. It is then directly concluded that, if $f(\cdot)$ is quadratic, we have $\rho = 0$ and $\|\mathcal{H}_f(x) v - \nabla^2 f(x) v\| = 0$.

\section{Zeroth-Order Negative Curvature Finding}
In this section, we introduce how to find the negative curvature direction near the saddle point using zeroth-order methods. Recently, based on the fact that the Hessian-vector product $\nabla^2 f(x)\cdot v$ can be approximated by $\nabla f(x+v) - \nabla f(x)$ with approximation error up to $\mathcal{O}(\|v\|^2)$, \cite{allen2018neon2} proposed a FO framework named Neon2 that can replace the Hessian-vector product computations in NCF subroutine with gradient computations and thus can turn a FO algorithm for finding FOSPs into a FO algorithm for finding SOSPs. Enlightened by Neon2, we propose two zeroth-order NCF frameworks (\emph{i.e.}, \textit{ZO-NCF-Online} and \textit{ZO-NCF-Deterministic}) using only function queries to solve nonconvex problems in the online stochastic setting and offline deterministic setting, respectively. 

\subsection{Stochastic Setting}

In this subsection, we focus on solving the NCF problem with zeroth-order methods under the online stochastic setting and propose \textit{ZO-NCF-Online}. Before introducing \textit{ZO-NCF-Online}, we first introduce \textit{ZO-NCF-Online-Weak} with weak confidence of $2/3$ for  solving the NCF problem. 

We summarize \textit{ZO-NCF-Online-Weak} in Algorithm~\ref{alg: ZO-NCF-Online-weak}. Specifically, \textit{ZO-NCF-Online-Weak} consists of at most $T=\mathcal{O}(\frac{\log^2 d}{\delta^2})$ iterations and works as follows: Given a detection point $x_0$, add a random perturbation with small magnitude $\sigma$ as the starting point. At the $t$-th iteration where $t=1,\dots,T$, set $\mu_t = \|x_t - x_0\|$ to be the smoothing parameter $\mu$ in \eqref{eq: ZO-Hessian-Vector-Estimator}. Then we  keep updating $x_{t+1} = x_t - \eta \mathcal{H}_{f_i}(x_0) (x_t - x_0)$ where $\mathcal{H}_{f_i}(x_0) (x_t - x_0)$ is the ZO Hessian-vector product estimator and stops whenever $\|x_{t+1} - x_0\| \ge r$ or the maximum iteration number $T$ is reached. Thus as long as Algorithm~\ref{alg: ZO-NCF-Online-weak} does not terminate, we have that the approximation error $\|\mathcal{H}_{f_i}(x_0) (x_t - x_0) - \nabla^2 f_i (x_0) (x_t - x_0) \|$ can be bounded by $\mathcal{O} (\sqrt{d} r^2)$ according to Lemma~\ref{lemma: error-bound-of-zo-hessian-vetor-product}. Note that, although the error bound is poorer by a factor of $\mathcal{O}(\sqrt{d})$ as compared to $Neon^{online}_{weak}$ in \cite{allen2018neon2} which used the difference of two gradients to approximate the Hessian-vector product and achieve an approximation error up to $\mathcal{O}(r^2)$, with our choice of $r$ in Algorithm~\ref{alg: ZO-NCF-Online-weak}, the error term is still efficiently upper bounded.

\begin{algorithm}[htb]
	\caption{ZO-NCF-Online-Weak ($f$, $x_0$, $\delta$)}
	\label{alg: ZO-NCF-Online-weak}
	\renewcommand{\algorithmicrequire}{\textbf{Initialization:}}
	\renewcommand{\algorithmicensure}{\textbf{Return}}
	\begin{algorithmic}[1]
		\State $\eta \gets \frac{\delta}{C_0^2 \ell^2 \log(100d)}$, $T \gets \frac{C_0^2 \log(100d)}{\eta \delta}$, $\sigma \gets \frac{\eta^2 \delta^3}{(100d)^{3C_0}\rho}$, $r \gets (100d)^{C_0} \sigma$
		\State $\xi \gets \sigma \frac{\xi'}{\|\xi'\|}$, with $ \xi' \sim \mathcal{N}(0,\mathbf{I})$
		\State $x_1 \gets x_0 + \xi$
		\For{$t=1,\dots,T$}
		\State $\mu_t \gets \|x_t - x_0\|$
		\State	
		    $x_{t+1} = x_t - \eta \mathcal{H}_{f_i}(x_0) (x_t - x_0)$ with $\mu = \mu_t$ and $i \in [n]$
		\If{$\|x_{t+1} - x_0\| \ge r$} 
		\Return $v=\frac{x_s-x_0}{\|x_s-x_0\|}$ for a uniformly random $s\in[t]$ 
		\EndIf
		\EndFor
		\Ensure $v=\bot$
	\end{algorithmic}
\end{algorithm}

Other than the additional error term caused by ZO approximation, the motivation  of \textit{ZO-NCF-Online-Weak} is almost the same as $Neon^{online}_{weak}$. That is, under reasonable control of the approximation error of the Hessian-vector product, using the update rule of Oja's method \cite{oja1982simplified} to approximately calculate the eigenvector corresponding to the minimum eigenvalue of $\nabla^2 f(x_0)=\frac{1}{n}\sum_{i=1}^n \nabla^2 f_i(x_0)$. Under similar analysis, we conclude that as long as the minimum eigenvalue of $\nabla^2 f(x_0)$ satisfies $\lambda_{min} (\nabla^2 f(x_0)) \le -\delta$, \textit{ZO-NCF-Online-Weak} will stop before $T$ and find a negative curvature direction that aligns well with the eigenvector corresponding to the minimum eigenvalue of $\nabla f^2(x_0)$.
Then we have the following lemma:

\begin{lemma}[ZO-NCF-Online-Weak]
\label{thm: online-weak}
    The output $v$ of Algorithm~\ref{alg: ZO-NCF-Online-weak} satisfies: If $\lambda_{min }(\nabla^2 f(x_0)) \le -\delta$, then with probability at least $2/3$, $v \neq \bot$ and $v^\T \nabla^2 f(x_0)v \le -\frac{3}{4}\delta$.
\end{lemma}

\begin{algorithm}[htb]
	\caption{ZO-NCF-Online}
	\label{alg: ZO-NCF-Online}
	\renewcommand{\algorithmicrequire}{\textbf{Input:}}
	\renewcommand{\algorithmicensure}{\textbf{Return}}
	\begin{algorithmic}[1]
		\Require $f(\cdot)=\frac{1}{n}\sum_{i=1}^nf_i(\cdot)$, $x_0$, $\delta>0$, $p\in(0,1]$.
		\For{$j=1,2,\cdots,\Theta(\log(1/p))$}
		\State $v_j \gets$ \hyperref[alg: ZO-NCF-Online-weak]{ZO-NCF-Online-Weak} ($f, x_0, \delta$)
		\If{$v_j \neq \bot$}
		\State $m \gets \Theta(\frac{\ell^2 \log(1/p)}{\delta^2}) , v' \gets \Theta(\frac{\delta}{d \rho})v_j$
		\State Draw $i_1,\dots,i_m$ uniformly randomly from $[n]$
		\State $z_j = \frac{1}{m} \sum_{k=1}^m \frac{(v')^T \mathcal{H}_{f_{i_k}}(x_0) v'}{\|v'\|^2} $
		\If{$z_j \le -\frac{3\delta}{4}$} 
		\Return $v= v_j$
        \EndIf
        \EndIf
        \EndFor
		\Ensure $v=\bot$
	\end{algorithmic}
\end{algorithm}

We summarize \textit{ZO-NCF-Online} in Algorithm~\ref{alg: ZO-NCF-Online}. Specifically, \textit{ZO-NCF-Online} repeatedly calls \textit{ZO-NCF-Online-Weak} for $\Theta(\log(1/p))$ times to boost the confidence of solving the NCF problem from $2/3$ to $1-p$. We have the following results:

\begin{lemma}
\label{thm: online}
    In the same setting as in Algorithm~\ref{alg: ZO-NCF-Online}, define $z = \frac{1}{m}\sum_{j=1}^{m} v^\T (\mathcal{H}_{f_{i_j}} (x_0))v$.
    Then, if $\|v\|\le \frac{\delta}{16d \rho}$ and $m=\Theta(\frac{\ell^2 }{\delta^2})$, with probability at least $1-p$, we have $\left| \frac{z}{\|v\|^2} - \frac{v^\T \nabla^2 f(x) v}{\|v\|^2} \right|\le \frac{\delta}{4}$.
\end{lemma}

\begin{theorem}
\label{thm: ZO-NCF-Online}
    Let $f(x) = \frac{1}{n} \sum_{i=1}^{n} f_i(x)$ where each $f_i$ is $\ell$-smooth and $\rho$-Hessian Lipschitz. For every point $x_0 \in \R^d$, every $\delta \in (0,\ell]$, the output of Algorithm~\ref{alg: ZO-NCF-Online} $v$ satisfies that, with probability at least $1-p$: If $v = \bot$, then $\nabla^2 f(x_0) \succeq -\delta \I$; If $v \neq \bot$, then $\|v\|=1$ and $v^\T \nabla^2 f(x_0) v \le -\frac{\delta}{2}$.
    The total function query complexity is 
    \begin{align*}
        \mathcal{O}\left( \frac{d \log^2 (d/p) \ell^2}{\delta^2}\right).
    \end{align*}
\end{theorem}

\subsection{Deterministic Setting}

In this subsection, we focus on solving the NCF problem with zeroth-order methods under the offline deterministic setting and propose \textit{ZO-NCF-Deterministic}. We summarize \textit{ZO-NCF-Deterministic} in Algorithm~\ref{alg: ZO-NCF-Deterministic}. Since we want to  compute the eigenvector corresponding to the most negative eigenvalue of $\nabla^2 f(x_0)$ approximately, one can convert it into an approximated top eigenvector computation problem of $\M:= -\frac{1}{\ell} \nabla^2 f(x_0) + (1-\frac{3 \delta}{4 \ell}) \I$. This is because all eigenvalues of $\nabla^2 f(x_0)$ in $[-\frac{3\delta}{4}, \ell]$ will be mapped to eigenvalues of $\M$ in $[-1,1]$, and all eigenvalues of $\nabla^2 f(x_0)$ smaller than $-\delta$ will be mapped to eigenvalues of $\M$ greater than $1 + \frac{\delta}{4\ell}$.

\begin{algorithm}[htb]
	\caption{ZO-NCF-Deterministic}
	\label{alg: ZO-NCF-Deterministic}
	\renewcommand{\algorithmicrequire}{\textbf{Input:}}
	\renewcommand{\algorithmicensure}{\textbf{Return}}
	\begin{algorithmic}[1]
		\Require Function $f(\cdot)$, point $x_0$, negative curvature $\delta>0$, confidence $p\in(0,1]$.
		\State $T \gets \frac{C_1^2 \log \frac{d}{p} \sqrt{\ell}}{\sqrt{\delta}}, \sigma \triangleq \left(d/p\right)^{-2 C_1} \frac{\delta}{T^4 \rho}, r \triangleq (d/p)^{C_1}\sigma$
		
		\State $\xi \gets \sigma \frac{\xi'}{\|\xi'\|}$, with $ \xi' \sim \mathcal{N}(0,\mathbf{I})$
		
		\State $x_1 \gets x_0 + \xi$, $y_0 \gets 0, y_1 \gets \xi$
		
		\For{$t=1,\dots,T$}
		\State $\mu_t = \|y_t\|$
		\State	
		    $y_{t+1} = 2\mathcal{M}(y_t)-y_{t-1}$ where $\mathcal{M}(y)= (-\frac{1}{\ell}\mathcal{H}_f(x_0)+(1-\frac{3\delta}{4\ell}))y$
		\State $x_{t+1} = x_0 + y_{t+1} - \mathcal{M}(y_t)$
		\If{$\|x_{t+1} - x_0\| \ge r$} 
		\Return $v=\frac{x_{t+1}-x_0}{\|x_{t+1}-x_0\|}$
		\EndIf
		\EndFor
		\Ensure $v= \bot$.
	\end{algorithmic}
\end{algorithm}

Similar to \textit{ZO-NCF-Online-Weak}, \textit{ZO-NCF-Deterministic} starts by adding a random perturbation $\xi$ to the detection point $x_0$. To find the negative curvature direction $v$ of $\nabla^2 f(x_0)$ such that $v^\T \nabla^2 f(x_0) v \le -\frac{\delta}{2}$, the classical power method which updates through $x_{T+1} = x_0 + \M^T \xi$ \cite{lucchi2021second} will take $T \ge \tilde{\Omega}(\frac{\ell}{\delta})$ number of iterations since eigenvalues of $\M$ greater than $1+\frac{\delta}{4\ell}$ grows in a speed $(1+\delta/\ell)^T$. To reduce the iteration complexity $T$, we can replace the  matrix polynomial $\M^T$ with the matrix Chebyshev polynomial $\mathcal{T}_T (\M)$ and virtually update $x_{T+1} = x_0 + \mathcal{T}_T(\M) \xi$.
\begin{definition}
Chebyshev polynomial $\{\mathcal{T}_n(x)\}_{n\ge 0}$ of the first kind is 
\begin{equation*}
    \mathcal{T}_0(x) = 1, \quad \mathcal{T}_1(x) = x, \quad \mathcal{T}_{n+1}(x) = 2x\cdot \mathcal{T}_n(x) - \mathcal{T}_{n-1}(x),
\end{equation*}
then it satisfies $\mathcal{T}_t(x) = \begin{cases}
    \cos (n\arccos (x)), & x \in[-1,1] \\
    \frac{1}{2} [(x-\sqrt{x^2-1})^n + (x+ \sqrt{x^2-1})^n], & x > 1 
    \end{cases}$.
\end{definition}
In the matrix case, we have the so-called matrix Chebyshev polynomial $\mathcal{T}_t(\M)$ \cite{allen2017faster}, which satisfies: $\mathcal{T}_{t+1}(\M)\xi = 2\M \mathcal{T}_t (\M) \xi - \mathcal{T}_{t-1}(\M)\xi$. Thus, eigenvalues of $\M$ greater than $1 + \frac{\delta}{4\ell}$ will grow to $ \left(1 + \delta/4\ell + \sqrt{(\delta/4\ell)^2 + \delta/2\ell} \right)^T \approx  \left( 1 + \sqrt{\delta/\ell} \right)^T$, so we only need to choose $T\ge \sqrt{\ell/\delta}$.

On the other hand, since we only have access to the zeroth-order information, we need to stably compute the matrix Chebyshev polynomial. In algorithm~\ref{alg: ZO-NCF-Deterministic}, we set $\mu_t = \|y_t\|$ and use $\mathcal{M}(y_t) = (-\frac{1}{\ell}\mathcal{H}_f(x_0)+(1-\frac{3\delta}{4\ell}))y_t$ to approximate $\M y_t$ with approximation error up to $\frac{2\rho\sqrt{d} r t}{\ell} \|y_t\|$. With proper choice of $r$, it allows us to use the inexact backward recurrence \cite{allen2017faster} to ensure a stable computation of matrix Chebyshev polynomial:
\begin{equation*}
    y_0 = 0, \quad y_1 =\xi, \quad y_{t+1} = 2 \mathcal{M}(y_t) - y_{t-1}.
\end{equation*}
Then the output $x_{T+1} = x_0 + y_{T+1} - \mathcal{M}(y_T)$ is close to $x_0 + \mathcal{T}_T(\M)\xi$ with a small approximation error. Finally, we have the following theorem:

\begin{theorem}
\label{thm: deterministic}
Let $f(x) = \frac{1}{n} \sum_{i=1}^{n} f_i(x)$ where each $f_i$ is $\ell$-smooth and $\rho$-Hessian Lipschitz. For every point $x_0 \in \R^d$, every $\delta \in (0,\ell]$, the output of Algorithm~\ref{alg: ZO-NCF-Deterministic} $v$ satisfies that, with probability at least $1-p$: If $v = \bot$, then $\nabla^2 f(x_0) \succeq -\delta \I$; If $v \neq \bot$, then $\|v\|=1$ and $v^\T \nabla^2 f(x_0) v \le -\frac{\delta}{2}$.
The function query complexity is 
\begin{align*}
    \mathcal{O}( \frac{d \log \frac{d}{p} \sqrt{\ell}}{\sqrt{\delta}}).
\end{align*}
\end{theorem} 

\section{Applications of Zeroth-Order Negative Curvature Finding}
\label{section: Applications of Zeroth-Order Negative Curvature Finding}
In this section, we focus on applying the zeroth-order negative curvature frameworks to the following ZO algorithms: ZO-GD, ZO-SGD, ZO-SCSG, and ZO-SPIDER. The following result shows that one can verify if a point $x$ is an $\epsilon$-approximate FOSP using \ref{eq: CoordGradEst}.
\begin{proposition}
\label{proposition: verify-gradient-online}
In the online setting, using \ref{eq: CoordGradEst} with a batch size of $\mathcal{O}\left( \left(\frac{128\sigma^2}{\epsilon^2}+1 \right)\log \frac{1}{p} \right)$ and smoothing parameter $\mu \le \sqrt{\frac{3\epsilon}{4 \rho \sqrt{d}}}$, we can verify with probability at least $1-p$, either $\|\nabla f(x)\|\ge \epsilon/2$ or $\|\nabla f(x)\| \le \epsilon$. In the deterministic setting, using once computation of \ref{eq: CoordGradEst} with smoothing parameter $\mu \le \sqrt{\frac{3\epsilon}{2\rho\sqrt{d}}}$, we can verify with probability $1$, either $\|\nabla f(x)\|\ge \epsilon/2$ or $\|\nabla f(x)\| \le \epsilon$.
\end{proposition}

\subsection{Applying Zeroth-Order Negative Curvature Finding to ZO-GD and ZO-SGD}

We apply ZO-NCF-Online to ZO-SGD to turn it into a local minima finding algorithm, and propose ZO-SGD-NCF in Algorithm~\ref{alg: ZO-SGD}. At each iteration, we use a batch size of $\mathcal{O}\left(\frac{\sigma^2}{\epsilon^2}\log\left(\frac{2K}{p}\right) \right)$ \ref{eq: CoordGradEst} to verify if $x_t$ is an $\epsilon$-approximate stationary point. If not, ZO-SGD-NCF either estimates the gradient $\nabla f_S (x_t) = \frac{1}{|S|} \sum_{i\in S} \nabla f_i(x_t)$ by \ref{eq: CoordGradEst} (\textbf{Option \uppercase\expandafter{\romannumeral1}}) or \ref{eq: RandGradEst} (\textbf{Option \uppercase\expandafter{\romannumeral2}}) with both mini-batch size $\mathcal{O}(\frac{\sigma^2}{\epsilon^2})$; If so, we call the ZO-NCF-Online subroutine. Then, If we find an approximate negative curvature direction $v$ around $x_t$, then we update $x_{t+1}$ by moving from $x_t$ in the direction $v$ with step-size $\delta/\rho$. We have the following theorem:

\begin{algorithm}[htb]
	\caption{ZO-SGD-NCF}
	\label{alg: ZO-SGD}
	\renewcommand{\algorithmicrequire}{\textbf{Input:}} 
	\renewcommand{\algorithmicensure}{\textbf{Output:}}
	\begin{algorithmic}[1]
		\Require Function $f$, starting point $x_0$, confidence $p\in(0,1)$, $\epsilon>0$ and $\delta>0$.
		\For{$t = 0,\dots,K-1$}
		\State uniformly randomly choose a set $\mathcal{B}$ with batch size $\mathcal{O}(\frac{\sigma^2}{\epsilon^2} \log(2K/p))$
		\If{$\|\hat{\nabla}_{coord} f_{\mathcal{B}}(x_t)\|\ge \frac{3\epsilon}{4}$ }  
		\State uniformly randomly choose $S \subseteq[n]$  
		\State \textbf{Option \uppercase\expandafter{\romannumeral1} : } $x_{t+1} \gets x_t - \eta   \hat{\nabla}_{coord}f_S (x_t)$
		\State \textbf{Option \uppercase\expandafter{\romannumeral2} : } $x_{t+1} \gets x_t - \eta  \hat{\nabla}_{rand}f_S (x_t)$
	    \Else 
	    \State $v \gets$ ZO-NCF-Online ($f, x_t, \delta, \frac{p}{2K}$)
	    \If{$v=\bot$}
	    \Return $x_t$
	    \Else
	    \quad $x_{t+1} = x_t \pm \frac{\delta}{\rho} v$
	    \EndIf
	    \EndIf
	    \EndFor
	\end{algorithmic}
\end{algorithm}

\begin{theorem}
\label{thm: ZO-SGD}
Under Assumption~\ref{assum: basic}, we set $\mu_1 = \sqrt{\frac{3\epsilon}{2\rho\sqrt{d}}}$ and other parameters as follows, 
\begin{align*}
    \textbf{Option \uppercase\expandafter{\romannumeral1}:} |S| & = \max\{\frac{32\sigma^2 }{\epsilon^2}, 1\}, K=\mathcal{O}(\frac{\rho^2 \Delta_f}{\delta^3} + \frac{\ell \Delta_f}{\epsilon^2}), \eta = \frac{1}{4\ell}, \mu_2 = \sqrt{\frac{3\epsilon}{4\rho\sqrt{d}}}; \\
    \textbf{Option \uppercase\expandafter{\romannumeral2}:}|S| & = \max\{\frac{8\sigma^2 }{\epsilon^2}, 1\}, K = \mathcal{O}(\frac{\rho^2 \Delta_f}{\delta^3} + \frac{d \ell \Delta_f}{\epsilon^2}), \eta = \frac{1}{32d\ell}, \mu_2 = \min \left\{\sqrt{\frac{3\epsilon}{4\rho d}}, \frac{\epsilon}{32 \sqrt{d} \ell} \right\}, 
\end{align*}
where $\mu_1$ and $\mu_2$ are only used in Line 3 and Line 5 (or Line 6) of Algorithm~\ref{alg: ZO-SGD}, respectively. With probability at least $1-p$, Algorithm~\ref{alg: ZO-SGD} outputs an $(\epsilon,\delta)$-approximate local minimum in function query complexity 
\begin{equation*}
    \textbf{Option \uppercase\expandafter{\romannumeral1}:} \tilde{\mathcal{O}} (\frac{d \sigma^2 \ell \Delta_f}{\epsilon^4} + \frac{d \sigma^2 \rho^2 \Delta_f}{ \epsilon^2 \delta^3} + \frac{d \ell^2 \rho^2 \Delta_f}{\delta^5} );
    \textbf{Option \uppercase\expandafter{\romannumeral2}:} 
    \tilde{ \mathcal{O}} (\frac{d^2 \sigma^2 \ell \Delta_f}{\epsilon^4} + \frac{d \sigma^2 \rho^2 \Delta_f}{\epsilon^2 \delta^3} + \frac{d \ell^2 \rho^2 \Delta_f}{\delta^5} ).
\end{equation*}
\end{theorem}

\begin{algorithm}[htb]
	\caption{ZO-GD-NCF}
	\label{alg: ZO-GD}
	\renewcommand{\algorithmicrequire}{\textbf{Input:}} 
	\renewcommand{\algorithmicensure}{\textbf{Output:}}
	\begin{algorithmic}[1]
		\Require Function $f$, starting point $x_0$, confidence $p\in(0,1)$, $\epsilon>0$ and $\delta>0$.
		\For{$t = 0,\dots,K-1$}
		\If{$\|\hat{\nabla}_{coord} f(x_t)\|\ge \frac{3 \epsilon}{4}$ }
        \State \textbf{Option \uppercase\expandafter{\romannumeral1} : } $x_{t+1} \gets x_t - \eta \hat{\nabla}_{coord}f (x_t)$
		\State \textbf{Option \uppercase\expandafter{\romannumeral2} : } $x_{t+1} \gets x_t - \eta \hat{\nabla}_{rand}f (x_t)$
	    \Else 
	    \State $v \gets$ ZO-NCF-Deterministic ($f, x_t, \delta, \frac{p}{K}$)
	    \If{$v=\bot$}
	    \Return $x_t$
	    \Else
	    \quad $x_{t+1} = x_t \pm \frac{\delta}{\rho} v$
	    \EndIf
	    \EndIf
	    \EndFor
	\end{algorithmic}
\end{algorithm}

\begin{remark}
\label{remark: ZO-SGD}
Note that the dominant term of the function query complexity in \textbf{Option \uppercase\expandafter{\romannumeral1}} is $\tilde{\mathcal{O}}(\frac{d}{\epsilon^4})$, while in \textbf{Option \uppercase\expandafter{\romannumeral2}} is $\tilde{\mathcal{O}}(\frac{d^2}{\epsilon^4})$. This is because \ref{eq: CoordGradEst} has a lower approximation error and thus can reduce the iteration complexity by a factor of $d$. Then the function query complexity of \textbf{Option \uppercase\expandafter{\romannumeral2}} is dominated by evaluating the magnitude of the gradient (Line 3 in Algorithm~\ref{alg: ZO-SGD}). 

\end{remark}

In the Deterministic setting, we  apply ZO-NCF-Deterministic to ZO-GD to turn it into a local minima finding algorithm and propose ZO-GD-NCF in Algorithm~\ref{alg: ZO-GD}. The update rule of ZO-GD-NCF is similar to that in ZO-SGD-NCF, the only difference is that we don't need to use mini-batch sampling of the stochastic gradient. Similarly, we have the following theorem:

\begin{theorem}
\label{thm: ZO-GD}
Under Assumption~\ref{assum: basic}, we set $\mu_1 = \sqrt{\frac{3\epsilon}{2\rho\sqrt{d}}}$ and other parameters as follows,
\begin{align*}
    \textbf{Option \uppercase\expandafter{\romannumeral1}: } K & =\mathcal{O}(\frac{\rho^2 \Delta_f}{\delta^3} + \frac{\ell \Delta_f}{\epsilon^2}), \eta = \frac{1}{4\ell}, \mu_2 = \sqrt{\frac{3\epsilon}{4\rho\sqrt{d}}}; \\
    \textbf{Option \uppercase\expandafter{\romannumeral2}: } K & = \mathcal{O}(\frac{\rho^2 \Delta_f}{\delta^3} + \frac{d \ell \Delta_f}{\epsilon^2}), \eta = \frac{1}{8d\ell}, \mu_2 = \min \left\{\sqrt{\frac{3\epsilon}{4\rho d}}, \frac{\epsilon}{16 \sqrt{d} \ell} \right\},
\end{align*}
where $\mu_1$ and $\mu_2$ are only used in Line 2 and Line 3 (or Line 4) of Algorithm~\ref{alg: ZO-GD}, respectively. With probability at least $1-p$, Algorithm~\ref{alg: ZO-GD} outputs an $(\epsilon,\delta)$-approximate local minimum in function query complexity 
\begin{equation*}
    \textbf{Option \uppercase\expandafter{\romannumeral1}:}\quad \tilde{\mathcal{O}} ( \frac{d \ell \Delta_f}{\epsilon^2} + d \frac{\sqrt{\ell}}{\sqrt{\delta}} \frac{ \rho^2 \Delta_f}{\delta^3} ); \quad
    \textbf{Option \uppercase\expandafter{\romannumeral2}:}\quad \tilde{\mathcal{O}} ( \frac{d^2\ell \Delta_f}{\epsilon^2} + d \frac{\sqrt{\ell}}{\sqrt{\delta}} \frac{\rho^2 \Delta_f}{\delta^3} ).
\end{equation*}
\end{theorem}

\subsection{Applying Zeroth-Order Negative Curvature Finding to ZO-SCSG and ZO-SPIDER}
In the stochastic setting, we can also apply the zeroth-order negative curvature finding to the variance reduction-based algorithms: SCSG \cite{lei2017non} and SPIDER \cite{fang2018spider}. Due to space limitation, We defer the detailed discussions of these applications to Appendix E and F.

To apply ZO-NCF-Online to SCSG, we first propose a zeroth-order variant of the SCSG \cite{lei2017non} method in Algorithm~\ref{alg: ZO-SCSG}. At the beginning of the $j$-th epoch, we estimate the gradient $\nabla f_{\mathcal{I}_j} (\tilde{x}_{j-1})$ by \ref{eq: CoordGradEst} over a batch sampling set $\mathcal{I}_j$ with size $B$. In the inner loop iterations, the stochastic gradient estimator $v_{k-1}^j$ is either constructed by \ref{eq: CoordGradEst} or by \ref{eq: RandGradEst} over a mini-batch sampling set $\mathcal{I}_{k-1}^j$ with size $b$. Then we apply ZO-NCF-Online to ZO-SCSG and propose the ZO-SCSG-NCF method (see Algorithm 7).

\begin{theorem}[informal, full version deferred to Appendix E]
\label{thm: ZO-SCSG-NCF}
With probability at least $\frac{2}{3}$, for both \textbf{Option \uppercase\expandafter{\romannumeral1}} and \textbf{Option \uppercase\expandafter{\romannumeral2}}, Algorithm~\ref{alg: ZO-SCSG-NCF} outputs an $(\epsilon, \delta)$-approximate local minimum in function query complexity 
\begin{equation*}
    \tilde{\mathcal{O}}( d ( \frac{\ell  \Delta_f }{\epsilon^\frac{4}{3} \sigma^{\frac{2}{3}} } + \frac{\rho^2 \Delta_f}{\delta^3} ) ( \frac{\sigma^2 }{\epsilon^2} + \frac{\ell^2}{\delta^2} ) + d \frac{\ell \Delta_f}{\epsilon^2} \frac{\ell^2}{\delta^2}  ).
\end{equation*}

\end{theorem}

We apply ZO-NCF-Online to ZO-SPIDER to turn it into a local minima finding algorithm and propose ZO-SPIDER-NCF in Algorithm~\ref{alg: ZO-SPIDER-NCF}. As a by-product, we also propose a zeroth-order variant of the SPIDER method in Appendix G that  can converge to an $\epsilon$-approximate FOSP with high probability rather than expectation. Using the same technique as in SPIDER-SFO$^{+}$ \cite{fang2018spider}, that is, instead of moving in a large single step with size $\delta/\rho$ along the approximate negative curvature direction as in ZO-SGD-NCF and ZO-SCSG-NCF, we can split it into $\delta/(\rho \eta)$ equal length mini-steps with size $\eta$. As a result, we can maintain the SPIDER estimates and improve the so-called  non-improvable coupling term $\frac{1}{\delta^3 \epsilon^2}$ by a fact of $\delta$.

\begin{theorem}[informal, full version deferred to Appendix F]
\label{thm: ZO-SPIDER-NCF}
With probability at least $\frac{3}{4}$, Algorithm~\ref{alg: ZO-SPIDER-NCF} outputs an $(\epsilon, \delta)$-approximate local minimum in function query complexity
\begin{equation*}
    \tilde{\mathcal{O}} \left( d \left( \frac{\sigma \ell \Delta_f}{\epsilon^3} + \frac{\sigma \ell \rho \Delta_f }{\epsilon^2 \delta^2} + \frac{\ell^2 \rho \Delta_f}{\delta^3 \epsilon} + \frac{\ell^2 \rho^2 \Delta_f}{\delta^5} + \frac{\sigma^2}{\epsilon^2} + \frac{\sigma \delta \ell}{\rho \epsilon^2} + \frac{\ell^2 }{\delta^2} \right) \right).
\end{equation*}

\end{theorem}

\begin{remark}
We can boost the confidence the of Theorem~\ref{thm: ZO-SCSG-NCF} and \ref{thm: ZO-SPIDER-NCF} to $1-p$ by running $\log(1/p)$ copies of Algorithm 7 and 8.
\end{remark}

\section{Numerical Experiments}

\textbf{Octopus Function.} We first consider the octopus function proposed by Du et al. \cite{du2017gradient}. The octopus function has $2^d$ local optimum: $x^* = (\pm 4\tau, \dots, \pm 4\tau)^\T$ and $2^{d}-1$ saddle points:
\begin{equation*}
    (0, \dots, 0)^\T, (\pm 4\tau, 0, \dots, 0)^\T, \dots, (\pm 4\tau, \dots, \pm 4\tau, 0)^\T.
\end{equation*}
We compare ZO-GD-NCF, ZPSGD, PAGD, and RSPI on the octopus function with growing dimensions. The parameters corresponding to the octopus function are set with $\tau = e, L = e, \gamma =1$. All algorithms are initialized at point $(0, \dots, 0)^\T$, which is a strict saddle point and the one farthest from the optimal points among the $2^d - 1$ saddle points. 

We set $\epsilon = 1e-4, \delta = \sqrt{\rho \epsilon}$ for all experiments  and report the function value v.s. the number of function queries in Figure \ref{fig: octopus}. For RSPI, we follow the hyperparameter update strategy as described in (\cite{lucchi2021second}, Appendix, Section F): We keep $\sigma_2$ constant and update $\sigma_1 = \rho_{\sigma_1}\sigma_1$ every $T_{\sigma_1}$ iterations. We conduct a grid search for $T_{\sigma_1}$ and $\rho_{\sigma_1}$.


\begin{figure}[!htb]
\centering
\subfigure[d=10]{
\centering
\includegraphics[width=0.23\textwidth]{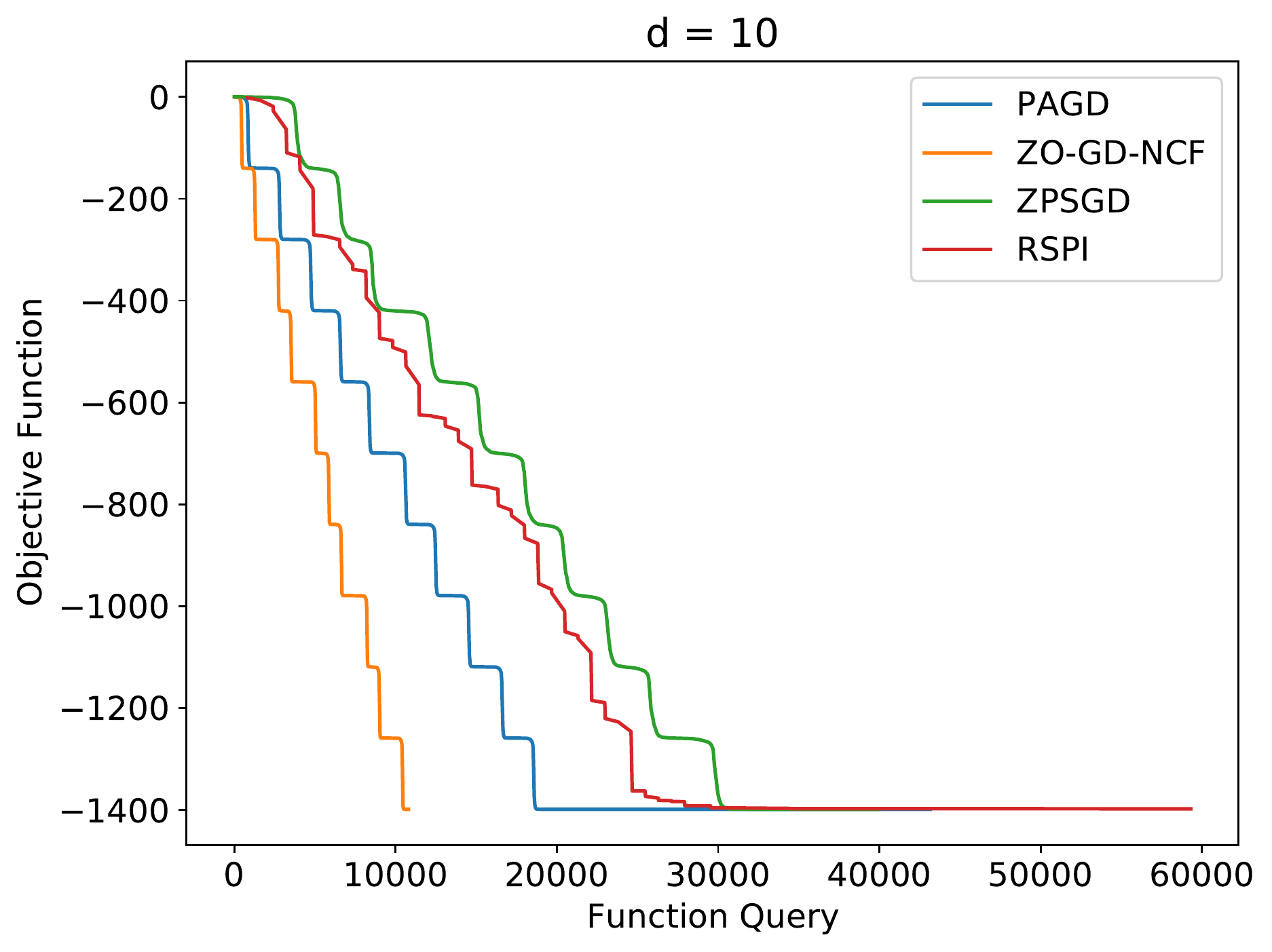}
}
\subfigure[d=30]{
\centering
\includegraphics[width=0.23\textwidth]{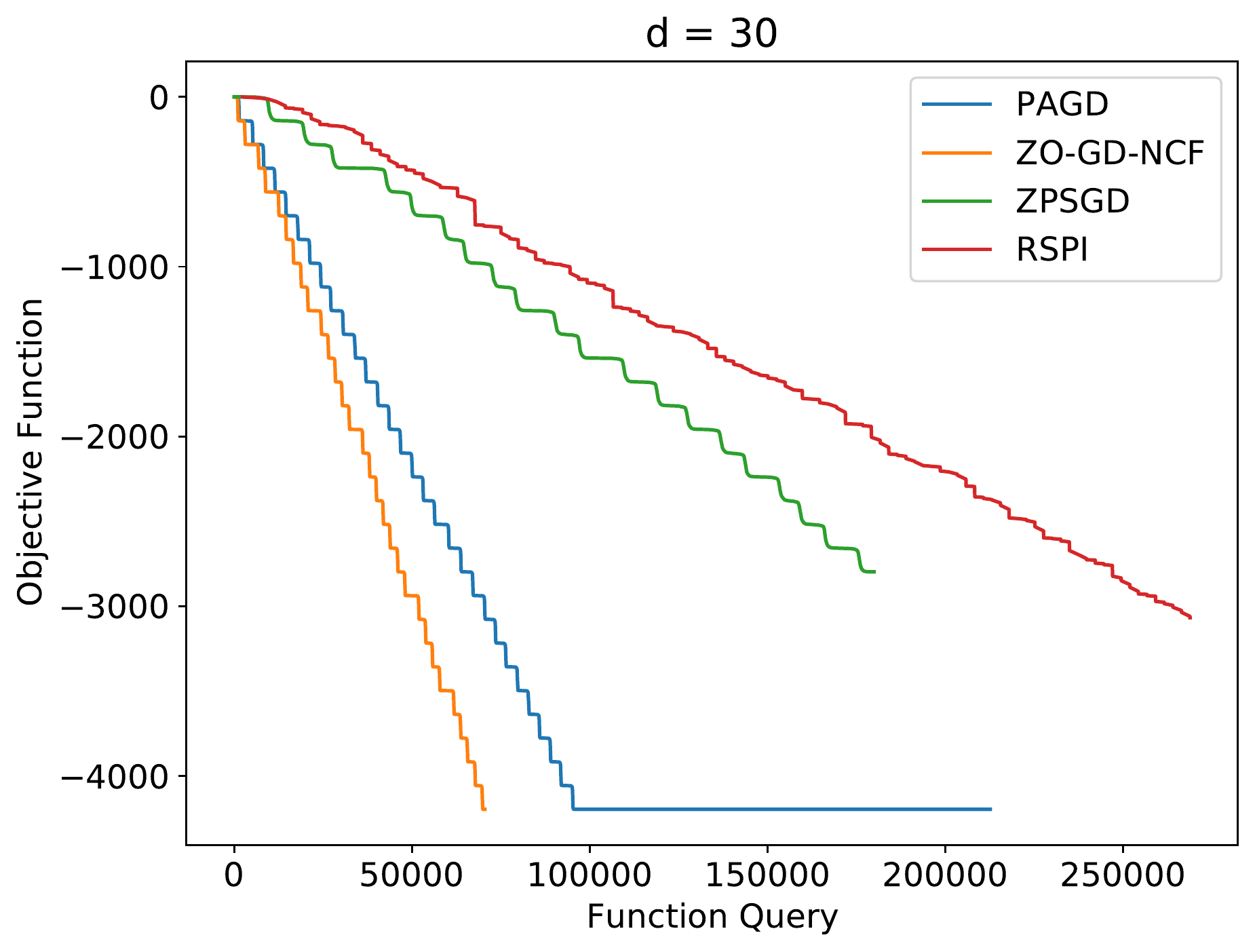}
}
\subfigure[d=50]{
\centering
\includegraphics[width=0.23\textwidth]{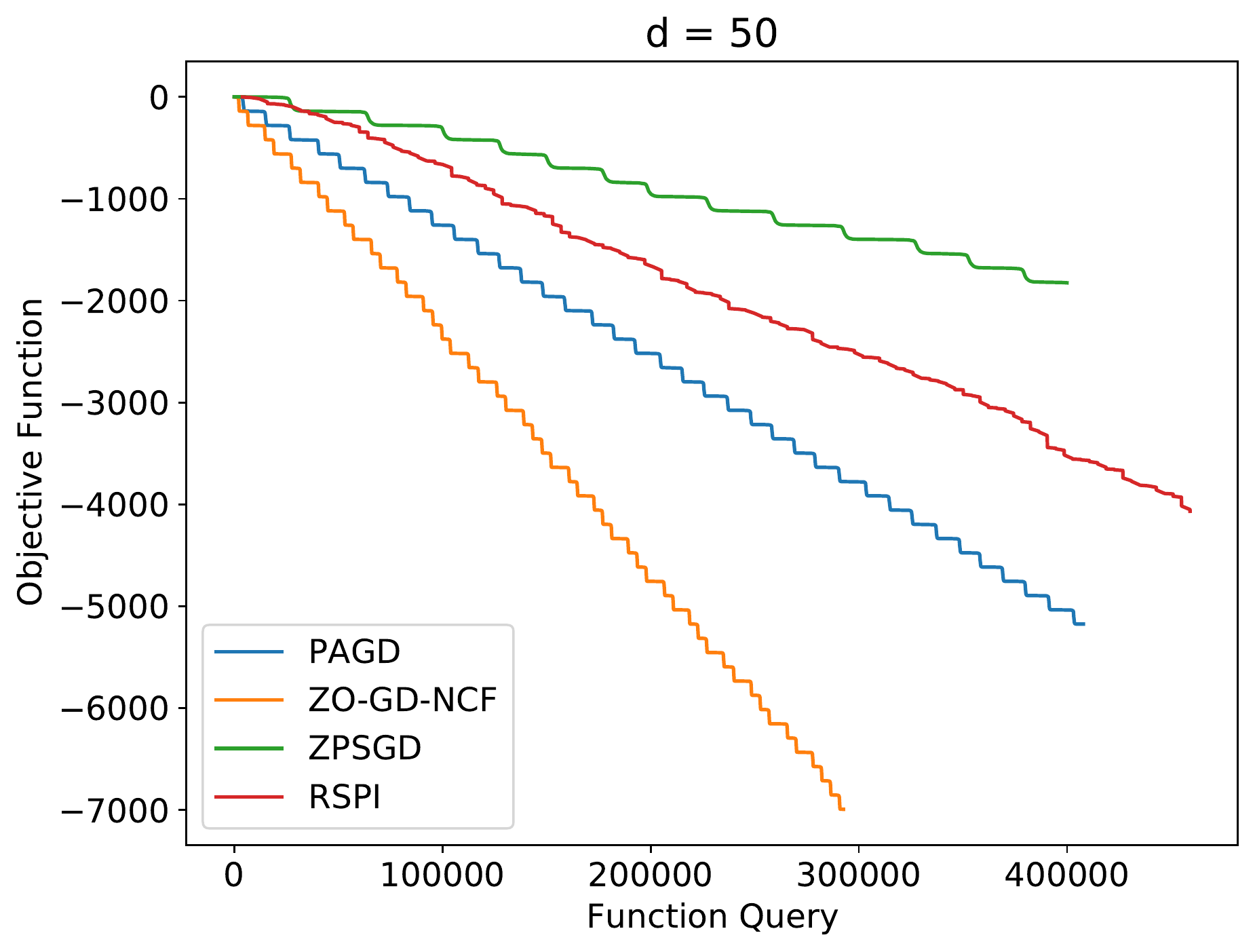}
}
\subfigure[d=100]{
\centering
\includegraphics[width=0.23\textwidth]{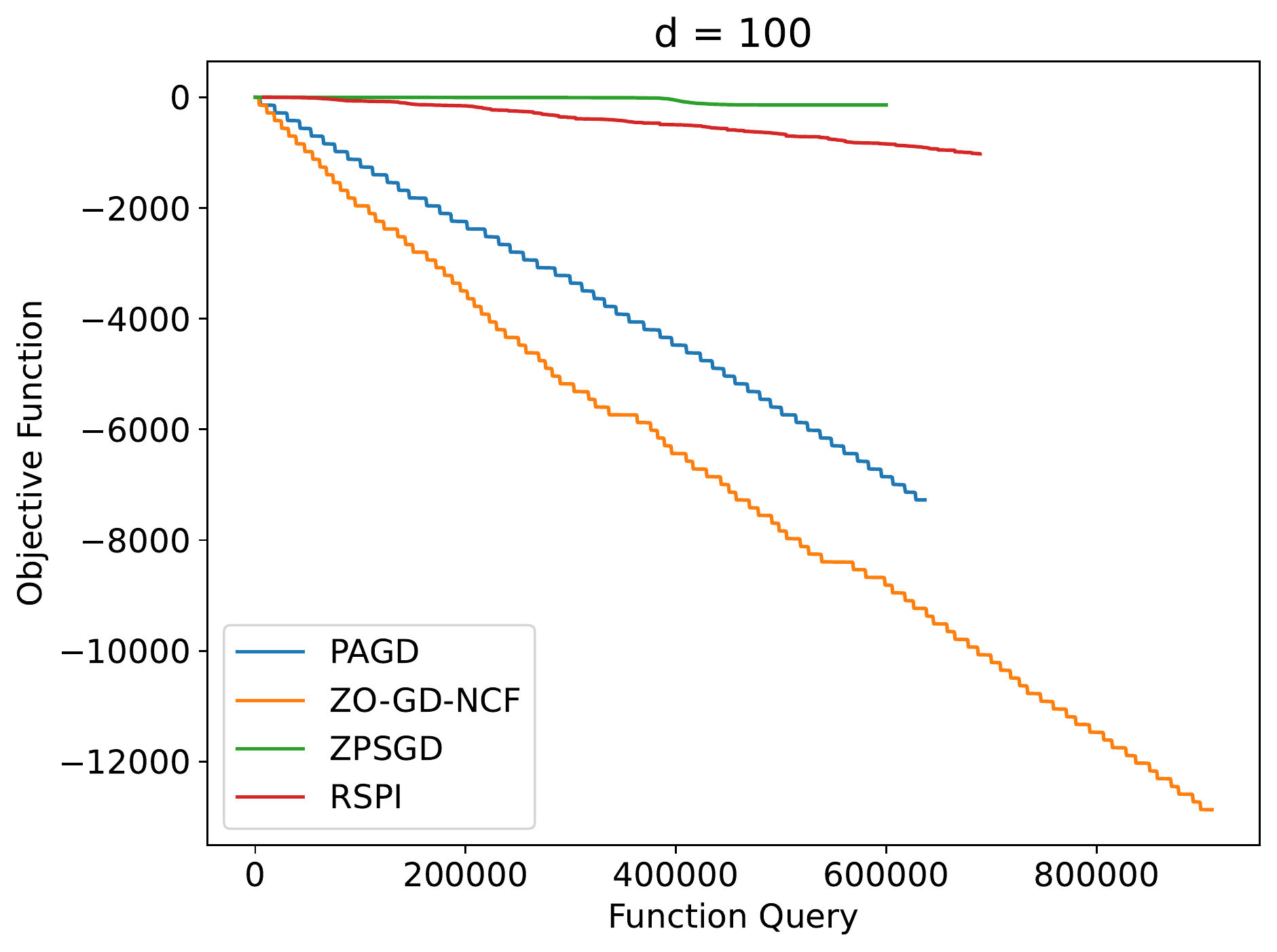}
}

\caption{Performance of ZO-GD-NCF, ZPSGD, PAGD, and RSPI on the octopus function with growing dimensions.}
\label{fig: octopus}

\end{figure}

The results in Figure~\ref{fig: octopus} illustrate that all algorithms are able to escape saddle points. With the increase of the dimension of the octopus function, more function queries are needed for each algorithm to converge to the local minimum. Note that in all experiments, RSPI performs worse than PAGD and ZO-GD-NCF. This is because RSPI is not a gradient based algorithm. Although it can efficiently escape from the saddle point using the negative curvature finding, it converges very slowly when the current point is far from the saddle point due to the random search.

We defer more experimental results to Appendix G.

\section{Conclusion}
In this paper, we analyse two types of ZO negative curvature finding frameworks, which can be used to find the negative curvature directions near a saddle point in the deterministic setting and stochastic setting, respectively. We apply the two frameworks to four ZO algorithms and analyse the complexities for converging to $(\epsilon, \delta)$-approximate SOSPs. Finally, we conduct several numerical experiments to verify the effectiveness of the proposed method in escaping saddle points. 

As a future work, it would be interesting to study the (zeroth-order) unified negative curvature finding frameworks with generic analysis that can be applied to any FOSPs finding algorithms.

\section*{Acknowledgments and Disclosure of Funding}
The authors thank four anonymous reviewers for their constructive comments and suggestions.
Bin Gu was partially supported by the National Natural Science Foundation of China under Grant 62076138.

\bibliography{ref}
\bibliographystyle{plain}

\section*{Checklist}


\begin{enumerate}

\item For all authors...
\begin{enumerate}
  \item Do the main claims made in the abstract and introduction accurately reflect the paper's contributions and scope?
    \answerYes{}
  \item Did you describe the limitations of your work?
    \answerYes{In the conclusion, we mention a way that could improve the current work.}
  \item Did you discuss any potential negative societal impacts of your work?
    \answerNo{}
  \item Have you read the ethics review guidelines and ensured that your paper conforms to them?
    \answerYes{}
\end{enumerate}

\item If you are including theoretical results...
\begin{enumerate}
  \item Did you state the full set of assumptions of all theoretical results?
    \answerYes{}
    \item Did you include complete proofs of all theoretical results?
    \answerYes{Please refer to the Appendix for complete proofs of the theoretical results.}
\end{enumerate}

\item If you ran experiments...
\begin{enumerate}
    \item Did you include the code, data, and instructions needed to reproduce the main experimental results (either in the supplemental material or as a URL)?
    \answerYes{}
    \item Did you specify all the training details (e.g., data splits, hyperparameters, how they were chosen)?
    \answerYes{}
    \item Did you report error bars (e.g., with respect to the random seed after running experiments multiple times)?
    \answerNo{}
    \item Did you include the total amount of compute and the type of resources used (e.g., type of GPUs, internal cluster, or cloud provider)?
    \answerNo{}
\end{enumerate}

\item If you are using existing assets (e.g., code, data, models) or curating/releasing new assets...
\begin{enumerate}
  \item If your work uses existing assets, did you cite the creators?
    \answerYes{}
  \item Did you mention the license of the assets?
    \answerYes{}
  \item Did you include any new assets either in the supplemental material or as a URL?
    \answerNo{}
  \item Did you discuss whether and how consent was obtained from people whose data you're using/curating?
    \answerNA{}
  \item Did you discuss whether the data you are using/curating contains personally identifiable information or offensive content?
    \answerNA{}
\end{enumerate}

\item If you used crowdsourcing or conducted research with human subjects...
\begin{enumerate}
  \item Did you include the full text of instructions given to participants and screenshots, if applicable?
    \answerNA{}
  \item Did you describe any potential participant risks, with links to Institutional Review Board (IRB) approvals, if applicable?
    \answerNA{}
  \item Did you include the estimated hourly wage paid to participants and the total amount spent on participant compensation?
    \answerNA{}
\end{enumerate}

\end{enumerate}

\newpage
\section*{Appendix}
\appendix

\section{Auxiliary Lemmas}

\begin{lemma}[\cite{nesterov2018lectures}, Lemma 1.2.3 \& 1.2.4]
\label{lemma: Lipschitz}
    If $f$ is $\ell$-Lipschitz smooth, then for all $x,y \in \mathbb{R}^d$, 
    \begin{equation*}
        |f(y)-f(x)-\nabla f(x)^T (y-x)| \le \frac{\ell}{2} \|y-x\|^2 .
    \end{equation*}
     If $f$ is $\rho$-Hessian Lipschitz, then for all $x,y \in \mathbb{R}^d$,
     \begin{equation*}
         \left\|\nabla f(y)-\nabla f(x)-\nabla^{2} f(x)(y-x)\right\| \leq \frac{\rho}{2}\|y-x\|^{2}
     \end{equation*}
    and
    \begin{equation*}
        \left|f(y)-f(x)-\nabla f(x)^{T}(y-x)-\frac{1}{2}(y-x)^{T} \nabla^{2} f(x)(y-x)\right| \leq \frac{\rho}{6}\|y-x\|^{3}
    \end{equation*}
\end{lemma}

\begin{lemma}[\cite{allen2018neon2}, Fact 2.2]
\label{lemma: variance reduced}
    If $v_1,\dots,v_n \in \R^d $ satisfy $\sum_{i=1}^n v_i =0$, and $S$ is a non-empty, uniform random subset of $[n]$. Then
    \begin{equation}
        \E \left[ \left\| \frac{1}{|S|} \sum_{i \in S}v_i \right\|^2 \right] \le \frac{\mathbb{I}[|S|<n]}{|S|} \frac{1}{n} \sum_{i\in[n]} \|v_i\|^2
    \end{equation}
\end{lemma}

\begin{lemma}[\cite{allen2018neon2}, Claim C.2]
\label{lemma: negative-curvature-reduction}
    If $v$ is a unit vector and $v^\T \nabla^2 f(y)v \le -\frac{\delta}{2}$, suppose we choose $y' = y \pm \frac{\delta}{\rho} v$ where the sign is random, then $f(y)-\E[f(y')]\ge \frac{\delta^3}{12\rho^2}$.
\end{lemma}

\subsection{Lemmas of ZO Gradient Estimators}
\label{subsection: zo-gradient-estimators}

\begin{lemma}
\label{lemma: coord-square-bound}
    For any given smoothing parameter $\mu$ and any $x\in \R^d$, if $f$ is $\ell$-Lipschitz smooth, then for both \ref{eq: CoordGradEst} and the forward difference version of the coordinate-wise gradient: $\hat{\nabla}_{coord} f(x) = \sum_{i=1}^d \frac{f(x+\mu e_i) - f(x)}{\mu}e_i$, we have 
    \begin{equation}
        \|\hat{\nabla}_{coord} f(x) - \nabla f(x)\|^2 \le \ell^2 d \mu^2.
    \end{equation}
    If we further assume that $f$ has $\rho$-Lipschitz Hessian, we have 
    \begin{equation}
        \|\hat{\nabla}_{coord} f(x) - \nabla f(x)\|^2 \le \frac{1}{36}\rho^2 d \mu^4
    \end{equation}
\end{lemma}
\begin{proof}
    For the $\ell$-Lipschitz gradient case, the proof directly follows from Lemma 3 in \cite{ji2019improved}.
    For the $\rho$-Hessian Lipschitz case, we have 
\begin{align*}
    \left\| \nabla f(x) - \hat{\nabla}_{coord}f(x) \right\| &=\left\| \sum_{i=1}^d \frac{f(x+\mu e_i) - f(x-\mu e_i)}{2\mu}e_i - \nabla f(x) \right\|  \\
    &=\frac{1}{2\mu} \left\| \sum_{i=1}^d (f(x+\mu e_i) - f(x-\mu e_i) - 2 \mu \nabla_i f(x))e_i \right\| 
\end{align*}

For all $i \in [d]$, we have
\begin{align*}
    & f(x+\mu e_i) - f(x-\mu e_i) - 2 \mu \nabla_i f(x) \\
    =& \left[f(x+\mu e_i) - f(x) - \mu \nabla_i f(x) - \frac{\mu^2}{2} \nabla_{ii}^2 f(x)\right] - \left[ f(x-\mu e_i) - f(x) + \mu \nabla_i f(x) - \frac{\mu^2}{2} \nabla_{ii}^2 f(x) \right] \\
    \le & \left|f(x+\mu e_i) - f(x) - \mu \nabla_i f(x) - \frac{\mu^2}{2} \nabla_{ii}^2 f(x) \right| + \left| f(x-\mu e_i) - f(x) + \mu \nabla_i f(x) - \frac{\mu^2}{2} \nabla_{ii}^2 f(x) \right| \\
    \overset{\textrm{\ding{172}}}{\le} & 2 \cdot \frac{\rho}{6}\mu^3 = \frac{\rho}{3} \mu^3
\end{align*}
where \ding{172} is due to Lemma~\ref{lemma: Lipschitz}. 
    \begin{align}
    & \left\| \nabla f(x) - \hat{\nabla}_{coord}f(x) \right\| \notag \\
    = & \frac{1}{2\mu} \left\| \sum_{i=1}^d (f(x+\mu e_i) - f(x-\mu e_i) - 2 \mu \nabla_i f(x))e_i \right\| \notag \\
    = & \frac{1}{2\mu} \sqrt{\sum_{i=1}^d \left( f(x+\mu e_i) - f(x-\mu e_i) - 2 \mu \nabla_i f(x) \right)^2} \notag\\
    \le & \frac{1}{2\mu} \sqrt{d \left( \frac{\rho \mu^3}{3} \right)^2} =  \frac{\sqrt{d} \rho \mu^2}{6}
    \end{align}
\end{proof}

\begin{lemma}
\label{lemma: RandGradEst}
Let $f_{\mu} (x) = \E_{u\sim U_B} f(x+\mu u)$ be a smooth approximation of $f(x)$, where $U_B$ is the uniform distribution over the $d$-dimension unit Euclidean ball $B$. Given the gradient estimator $\hat{\nabla}_{rand} f(x) = d\frac{f(x+\mu u) - f(x - \mu u)}{2\mu}u$, we have 

If we assume that $f$ is $\ell$-Lipshitz smooth, then it has similar properties as its forward version : $\hat{\nabla}_{rand} f(x) = d\frac{f(x+\mu u) - f(x)}{\mu}u$.
\begin{itemize}
    \item[\normalfont(1)] $|f_{\mu}(x) - f(x)| \le \frac{\ell \mu^2}{2}$.
    \item[\normalfont(2)] $ \E \hat{\nabla}_{rand} f_S (x) = \E \frac{1}{|S|}\sum_{i\in S} f_i(x) = \nabla f_{\mu}(x)$, where $S \in [n]$.
    \item[\normalfont(3)] $\|\nabla f_{\mu}(x) - \nabla f(x)\| \le \frac{\ell d \mu}{2}$ for any $x \in \R^d$.
    \item[\normalfont(4)] $\E [\|\hat{\nabla}_{rand} f(x)\|^2] \le 2d \|\nabla f(x)\|^2 + \frac{\ell^2 d^2 \mu^2}{2}$ for any $x\in \R^d$.
    \item[\normalfont(5)] $\E \|\hat{\nabla}_{rand} f(x) - \hat{\nabla}_{rand} f(y)\|^2 \le 3d \ell^2 \|x-y\|^2 + \frac{3 \ell^2 d^2 \mu^2}{2}$ for any $x, y \in \R^d$
\end{itemize}
If we further assume that $f$ is $\rho$-Hessian Lipschitz, we have 
\begin{itemize}
    \item[\normalfont(1)] $\left|f_\mu (x) - f(x) - \frac{\mu^2}{2}  \frac{\sum_{i=1}^d \lambda_i}{d} \right| \le \frac{\rho \mu^3}{6}$, where $\lambda_i, i=1,\dots, d$ are eigenvalues of $\nabla^2 f(x)$. 
    \item[\normalfont(2)] $\|\nabla f_{\mu}(x) - \nabla f(x)\| \le \frac{\rho d \mu^2}{6}$ for any $x\in \R^d$.
    \item[\normalfont(3)] $\E \|\hat{\nabla}_{rand} f(x)\|^2 \le d \|\nabla f(x)\|^2 + \frac{\rho^2 d^2 \mu^4}{36} $ for any $x\in \R^d$.
    \item[\normalfont(4)] $\E \|\hat{\nabla}_{rand} f(x) - \hat{\nabla}_{rand} f(y)\|^2 \le 2d \ell^2 \|x-y\|^2 + \frac{\rho^2 d^2 \mu^4}{18}$ for any $x, y \in \R^d$

\end{itemize}
\begin{remark}
The inequality of item {\normalfont(1)} shows that increasing the level of smoothness of function $f$ beyond $\ell$-smooth cannot improve the approximation ability of $f$ by $f_\mu$. Consider a special case that $f^*(x) = \frac{\ell \|x\|^2}{2}$, then we have $f^*$ is $\ell$-Lipschitz smooth and $0$-Hessian Lipschitz, and according to item  {\normalfont(1)}
$$|f^*_\mu (x) - f^*(x)| = \left |\frac{\mu^2}{2} \frac{\sum_{i=1}^d \lambda_i}{d}\right|  = \frac{\ell \mu^2}{2}.$$ 
This means that $\frac{\ell \mu^2}{2}$ is a  tight bound of $|f_{\mu}(x) - f(x)|$ when $f$ is assumed to be $\ell$-smooth and $\rho$-Hessian Lipschitz.
\end{remark}
\end{lemma}

\begin{proof}
Let $\alpha(d)$ be the volume of the unit ball in $\R^d$, and $\beta(d)$ be the surface area of the
unit sphere in $\R^d$. Denote by $B$ and $S_p$ the unit ball and unit sphere, respectively.

When $f$ is $\ell$-Lipschitz smooth, the proof directly follow from Lemma4.1 in \cite{gao2018information} and \cite{ji2019improved}.

When $f$ is $\rho$-Hessian Lipschitz, We first prove item $(1)$. 
\begin{align*}
    & \frac{1}{\alpha(d)} \int_{B} \left[ f(x+\mu u) - f(x) - \<\nabla f(x), \mu u\> - \frac{\mu^2}{2} \<\nabla^2 f(x)u, u\> \right] du \\
    = & f_\mu (x) - f(x) - \frac{\mu^2}{2} \E_{u \sim U_B}\<\nabla^2 f(x)u, u\> \\
    = & f_\mu (x) - f(x) - \frac{\mu^2}{2}  \frac{\textrm{tr}\left(\nabla^2 f(x)\right)}{d} 
    = f_\mu (x) - f(x) - \frac{\mu^2}{2} \frac{\sum_{i=1}^d \lambda_i}{d} 
\end{align*}
where $\lambda_i, i=1,\dots, d$ are eigenvalues of $\nabla^2 f(x)$. Therefore
\begin{align*}
    & \left|f_\mu (x) - f(x) - \frac{\mu^2}{2}  \frac{\sum_{i=1}^d \lambda_i}{d} \right| \\
    \le & \frac{1}{\alpha(d)} \int_{B} \left| f(x+\mu u) - f(x) - \<\nabla f(x), \mu u\> - \frac{\mu^2}{2} \<\nabla^2 f(x)u, u\> \right| du \\
    \le & \frac{1}{\alpha(d)} \int_B \frac{\rho \mu^3}{6} \|u\|^3  du \overset{\textrm{\ding{172}}}{\le}  \frac{\rho \mu^3}{6} \frac{d}{d+3} \le \frac{\rho \mu^3}{6}
\end{align*}
where \ding{172} is due to $\frac{1}{\alpha(d)} \int_B \|u\|^p du = \frac{d}{d+p}$ as proved in Lemma 7.3(a) in \cite{gao2018information}.

Then we prove item $(2)$. Denote 
$$a_u(x, \mu) = f(x+\mu u) - f(x) - \<\nabla f(x), \mu u\> - \frac{\mu^2}{2} u^\T \nabla^2 f(x) u,$$ 
we have $|a_u(x, \mu)| = |a_u (x, -\mu)| \le \frac{\mu^3}{6} \rho \|u\|^3$ according to Lemma~\ref{lemma: Lipschitz}. Then we have 
\begin{align*}
    & \|\nabla f_{\mu}(x) - \nabla f(x) \| \\
    = & \left\| \frac{1}{\beta(d)} \int_{S_p} \frac{d}{\mu} f(x+\mu u) u du - \nabla f(x)\right\| \\
    = & \left\| \frac{1}{\beta(d)} \int_{S_p} \frac{d}{2\mu} \left(f(x+\mu u) - f(x-\mu u)\right) u du - \nabla f(x)\right\| \\
    = & \left\| \frac{1}{\beta(d)} \int_{S_p} \frac{d}{2\mu} \left(f(x+\mu u) - f(x-\mu u)\right) u du - 2\frac{1}{
    \beta(d)}\int_{S_p} \frac{d}{2 \mu}\<\nabla f(x), \mu u \>u du \right\| \\
    = & \frac{d}{2 \beta(d) \mu} \left\| \int_{S_p} \left( f(x+\mu u) - f(x - \mu u) - 2\<\nabla f(x), \mu u\>\right) u du  \right\| \\
    \le & \frac{d}{2 \beta(d)\mu} \int_{S_p} (|a_u(x, \mu)| + |a_u (-\mu)|) \|u\| du
    \le  \frac{\rho d \mu^2}{6 \beta(d)} \int_{S_p}  \|u\|^4 du= \frac{\rho d \mu^2}{6}
\end{align*}

For item $(3)$. we have 
\begin{align*}
    &\E \|\hat{\nabla}_{rand} f(x)\|^2  \\
    = & \frac{1}{\beta(d)} \int_{S_p} \frac{d^2}{4\mu^2} \left| f(x+\mu u) - f(x-\mu u) \right|^2 \|u\|^2 du \\
    = & \frac{d^2}{\beta(d) 4\mu^2} \int_{S_p} \left| a_u(x, \mu) - a_u(x, -\mu) + 2 \<\nabla f(x), \mu u\> \right|^2  du \\
    \le & \frac{d^2}{\beta(d) 4\mu^2} \left[\int_{S_p} 2\left( |a_u(x, \mu)|^2 + |a_u(x, -\mu)|^2  \right)du + 4 \mu^2 \int_{S_p} \nabla^\T f(x) u u^\T \nabla f(x) du \right] \\
    \overset{\textrm{\ding{172}}}{=}& \frac{d^2}{\beta(d) 4\mu^2} \left[\int_{S_p} 2\left( |a_u(x, \mu)|^2 + |a_u(x, -\mu)|^2  \right)du + 4 \mu^2 \frac{\beta(d)}{d} \|\nabla f(x)\|^2 \right] \\
    \le & \frac{d^2}{\beta(d) 4\mu^2} \left[\int_{S_p} 4 \left( \frac{\mu^3}{6} \rho \|u\|^3 \right)^2du + 4 \mu^2 \frac{\beta(d)}{d} \|\nabla f(x)\|^2 \right] \\
    = & d \|\nabla f(x)\|^2 + \frac{\rho^2 d^2 \mu^4}{36}
\end{align*}
where \ding{172} is due to $\E(u u^\T) = \frac{1}{\beta(d) }\int_{S_p}u u^\T = \frac{1}{d}\I$ as proved in Lemma 6.3 in \cite{gao2018information}. 

For item $(4)$, we have
\begin{align*}
    & \E\|\hat{\nabla}_{rand} f(x) - \hat{\nabla}_{rand} f(y) \|^2 \\
    = & \E \left\| d \frac{f(x+\mu u) - f(x-\mu u)}{2 \mu}u - d \frac{f(y+\mu u) - f(y-\mu u)}{2 \mu}u \right\|^2 \\
    = & \frac{d^2}{4\mu^2} \E \left\| \left[ f(x+\mu u) - f(x-\mu u) \right]u - \left[ f(y+\mu u) - f(y-\mu u) \right]u \right\|^2 \\
    = & \frac{d^2}{4\mu^2} \E \left\| \left(a_u(x,\mu) - a_u(x,-\mu) -(a_u(y,\mu) + a_u(y,-\mu)  \right)u + 2 \<\nabla f(x) - \nabla f(y), \mu u \>u \right\|^2 \\
    \le & \frac{d^2}{4\mu^2} \E \left[ 2 \left( |a_u(x,\mu)|^2 + |a_u(x,-\mu)|^2 + |a_u(y,\mu)|^2 + |a_u(y,-\mu)|^2 \right) \|u\|^2 \right.\\
    & \left. + 8\|\<\nabla f(x) - \nabla f(y), \mu u \>u \|^2 \right] \\
    \le & \frac{d^2}{4\mu^2} \E \left[ 8 \frac{\mu^6}{36} \rho^2 \|u\|^6 \cdot\|u\|^2  + 8\|\<\nabla f(x) - \nabla f(y), \mu u \>u \|^2 \right] \\
    = & 2d^2 \E \|\<\nabla f(x) - \nabla f(y), u\> u\|^2 + \frac{\rho^2 d^2 \mu^4}{18} \E\|u\|^8 
    =  2d^2 \E \<\nabla f(x)-\nabla f(y), u\>^2 + \frac{\rho^2 d^2 \mu^4}{18} \\
    = & 2 d^2 (\nabla f(x) - \nabla f(y))^\T \E (u u^\T) (\nabla f(x) - \nabla f(y)) + \frac{\rho^2 d^2 \mu^4}{18} \\
    \overset{\textrm{\ding{172}}}{=} & 2 d \|\nabla f(x) - \nabla f(y)\|^2 + \frac{\rho^2 d^2 \mu^4}{18} \\
    \le & 2d \ell^2 \|x-y\|^2 + \frac{\rho^2 d^2 \mu^4}{18}
\end{align*}
where \ding{172} is due to $\E(u u^\T) = \frac{1}{\beta(d) }\int_{S_p}u u^\T = \frac{1}{d}\I$ as proved in Lemma 6.3 in \cite{gao2018information}.

\end{proof}

\subsection{ZO Hessian-Vector Product Estimator}

\begin{proof}[\textbf{Proof of Lemma~\ref{lemma: error-bound-of-zo-hessian-vetor-product}}]

If $f$ is $\rho$-Hessian Lipschitz, 
\begin{align}
\label{eq: hessian-vector-error-bound}
    \left\|\nabla f(x+v) - \nabla f(x) - \nabla^2 f(x) v \right\| \le \frac{\rho}{2} \left\| v \right\|^2
\end{align}
this inequality uses Lemma~\ref{lemma: Lipschitz}. 
Then we have 
    \begin{align*}
        & \left\|\nabla^2 f(x_0) v - \mathcal{H}_f(x) v \right\| \\
        = & \left\| \nabla^2 f(x_0) v - \left( \nabla f(x+v) - \nabla f(x) \right) + \left( \nabla f(x+v) - \nabla f(x) \right) - \mathcal{H}_f(x) v \right\| \\
        \overset{\textrm{\ding{172}}}{\le} & \left\| \nabla^2 f(x_0) v - \left( \nabla f(x+v) - \nabla f(x) \right) \right\| + \left\| \left( \nabla f(x+v) - \nabla f(x) \right) - \mathcal{H}_f(x) v \right\| \\
        = & \left\| \nabla^2 f(x_0) v - \left( \nabla f(x+v) - \nabla f(x) \right) \right\| + \left\| \left( \nabla f(x+v) - \nabla f(x) \right) - (\hat{\nabla}_{coord} f(x+v) - \hat{\nabla}_{coord} f(x)) \right\| \\
        \overset{\textrm{\ding{173}}}{\le} & \left\| \nabla^2 f(x_0) v - \left( \nabla f(x+v) - \nabla f(x) \right) \right\| + \left\|\nabla f(x+v) - \hat{\nabla}_{coord} f(x+v) \right\| + \left\|\nabla f(x) - \hat{\nabla}_{coord} f(x) \right\| \\
        \overset{\textrm{\ding{174}}}{\le} & \frac{\rho}{2} \|v\|^2 + \frac{\sqrt{d} \rho \mu^2}{3} 
    \end{align*}
    where \ding{172} and \ding{173} are due to the triangle inequality; \ding{174} is due to Eq. \eqref{eq: hessian-vector-error-bound} and Lemma~\ref{lemma: coord-square-bound}.
\end{proof}

\section{Proof of Proposition~\ref{proposition: verify-gradient-online}}

\begin{proof}
    \textbf{Online stochastic setting.} Let $S :=\{ S_1,\dots,S_m\}$ be $m=\mathcal{O}(\log\frac{1}{p})$ random uniform subsets of $[n]$, each of cardinality $B=\max\{\frac{128\sigma^2}{\epsilon^2},1\}$. Denote by $v_j = \frac{1}{B}\sum_{i\in S_j}\nabla f_i(x)$ and $\hat{v}_j = \frac{1}{B}\sum_{i\in S_j}\hat{\nabla}_{coord}f_i (x)$, according to Lemma~\ref{lemma: variance reduced} we have 
    \begin{align*}
        \E_{S_j} \left[ \|v_j - \nabla f(x)\|^2 \right] \le \frac{1}{B} \cdot \frac{1}{n} \sum_{i\in [n]} \|v_j - \nabla f(x)\|^2 \le \frac{\sigma^2}{B} = \frac{\epsilon^2}{128}
    \end{align*}
    Then, according to the Chebyshev’s inequality: $P(|x- \E(x)|>u)\le \frac{var(x)}{u^2}$, with probability at least $1/2$ over the randomness of $S_j$ we have
    \begin{align*}
        \left|\|\hat{v}_j\| - \|\nabla f(x)\| \right| &\le |\|\hat{v}_j\|-\|v_j\|| + |\|v_j - \nabla f(x)\|| \\
        & \le \|\hat{v}_j - v_j\| +\|v_j - \nabla f(x)\| \le \frac{\rho \sqrt{d}\mu^2}{6} + \frac{\epsilon}{8}  \le \frac{\epsilon}{4}
    \end{align*}
    where the third inequality comes from Lemma~\ref{lemma: coord-square-bound} and the Chebyshev's inequality by setting $u = \frac{\epsilon}{8}$; the last inequality is because we choose the smoothing parameter $\mu$ such that $\mu \le 
    \sqrt{\frac{3\epsilon}{4\rho \sqrt{d}}}$.
    We denote the non-decreasing order 
    \begin{align*}
        & \pi \circ S :=  \{S_{\pi(1)}, S_{\pi(2)}, \dots, S_{\pi(m)}\} \\
        \textrm{s.t.} \quad & \left|\|\hat{v}_{\pi(1)}\| - \|\nabla f(x)\| \right| \ge \dots \ge \left|\|\hat{v}_{\pi(m)}\| - \|\nabla f(x)\| \right|
    \end{align*}
    Then we define the event
    \begin{equation*}
        \mathcal{H}_{j} = \left( \left| \|\hat{v}_{\pi(j)}\| - \|\nabla f(x)\| \right| \ge \frac{\epsilon}{4} \right)
    \end{equation*}
    We have $\textrm{Pr} \left( \mathcal{H}_j  \right) \le \frac{1}{2}$ for all $j \in [m]$. Using the fact that
    \begin{equation*}
        \mathcal{H}_{\pi( \left\lfloor \frac{m}{2} \right\rfloor )} \subseteq \left( \left| \|\hat{v}_{\pi(j)} \|- \|\nabla f(x)\| \right| \ge \frac{\epsilon}{4}, \forall j \le \left\lfloor \frac{m}{2} \right\rfloor  \right) = \bigcap_{j=1}^{\left\lfloor \frac{m}{2} \right\rfloor } \mathcal{H}_{\pi(j)}
    \end{equation*}
    we have 
    \begin{equation*}
        \textrm{Pr} \left(\mathcal{H}_{\pi( \left\lfloor \frac{m}{2} \right\rfloor )}\right) \le \textrm{Pr} \left(\bigcap_{j=1}^{\left\lfloor \frac{m}{2} \right\rfloor } \mathcal{H}_{\pi(j)}\right) = \prod_{j=1}^{\left\lfloor \frac{m}{2} \right\rfloor } \textrm{Pr} \left(\mathcal{H}_{\pi(j)}\right) \le \left(\frac{1}{2}\right)^{\left\lfloor \frac{m}{2} \right\rfloor}
    \end{equation*}
    
    If we choose $m=\mathcal{O}(\log(1/p))$, we have with probability at least $1-p$, it satisfies that at least $\left\lfloor \frac{m}{2} \right\rfloor + 1$ of the vectors $v_j$ satisfy $\left|\|\hat{v}_j\| - \|\nabla f(x)\| \right|\le \frac{\epsilon}{4}$. Then we select $v^* = v_j$ where $j\in[m]$ is index that gives the median value of $\|\hat{v}_j\|$,  then it satisfies $|\|\hat{v}_j\|- \|\nabla f(x)\||\le \frac{\epsilon}{4}$. Finally, we can check if $\|\hat{v}_j\| \le \frac{3\epsilon}{4}$, then $\|\nabla f(x)\| \le \epsilon$, and if not, then $\|\nabla f(x)\| \ge \frac{\epsilon}{2}$.
    
    \textbf{Deterministic.} The case in the offline deterministic setting is much simpler than the online setting. According to Lemma~\ref{lemma: coord-square-bound} we have
    \begin{equation*}
        \|\hat{\nabla}_{coord} f(x) - \nabla f(x)\| \le \frac{\rho\sqrt{d} \mu^2}{6}
    \end{equation*}
    If we choose $\mu \le \sqrt{\frac{3\epsilon}{2\rho \sqrt{d}}}$ we get $\|\hat{\nabla}_{coord} f(x) - \nabla f(x)\| \le \frac{\epsilon}{4}$. Thus, we can check if $\|\hat{\nabla}_{coord} f(x)\|$, then $\|\nabla f(x)\| \le \epsilon$, and if not, then $\|\nabla f(x)\| \ge \frac{\epsilon}{2}$.
\end{proof}

\section{Proof of Zeroth-Order Negative Curvature Search}
\subsection{Proof of Online setting}
\begin{proof}[\textbf{Proof of Lemma~\ref{thm: online-weak}}]
    We first recall the parameter settings in Algorithm~\ref{alg: ZO-NCF-Online-weak}:
    \begin{equation*}
    \label{eq: T1}\tag{T1}
    \eta \gets \frac{\delta}{C_0^2 \ell^2 \log(100d)}, T \gets \frac{C_0^2 \log(100d)}{\eta \delta}, \sigma \gets (100d)^{-3C_0}\frac{\eta^2 \delta^3}{\rho}, r \gets (100d)^{C_0} \sigma
    \end{equation*}
    Denote by $i_t \in [n]$ the random index $i$ chosen at the $t$-th iteration in Algorhtm \ref{alg: ZO-NCF-Online-weak}. 
    Let $\xi_t$ be the error vector such that
    \begin{equation*}
    \label{eq: T2}\tag{T2}
        \xi_t := \nabla^2 f_{i_t}(x_0)(x_t - x_0) - \mathcal{H}_{f_{i_t}}(x_0)(x_t - x_0).
    \end{equation*}
    Then the error vector can be bounded by
    \begin{equation*}
        \|\xi_t \| \le  \rho(\frac{\|x_t-x_0\|^2}{2} + \frac{\sqrt{d}\mu_t^2}{3}) \le \rho \sqrt{d}\|x_t-x_0\|^2.
    \end{equation*}
    where the first inequality is due to By Lemma~\ref{lemma: error-bound-of-zo-hessian-vetor-product}; the second inequality is due to $\mu_t = \|x_t - x_0\|$ and $d\ge1$. According to the definition of $\xi_t$ in \eqref{eq: T2}, we have 
    \begin{equation*}
    \label{eq: T3}\tag{T3}
        x_{t+1} = x_t - \eta \mathcal{H}_{f_{i_t}}(x_0) (x_t - x_0) = x_t -\eta \nabla^2 f_{i_t} (x_0) (x_t - x_0) +\eta \xi_t.
    \end{equation*}
    Then we define the following notations, 
    \begin{equation*}
        z_t = x_t - x_0,\quad \A_t = \B_t + \mathbf{R}_t \quad \text{where} \quad \B_t = \nabla^2 f_{i_t}(x_0), \quad \mathbf{R}_t = -\frac{\xi_t z_t^\T}{\|z_t\|^2},
    \end{equation*}
    From \eqref{eq: T3}, we have 
    \begin{equation*}
        z_{t+1} = z_t -\eta \B_t z_t + \eta \xi_t = (\I - \eta \A_t)z_t.
    \end{equation*}
    As long as Algorithm~\ref{alg: ZO-NCF-Online-weak} does not terminate, we have
    \begin{equation*}
        \|\mathbf{R}_t\| \le \rho \sqrt{d} \|z_t\| \overset{\textrm{\ding{172}}}{\le}  \rho \sqrt{d} r, \quad \|\B_t\|\le \ell, \quad \|\A_t\| \le \|\B_t\| + \|\mathbf{R}_t\| \le \|\B_t\| + \rho \sqrt{d} r  \overset{\textrm{\ding{173}}}{\le} 2\ell
    \end{equation*}
    where \ding{172} holds since we always have $\|z_t\|=\|x_t-x_0\|\le r$ as long as Algorithm~\ref{alg: ZO-NCF-Online-weak} does not terminate; \ding{173} holds since the parameter setting of $r$ as in Algorithm~\ref{alg: ZO-NCF-Online-weak} such that $\rho \sqrt{d} r = \frac{\sqrt{d} \delta^5 }{C_0^4 (100d)^{2C_0} \log^2 (100d) \ell^4} \le \frac{\delta^5}{\ell^4} \le \ell$.
    
    Define
    \begin{align*}
       \Phi_t &  = z_{t+1}z_{t+1}^\T \overset{\textrm{\ding{172}}}{=} (\I - \eta \A_t)\cdots(\I-\eta \A_1)\xi \xi^\T (\I-\eta \A_1)^\T \cdots (\I - \eta \A_t)^\T \\
        &= (\I - \eta \A_t) \Phi_{t-1} (\I - \eta \A_t)^\T \\
        &= (\I - \eta \A_t) \Phi_{t-1} (\I - \eta \A_t) \\
        w_t & = \frac{z_t}{\|z_t\|} = \frac{z_t}{\sqrt{\mathrm{tr} (\Phi_{t-1})}}
    \end{align*}
    where \ding{172} is because $z_1 = x_1 - x_0 =\xi$. As long as Algorithm~\ref{alg: ZO-NCF-Online-weak} does not terminate, we have
    \begin{align*}
        \mathrm{tr}(\Phi_t) &= \|z_{t+1}\|^2 =  \|(\I -\eta \A_t)z_t\|^2 \\
        & = \|z_t\|^2 - 2 \eta z_t^\T\A_t z_t + \eta^2 z_t^\T \A_t^2 z_t \\
        & = \|z_t\|^2 \left( 1-2\eta w_t^\T \A_t w_t + \eta^2 w_t^\T \A_t^2 w_t \right) \\
        & = \mathrm{tr} (\Phi_{t-1}) \left( 1-2\eta w_t^\T \A_t w_t + \eta^2 w_t^\T \A_t^2 w_t \right)  \\
        & \le \mathrm{tr} (\Phi_{t-1}) \left( 1-2\eta w_t^\T \A_t w_t + 4\eta^2 \ell^2 \right) \\
        & \le \mathrm{tr} (\Phi_{t-1}) \left( 1-2\eta w_t^\T \B_t w_t + 2\eta \|\mathbf{R}_t\| + 4\eta^2 \ell^2 \right) \\
        & \le \mathrm{tr} (\Phi_{t-1}) \left( 1-2\eta w_t^\T \B_t w_t + 2 \eta \rho \sqrt{d} r + 4 \eta^2 \ell^2 \right) \\
        & \le \mathrm{tr} (\Phi_{t-1}) \left( 1-2\eta w_t^\T \B_t w_t + 8 \eta^2 \ell^2 \right)
    \end{align*}
    where the last inequality is because $  \frac{\rho \sqrt{d} r}{\eta} = \frac{\sqrt{d} \delta^4}{(100d)^{2C_0} C_0^2 \log(100d) \ell^2} \le \frac{\delta^4}{\ell^2} \le \ell^2 $. On the other hand,
    \begin{align*}
        \mathrm{tr}(\Phi_t) &= \mathrm{tr} (\Phi_{t-1}) \left( 1-2\eta w_t^\T \A_t w_t + \eta^2 w_t^\T \A_t^2 w_t \right)  \\
        & \ge \mathrm{tr} (\Phi_{t-1}) \left( 1-2\eta w_t^\T \A_t w_t \right)  \\
        & \ge \mathrm{tr} (\Phi_{t-1}) \left( 1-2\eta w_t^\T \B_t w_t - 2\eta \|\mathbf{R}_t\|  \right) \\
        & \ge \mathrm{tr} (\Phi_{t-1}) \left( 1-2\eta w_t^\T \B_t w_t - 4 \eta \rho d r  \right) \\
        & \ge \mathrm{tr} (\Phi_{t-1}) \left( 1-2\eta w_t^\T \B_t w_t - 8\eta^2 \ell^2  \right)
    \end{align*}
    Then take the logarithm on both sides of the inequality, we get
    \begin{equation*}
        \log\left( 1-2\eta w_t^\T \B_t w_t - 8\eta^2 \ell^2  \right) \le \log\left(\mathrm{tr}(\Phi_t)\right) - \log \left(\mathrm{tr}(\Phi_{t-1})\right)  \le   \log\left( 1-2\eta w_t^\T \B_t w_t + 8 \eta^2 \ell^2 \right).
    \end{equation*}
    Define 
    \begin{equation*}
        \lambda = -\lambda_{min}(\nabla^2 f(x_0)) = -\lambda_{min} (\E [\B_t]), \quad \A := \nabla^2 f(x_0) = \E[\B_t].
    \end{equation*}
    We know $w_t^\T \B_t w_t \in [-\ell,\ell]$ and $\E [w_t^\T \B_t w_t] = w_t^\T \A w_t \ge -\lambda$. Then we have  
    \begin{align*}
        \log\left(\mathrm{tr}(\Phi_t)\right) - \log \left(\mathrm{tr}(\Phi_{t-1})\right) & \in [\log\left( 1-2\eta w_t^\T \B_t w_t - 8\eta^2 \ell^2  \right), \log\left( 1-2\eta w_t^\T \B_t w_t + 8 \eta^2 \ell^2 \right)]\\
        & \overset{\textrm{\ding{172}}}{\subset} [-2(2\eta \ell+ 8 \eta^2 \ell^2 ), 2 \eta \ell + 8 \eta^2 \ell^2] \\
        & \overset{\textrm{\ding{173}}}{\subset} [-6 \eta \ell, 3\eta \ell] \\
        \E \left[ \log \left( 1-2\eta w_t^\T \B_t w_t + 8 \eta^2 \ell^2 \right) \right] & \overset{\textrm{\ding{174}}}{\le} \log \left(\E \left[\log \left( 1-2\eta w_t^\T \B_t w_t + 8 \eta^2 \ell^2 \right)\right] \right) \le  2 \eta \lambda + 8 \eta^2 \ell^2 
    \end{align*}
    above \ding{172} uses the fact that $|\log(1-x)| \le 2|x|, \forall x \in[-1/2,1/2]$ and $\log(1+x) < x, \forall x >0$; \ding{173} is because $\eta  = \frac{\delta}{C_0^2 \ell^2 \log(100d)} \le \frac{1}{8 \ell}$; \ding{174} is due to the concavity of log and Jensen's inequality. Applying Lemma~\ref{lemma: Azuma-Hoeffding inequality} by setting $\rho = 3\eta \ell, \beta = 2\eta\lambda + 8\eta^2 \ell^2$, we have 
    \begin{equation*}
        \mathrm{Pr}\left[ \log ( \mathrm{tr}(\Phi_t)) -\log(\mathrm{tr}(\Phi_0))  \ge (2 \eta \lambda + 8 \eta^2 \ell^2)t + 6 \eta \ell \sqrt{t \log(\frac{t}{p})} \right]\le p/t.
    \end{equation*}
    then we define
    \begin{equation*}
        T_0 =  \min \left\{ \frac{1}{2\eta \lambda}\cdot \frac{\log(\frac{r^2}{\sigma^2})}{3}, \frac{1}{8\eta^2 \ell^2}\cdot \frac{\log(\frac{r^2}{\sigma^2})}{3}, T_0' \right\}
    \end{equation*}
    where $T_0'$ is the largest positive integer such that $6\eta\ell \sqrt{T_0'\log(\frac{T_0'}{p})} \le \frac{\log(\frac{r^2}{\sigma^2})}{3}$.
    Thus, with the choice of $T_0$ we have
    \begin{equation*}
        (2 \eta \lambda + 8 \eta^2 \ell^2)T_0 + 6 \eta \ell \sqrt{T_0 \log(\frac{T_0}{p})} \le \frac{\log(\frac{r^2}{\sigma^2})}{3} + \frac{\log(\frac{r^2}{\sigma^2})}{3} + \frac{\log(\frac{r^2}{\sigma^2})}{3} = \log(\frac{r^2}{\sigma^2})
    \end{equation*}
    Therefore, for $t \le T_0$, we have 
    \begin{equation*}
        \mathrm{Pr}\left[ \log ( \mathrm{tr}(\Phi_t)) -\log(\mathrm{tr}(\Phi_0)) \ge \log(\frac{r^2}{\sigma^2}) \right] \le \frac{p}{T_0}
    \end{equation*}
    Taking a union bound over $t\in[T_0]$, we have 
    \begin{align*}
        \mathrm{Pr}\left[ \bigcup_{t=1}^{T_0} \log ( \mathrm{tr}(\Phi_t)) -\log(\mathrm{tr}(\Phi_0)) \ge \log(\frac{r^2}{\sigma^2}) \right] 
        \le \sum_{t=1}^{T_0} \mathrm{Pr}\left[ \log ( \mathrm{tr}(\Phi_t)) -\log(\mathrm{tr}(\Phi_0)) \ge \log(\frac{r^2}{\sigma^2}) \right] 
        \le p  
    \end{align*}
    which is equivalent to
    \begin{equation*}
        \mathrm{Pr}\left[ \bigcap_{t=1}^{T_0} \log ( \mathrm{tr}(\Phi_t)) -\log(\mathrm{tr}(\Phi_0)) \ge \log(\frac{r^2}{\sigma^2}) \right]
         \ge 1-p
    \end{equation*}
    By definition, $\mathrm{tr}(\Phi_t) = \|x_t - x_0\|^2$ and $\mathrm{tr}(\Phi_0) = \|\xi\|^2 = \sigma^2$. Then we know with probability at least $1-p$, for every $t \in [T_0], \|x_{t+1} - x_t\| < r$. Thus Algorithm~\ref{alg: ZO-NCF-Online-weak} will not terminate before iteration $T_0$.
    
    \fbox{
    \parbox{\textwidth}{
	Then we prove that when $\lambda > \delta$, Algorithm~\ref{alg: ZO-NCF-Online-weak} outputs a vector $v$, with probability at least $\frac{2}{3}$, $v^\T\A v \le -\frac{3}{4}\delta$.
    }
    }
    We first note when  $\lambda \geq \delta$, using our choice of $\eta$  and $T_{0}$, we have  $T_{0} \le \frac{\log \left(\frac{r^2}{\sigma^2}\right)}{6 \eta \lambda}$. Denote by  $v_{t+1} \stackrel{\text { def }}{=}\left(\I-\eta \B_{t}\right) \cdots\left(\I-\eta \B_{1}\right) \xi$ and $u_{t} \stackrel{\text { def }}{=} z_{t}-v_{t}$  with  $u_{1}=0$, we have
\begin{align*}
    u_{t+1} = z_{t+1}-v_{t+1} &=\prod_{s=1}^{t}\left(\I-\eta \mathbf{A}_{s}\right) \xi-\prod_{s=1}^{t}\left(\I-\eta \B_{s}\right) \xi\\
    &=\left(\I-\eta\B_{t}\right)\left(z_{t}-v_{t}\right)-\eta \mathbf{R}_{t} z_{t}\\
    & = \left(\I-\eta\B_{t}\right)u_t-\eta \mathbf{R}_{t} z_{t}
\end{align*}
then, before Algorithm~\ref{alg: ZO-NCF-Online-weak} stops, we have:
\begin{equation*}
\label{eq: T4}\tag{T4}
    \left\|u_{t+1}-\left(\mathbf{I}-\eta \mathbf{B}_{t}\right) u_{t}\right\|
    = \eta\left\|\mathbf{R}_{t} z_{t}\right\| \overset{\textrm{\ding{172}}}{\le}  2d \eta \rho  r^2
\end{equation*}
above, \ding{172} is due to $\|\mathbf{R}_t z_t\| \le \|\mathbf{R}_t\|\cdot\|z_t\|\le 2\rho d \|z_t\|^2$. Using Young's inequality  $\|a+b\|_{2}^{2} \leq(1+\beta)\|a\|_{2}^{2}+\left(1+\frac{1}{\beta}\right)\|b\|_{2}^{2}$  for every  $\beta>0$, we have:
\begin{align*}
    \left\|u_{t+1}\right\|^{2} 
    & \le (1+\eta^2 \ell^2)\|\left(\I-\eta \B_{t}\right) u_{t}\|^2 + (1+\frac{1}{\eta^2 \ell^2})\|u_{t+1}-\left(\mathbf{I}-\eta \mathbf{B}_{t}\right) u_{t}\|^2 \\
    & \overset{\textrm{\ding{172}}}{\le}  (1+\eta^2 \ell^2)\|\left(\I-\eta \B_{t}\right) u_{t}\|^2 + (1+\frac{1}{\eta^2 \ell^2}) 4 \eta^2 d^2 \rho^2 r^4 \\
    & = (1+\eta^2 \ell^2)\|\left(\I-\eta \B_{t}\right) u_{t}\|^2 + \frac{1+\ell^2 \eta^2}{\ell^2} 4 d^2 \rho^2 r^4 \\
    & \overset{\textrm{\ding{173}}}{\le} \left(1+\eta^{2} \ell^{2}\right)\|(\I-\eta \B_{t}) u_{t}\|^{2}+8 d^2 \frac{\rho^2}{\ell^2} r^{4} \\
    & = \left(1+\eta^{2} \ell^{2}\right) \|u_t\|^2 \left(1-2 \eta \frac{u_t^\T \B_t u_t}{\left\|u_t\right\|^{2}}+ \eta^2\frac{u_t^\T \B_t^2 u_t}{\|u_t\|^2}\right)+8 d^2 \left(\frac{\rho}{\ell}\right)^{2} r^{4} \\
    & \overset{\textrm{\ding{173}}}{\le} \|u_{t}\|^{2}\left(1-2 \eta \frac{u_t^\T \B_t u_t}{\left\|u_t\right\|^{2}}+10 \eta^{2} \ell^{2}\right)+8 d^2 \left(\frac{\rho}{\ell}\right)^{2} r^{4}.
\end{align*}
above, \ding{172} holds due to \eqref{eq: T4}; \ding{173} is because $\eta \ell\le 1$; \ding{174} is because $\|\B_t\| \le \ell$.  Apply Lemma~\ref{lemma: probability-inequality} by setting $\alpha = 2\eta \ell + 10 \eta^2 \ell^2 \le 2.5\eta \ell$ and $\beta = 2\eta \lambda$, we have 
\begin{equation}
\label{eq: pr-2}
    \mathrm{Pr}\left[ \|u_t\| \ge 16 d \frac{\rho}{\ell} r^2 t e^{\eta \lambda t + 8\eta \ell \sqrt{t \log\frac{t}{p}}} \right] \le p
\end{equation}
Apply Lemma~\ref{lemma: growth-and-curvature} we conclude that for each $t\in[T]$, w.p. at least $99/100$:
\begin{itemize}
    \item Norm growth:  $\|v_t\| \ge \frac{1}{C}\left(e^{\left(\eta \lambda-32 \eta^{2} \ell^{2}\right) t} \sigma / \sqrt{d}\right)$.

    \item Negative curvature:  $\frac{-v_{t+1}^{\top} \mathbf{A} v_{t+1}}{\left\|v_{t+1}\right\|^{2}} \leq-(1-C \eta \ell) \lambda+C\left(\frac{\log d}{\eta t}+\sqrt{\frac{\ell^{2}}{t}}+\lambda \eta^{2} \ell^{2} t\right)$.
\end{itemize}
Define
\begin{equation}
\label{eq: T_1}
    T_1 = \frac{\log \frac{2 C \sqrt{d} r}{\sigma}}{\eta \lambda-32 \eta^{2} \ell^{2}}=\frac{C_{0}(\log d / p)+\log (2 C \sqrt{d})}{\eta \lambda-32 \eta^{2} \ell^{2}} \leq \frac{2 C_{0} \cdot \log d / p}{\eta \lambda-32 \eta^{2} \ell^{2}} \leq \frac{4 C_{0} \cdot \log d / p}{\eta \lambda} \leq \frac{4 C_{0} \cdot \log d / p}{\eta \delta}<T
\end{equation}
where the second equality is because $r= (d/p)^{C_0} \sigma$. When $t=T_1$, by ``norm growth'' property, we know that w.p. at least $99/100$:
\begin{equation}
\label{eq: pr-3}
    \|v_{T_1}\| \ge \frac{1}{C}\left( e^{\left(\eta \lambda-32 \eta^{2} \ell^{2}\right) T_1} \sigma / \sqrt{d} \right) \overset{\eqref{eq: T_1}}{=} 2r
\end{equation}
Combing with \eqref{eq: pr-2}, we have w.p. at least $98/100$,
\begin{align}
\label{eq: pr-4}
\frac{\|u_{T_{1}}\|}{\|v_{T_{1}}\|} 
& \le \frac{16 d \rho r^{2} T_{1} e^{\eta \lambda T_{1}+8 \eta \ell \sqrt{T_{1} \log \frac{T_{1}}{p}}}}{\ell \cdot 2 r} \notag\\
& \le 16 C\left(\frac{d^{\frac{3}{2}} \rho r^{2} T_{1} e^{\eta \lambda T_{1}+8 \eta \ell \sqrt{T_{1} \log \frac{T_{1}}{p}}}}{\ell e^{\left(\eta \lambda-32 \eta^{2} \ell^{2}\right) T_{1} }\sigma}\right) \notag\\
&\leq 16 C\left(\frac{d^{\frac{3}{2}} \rho r^{2} T_{1}}{\ell \sigma} e^{8 \eta \ell \sqrt{T_{1} \log \frac{T_{1}}{p}}+32 \eta^{2} \ell^2 T_{1}}\right) \notag\\
& \overset{\textrm{\ding{172}}}{\le} 16 C\left(\frac{d^{\frac{3}{2}} \rho r^{2} T_{1}}{\ell \sigma} e^{16 \sqrt{\log \frac{T_{1}}{p}}}\right) \notag\\
& \overset{\textrm{\ding{173}}}{\le} 16 C\left(\frac{r^2}{\sigma} \cdot \frac{d^{\frac{3}{2}} \rho  T_{1}^{2}}{\ell  p}\right)
\overset{\textrm{\ding{174}}}{\le}  \frac{\delta}{100 \ell} 
\le \frac{1}{100}
\end{align}
where inequality \ding{172} is because $\eta^2 \ell^2 T_1 \le \eta^2 \ell^2 T = 1$; Inequality \ding{173} is because $e^{16 \sqrt{\log (x)}} < x$ for sufficiently large $x$; Inequality \ding{174} is because 
\begin{equation*}
    \frac{r^2}{\sigma} = (\frac{d}{p})^{2C_0} \sigma = (\frac{d}{p})^{-C_0} \frac{\eta^2 \delta^3}{\rho}
\end{equation*}
and thus, for sufficiently large $C_0$ and $p = 1/100$, we have
\begin{equation*}
    (\frac{d}{p})^{-C_0}\le \frac{p}{d^{3/2}}\cdot\frac{1}{C_0 (4C_0 \log(\frac{d}{p}))^2}
\end{equation*}
This implies 
\begin{equation*}
     \frac{r^2}{\sigma} \le \frac{\eta^2 \delta^3 p}{C_0 d^{3/2} \rho (4C_0 \log(\frac{d}{p}))^2} =\frac{\delta p}{C_0 d^{3/2} \rho (\frac{4C_0 \log(\frac{d}{p})}{\eta \delta})^2} \overset{\eqref{eq: T_1}}{\le }\frac{\delta p }{C_0 d^{3/2} \rho T_1^2 }.
\end{equation*}
Thus for sufficiently large $C_0 \ge 1600C$ \ding{174} holds.

Putting together with $\|v_{T_1}\|\ge 2r$, we have w.p. at least $97/100$, $\|z_{T_1}\|= \|u_{T_1} + v_{T_1}\| \ge r$. This means that Algorithm~\ref{alg: ZO-NCF-Online-weak} will terminate within $T_1 \le T$ iterations.

Since w.p. $\ge 99/100$, Algorithm~\ref{alg: ZO-NCF-Online-weak} will not terminate before $T_0\ge \frac{\log(r^2/\sigma^2)}{6\eta \lambda}$. Thus w.p. at least $96/100$, Algorithm~\ref{alg: ZO-NCF-Online-weak} will terminate at $t\in[T_0, T_1]$.

Using the ``negative curvature'' property, we have w.p. at least $\ge 99/100$,
\begin{equation*}
\label{eq: pr-5}
    \frac{v_t^\T \A v_t}{\|v_t\|^2} \le -(1-C\eta \ell)\lambda + C\left( \frac{\log d}{\eta T_0} + \sqrt{\frac{\ell^2}{T_0}} + \lambda \eta^2 \ell^2 T_1 \right).
\end{equation*}
Since $T_0 \ge \frac{\log(\frac{r}{\sigma})^2}{6\eta \lambda} = \frac{C_0 \log(d/p)}{3\eta \lambda} \ge \frac{C_0 \log d}{3 \eta \lambda}$, thus,
\begin{equation*}
    \frac{\log d}{\eta T_0} \le \frac{3 \lambda}{C_0}
\end{equation*}
By the choice of $\eta$ we have $\eta \ell^2 \le \frac{\delta}{C_0^2} \le \frac{\lambda}{C_0^2}$ and 
\begin{align*}
    \frac{\ell^2}{T_0} \le \frac{3 \eta \lambda \ell^2}{C_0} &  \le \frac{\lambda^2}{C_0}, \\
    \lambda \eta^2 \ell^2 T_1 \le \lambda \eta^2 \ell^2 \frac{4C_0 \log(d/p)}{\eta \lambda} & \le \lambda \frac{4C_0 \log d}{C_0^2 \log d} \le \frac{4\lambda}{C_0}
\end{align*}
Then we have w.p. at least $95/100$, Algorithm~\ref{alg: ZO-NCF-Online-weak} terminates at $t\in [T_0,T_1]$ and 
\begin{equation*}
    \frac{v_t^\T \A v_t}{\|v_t\|^2} \le -\frac{15}{16}\lambda \le -\frac{15}{16}\delta.
\end{equation*}
Since $\|u_t + v_t\| = \|z_t\| \ge r$, we have w.p. at least $99/100$,
\begin{equation*}
    \frac{\|u_t\|}{\|u_t\|+\|v_t\|} \le \frac{\|u_t\|}{\|u_t + v_t\|} \le \frac{16 d \rho r^2 T_1 e^{\eta \lambda T_1 + 8\eta \ell \sqrt{T_1 \log\frac{T_1}{p}}}}{\ell r} \overset{\eqref{eq: pr-4}}{\le} \frac{\delta}{50 \ell}.
\end{equation*}
This implies $\frac{\|u_t\|}{\|v_t\|} \le \frac{\delta}{49\ell}$. In sum, we have w.p. at least $94/100$:
\begin{align*}
\frac{z_{t}^{\top} \mathbf{A} z_{t}}{\left\|z_{t}\right\|^{2}} 
&= \frac{\left\|v_{t}\right\|^{2}}{\left\|z_{t}\right\|^{2}} \cdot \frac{z_{t}^{\top} \mathbf{A} z_{t}}{\left\|v_{t}\right\|^{2}} 
= \frac{\left\|v_{t}\right\|^{2}}{\left\|z_{t}\right\|^{2}} \cdot \frac{(v_t+u_t)^\T \A (v_t + u_t)}{\|v_t\|^2} \\
&\leq \frac{\left\|v_{t}\right\|^{2}}{\left\|z_{t}\right\|^{2}} \cdot \frac{v_{t}^{\top} \mathbf{A} v_{t}+4 \ell \left\|u_t\right\|\left\|v_{t}\right\|}{\left\|v_{t}\right\|^{2}} \\
& \leq \frac{\left\|v_{t}\right\|^{2}}{\left\|z_{t}\right\|^{2}} \cdot\left(\frac{v_{t}^{\top} \mathbf{A} v_{t}}{\left\|v_{t}\right\|^{2}}+\frac{4 \ell\left\|u_{t}\right\|}{\left\|v_{t}\right\|}\right) \leq \frac{\left\|v_{t}\right\|^{2}}{\left\|z_{t}\right\|^{2}} \cdot\left(-\frac{15}{16} \delta+\frac{4}{49} \delta\right) \leq-\frac{17 \delta}{20} \frac{\left\|v_{t}\right\|^{2}}{\left\|z_{t}\right\|^{2}} \\
& \leq-\frac{17 \delta}{20}\left(1-\frac{\left\|u_{t}\right\|^{2}}{\left\|z_{t}\right\|^{2}}\right) \leq-\frac{17 \delta}{20} \cdot \frac{49}{50}<-\frac{3}{4} \delta
\end{align*}

\end{proof}

\begin{proof}[\textbf{Proof of Lemma~\ref{thm: online}}]
    By Lemma~\ref{lemma: error-bound-of-zo-hessian-vetor-product} with $\mu = \|v\|$ we have 
    \begin{equation*}
        \left\| v^\T (\mathcal{H}_{f_i}(x) - \nabla^2 f_i(x)) v \right\| \le \rho \sqrt{d} \|v\|^3. 
    \end{equation*}
    Define 
    \begin{equation*}
        z_j = v^\T \mathcal{H}_{f_{i_j}} v,
    \end{equation*}
    then $z_1,\dots,z_m$ are i.i.d. random variables with 
    \begin{align*}
        |z_j| & \le \left\| v^\T (\mathcal{H}_{f_{i_j}}(x) - \nabla^2 f_i(x)) v \right\| + \left\| v^\T \nabla^2 f_{i_{j}}(x) v \right\| \\
        & \le  \rho \sqrt{d} \|v\|^3 + \ell \|v\|^2.
    \end{align*}
    By Chernoff inequality, we have
    \begin{equation*}
        \mathrm{Pr} \left[ |z- \E[z]| \ge 2(\rho \sqrt{d} \|v\|^3 + \ell \|v\|^2) \sqrt{\frac{1}{m} \log \frac{1}{p}} \right] \le p.
    \end{equation*}
    Since $|\E[z] - v^\T \nabla^2 f(x) v| \le \rho \sqrt{d} \|v\|^3$, we have
    \begin{equation*}
        \mathrm{Pr}\left[ \left| \frac{z}{\|v\|^2} - \frac{v^\T \nabla^2 f(x) v}{\|v\|^2} \right| \le 2(\rho \sqrt{d} \|v\| + \ell) \sqrt{\frac{1}{m} \log \frac{1}{p}} + \rho \sqrt{d} \|v\| \right] \ge 1-p.
    \end{equation*}
        
\end{proof}

\begin{lemma}
\label{lemma: Azuma-Hoeffding inequality}
    Consider the random variables $\{x_t\}_{t=0}^T$ with respect to random events $\{\mathcal{F}\}_{t=0}^T$ and $\log (1-a_0), \log (1-a_1), \dots, \log (1-a_T) \in [-2\rho,\rho]$ where each  $x_{t}$  and  $a_{t}$  only depend on  $\mathcal{F}_{1}, \ldots, \mathcal{F}_{t}$.
    \begin{equation*}
        \log x_t = \log x_{t-1} + \log(1-a_t) \quad \text { and } \quad \mathbb{E}\left[\log (1-a_{t}) \mid \mathcal{F}_{1}, \ldots, \mathcal{F}_{t-1}\right] \leq \beta.
    \end{equation*}
    Then we have for every $p\in (0,1)$,
    \begin{equation*}
        \mathrm{Pr} \left[ \log x_t - \log x_0 \ge \beta T + 2 \rho \sqrt{T \log \frac{T}{p}} \right] \le p/T
    \end{equation*}
\end{lemma}
\begin{proof}
    Applying a general form of Azuma-Hoeffding inequality \cite{azuma-hoeffding}, we have 
    \begin{equation*}
        \mathrm{Pr} \left[ \log x_t - \log x_0 \ge \epsilon_1 \right] \le \exp \{-\frac{2\epsilon_1^2}{((\beta+2\rho)^2)T }\}
    \end{equation*}
    Let $\exp \{-\frac{2\epsilon_1^2}{(\mu+2\rho)^2 T }\} = p/T$, we get
    \begin{align*}
        \epsilon_1 &= (\beta + 2\rho)\sqrt{\frac{1}{2}T\log \frac{T}{p}}\\
        &\le \beta T + 2\rho \sqrt{T\log \frac{T}{p}} := \epsilon
    \end{align*}
    So we have 
    \begin{equation*}
        \mathrm{Pr} [ \log x_t - \log x_0 \ge \epsilon ] \le \mathrm{Pr} [ \log x_t - \log x_0 \ge \epsilon_1 ] \le p/T
    \end{equation*}
\end{proof}

\begin{lemma}[\cite{allen2018neon2}]
\label{lemma: probability-inequality}
    Consider random events  $\left\{\mathcal{F}_{t}\right\}_{t \geq 1}$  and random variables  $x_{1}, \ldots, x_{T} \geq 0$  and  $a_{1}, \ldots, a_{T} \in   [-\alpha, \alpha]$  for  $\alpha \in[0,1 / 2]$  where each  $x_{t}$  and  $a_{t}$  only depend on  $\mathcal{F}_{1}, \ldots, \mathcal{F}_{t}$.  Letting  $x_{0} = 0$  and suppose there exist constant  $b \geq 0$  and $\beta>0$  such that for every  $t \geq 1$:
    \begin{equation*}
        x_{t} \leq x_{t-1}\left(1-a_{t}\right)+b \quad \text { and } \quad \mathbb{E}\left[a_{t} \mid \mathcal{F}_{1}, \ldots, \mathcal{F}_{t-1}\right] \geq-\beta.
    \end{equation*}
Then, we have for every  $p \in(0,1)$: $\mathrm{Pr}\left[x_{T} \geq T \cdot b \cdot e^{\beta T+2 \alpha \sqrt{T \log \frac{T}{p}}}\right] \leq p$.
\end{lemma}

\begin{lemma}[\cite{allen2018neon2}]
\label{lemma: growth-and-curvature}
There exists an absolute constant  $C>0$  such that the following holds: Suppose  $\B_1, \B_{2}, \ldots, \B_{t}$  are i.i.d. random matrices with  $\left\|\B_{i}\right\| \leq \ell$  and  $\E\left[\B_{i}\right]=-\A$. Suppose also  $\lambda_{\max }(\A)=\lambda \geq 0$.  Let
\begin{equation*}
    \forall i=0,1, \ldots, t: \quad v_{i+1} \stackrel{\text { def }}{=}\left(\mathbf{I}-\eta \mathbf{B}_{i}\right) \cdots\left(\mathbf{I}-\eta \mathbf{B}_{1}\right) \xi
\end{equation*}
where  $\xi$ is a random Gaussian vector with norm $\sigma$, and  $\eta \in\left(0, \sqrt{\frac{1}{1350000 t \ell^{2}}}\right]$  is the learning rate. Then, with probability at least  $99/100$:

1. Norm growth:  $\left\|v_{t+1}\right\|_{2} \geq \frac{1}{C}\left(e^{\left(\eta \lambda-32 \eta^{2} \ell^{2}\right) t} \sigma / \sqrt{d}\right)$.

2. Negative curvature:  $-\frac{v_{t+1}^{\top} \mathbf{A} v_{t+1}}{\left\|v_{t+1}\right\|_{2}^{2}} \leq-(1-C \eta \ell) \lambda+C\left(\frac{\log d}{\eta t}+\sqrt{\frac{\ell^{2}}{t}}+\lambda \eta^{2} \ell^{2} t\right)$.
\end{lemma}

\subsection{Proof of Deterministic setting}

\begin{proof}[\textbf{Proof of Theorem~\ref{thm: deterministic}}]
For notation simplicity, we denote 
\begin{equation*}
    \A = \nabla^2 f(x_0),\quad \M= -\frac{1}{\ell} \nabla^2 f(x_0) + (1-\frac{3\delta}{4\ell}) \I,\quad \lambda = -\lambda_{min}(\A).
\end{equation*}
Then, we know that all the eigenvalues of $\M$ lie in $[-1,1+ \frac{\lambda-3\delta/4}{\ell}]$. Define 
\begin{equation*}
    \mathcal{M}(y) = \left(-\frac{1}{\ell}\mathcal{H}_f (x_0) + (1- \frac{3\delta}{4\ell})\right) y
\end{equation*}
and use it to approximate $\M y$. Recall that
\begin{equation*}
    y_0 = 0,\quad y_1 = \xi,\quad y_t = 2\mathcal{M}(y_{t-1}) - y_{t-2}.
\end{equation*}
If we set $x_{t+1} = x_0 + y_{t+1} - \mathcal{M}(y_t)$, then it satisfies $x_{t+1} - x_0 \approx \mathcal{T}_t (\M)\xi$ according to the Definition~\ref{def: inexact-backward-recurrence}.

Denote by 
\begin{equation*}
    x_{t+1}^* \triangleq x_0 + \mathcal{T}_t(\M)\xi 
\end{equation*}
the exact solution. We have
\begin{equation*}
    y_t = 2\mathcal{M}(y_{t-1}) -y_{t-2} = 2(x_0-x_t+y_t)-y_{t-2} \Longrightarrow y_t - y_{t-2} = 2(x_t - x_0).
\end{equation*}
Since $\|x_t-x_0\|\le r$ for each $t$ before Algorithm~\ref{alg: ZO-NCF-Deterministic} terminates, we have
\begin{align*}
    \|y_t\| &\le \|y_{t-2}\| + \|y_t - y_{t-2}\| \le \|y_{t-2}\| + 2r \\
    & \le \|y_{t-4}\| + 4r \\
    & \le \cdots \\
    & \le 2tr
\end{align*}
From Lemma~\ref{lemma: Lipschitz}, we have 
\begin{equation*}
    \|\mathcal{M}(y_t) - \M y_t\| = \|-\frac{\mathcal{H}_f(x_0)  }{\ell}y_t + \frac{\nabla^2 f(x_0)}{\ell} y_t\| \le \frac{1}{\ell} \rho \left( \frac{\|y_t\|^2}{2} + \frac{\sqrt{d} \mu_t^2}{3} \right)  \le \frac{\rho \sqrt{d}}{\ell} \|y_t\|^2 \le \frac{2\rho \sqrt{d} r t}{\ell}\|y_t\|.
\end{equation*}
Recall from Definition~\ref{def: chebyshev-polynomial} that
\begin{equation*}
     \mathcal{T}_{t}(x) \in
    \left\{
    \begin{array}{ll}
    [-1,1] & \text { if } x \in[-1,1] \\
    \left[\frac{1}{2}\left(x-\sqrt{x^{2}-1}\right)^{t}, \left(x+\sqrt{x^{2}-1}\right)^{t}\right] & \text { if } x>1
    \end{array}\right.
\end{equation*}
On the other hand, we have for every $x>1, a=x+\sqrt{x^2-1}, b=x-\sqrt{x^2-1}$, it satisfies
\begin{equation*}
    \mathcal{U}_t (x) = \frac{1}{a-b}(a^{t+1} - b^{t+1}) = \sum_{i=0}^{t} a^i b^{t-i} \le (t+1) a^t
\end{equation*}
Then we apply Lemma~\ref{lemma: stable-computation-of-Chebyshev-Polynomials} with eigenvalues of $\M$ in $[a,b] = \left[ 0,1+ \frac{\lambda - 3\delta/4}{\ell} \right]$ and 
\begin{align*}
    & \gamma \triangleq \max \left\{ 1 + \frac{\lambda - 3\delta/4}{\ell} + \sqrt{(1 + \frac{\lambda - 3\delta/4}{\ell})^2 -1}, 1 \right\}, \quad C_c = \gamma^t \|\xi\|=\gamma^t \sigma, \quad C_T =2, \\
    & C_U = t+2, \quad \epsilon = \frac{2\rho \sqrt{d} rt}{\ell}.
\end{align*}
Then according to Lemma~\ref{lemma: stable-computation-of-Chebyshev-Polynomials}, we have 
\begin{equation*}
    \|x_{t+1}^* - x_{t+1}\| \le \frac{40 \sqrt{d} \rho r t^4 \gamma^t \sigma}{\ell}
\end{equation*}
\fbox{
    \parbox{\textwidth}{
	Then we prove that if $\lambda_{min} (\nabla^2 f(x_0)) \le -\delta$, then with probability at least $1-p$, it satisfies that if $v \neq \bot, \|v\|=1$, and $v^\T \nabla^2 f(x_0) v \le -\frac{1}{2}\delta$. In other words, we can assume that $\lambda \ge \delta$.
    }
}

$\lambda \ge \delta$ implies $\gamma >1$, so we can let 
\begin{equation*}
    T_1 \triangleq \frac{\log \frac{4dr}{p \sigma}}{\log \gamma} \le T.
\end{equation*}
By Definition~\ref{def: chebyshev-polynomial} we know that $\|\mathcal{T}_{T_1}(\M)\| \ge \frac{1}{2}\gamma^{T_1} = \frac{2dr}{p\sigma}$. Thus, with probability at least $1-p$, $\|x_{T_1 + 1}^* - x_0\| = \|\mathcal{T}_{T_1}(\M)\xi\| \ge 2r$. Morever, at iteration $T_1$, we have
\begin{equation*}
    \|x_{T_1+1}^* - x_{T_1+1}\| \le \frac{80d \rho r T_1^4 \gamma^{T_1} \sigma}{\ell} \le \frac{80d \rho r T_1^4 \sigma}{\ell} \cdot \frac{4dr }{p \sigma} \le \frac{512 d^2 \rho T_1^4 }{\ell} \cdot \frac{r^2}{p} \overset{\textrm{\ding{172}}}{\le} \frac{\delta}{100\ell}r \le \frac{1}{16}r,
\end{equation*}
where \ding{172} is because $r \le \frac{\delta p}{51200 d^2\rho T_1^4}$. This means $\|x_{T_1+1} - x_0\| \ge r$ so the algorithm must terminate before iteration $T_1 \le T$.

On the other hand, since $\|\mathcal{T}_t(\M)\|\le \gamma^t$, we know that the algorithm will not terminate until $t\ge T_0$:
\begin{equation*}
    T_0 \triangleq \frac{\log \frac{r}{2\sigma}}{\log \gamma}
\end{equation*}
At the time of $t\ge T_0$ of termination, define $\gamma' = 1 + \frac{\lambda - 3\delta/4}{\ell}$, by Definition~\ref{def: chebyshev-polynomial} we have 
\begin{itemize}
    \item $\mathcal{T}_t(\gamma') \ge \frac{1}{2}\gamma^t \ge \frac{1}{2}\gamma^{T_0} \ge \frac{r}{4\sigma} = (\frac{d}{p})^{\Theta(1)}$.
    \item $\forall x \in [-1,1], \mathcal{T}_t(x)\in [-1,1]$.
\end{itemize}
Since all the eigenvalues of $\A$ that are $\ge -3\delta/4$ are mapped to the eigenvalues of $\M$ that are in $[-1,1]$, and the smallest eigenvalue of $\A$ is mapped to the largest eigenvalue $\gamma'$ of $\M$. So we have, with probability at least $1-p$, letting $v_t\triangleq x_{t+1}^*-x_0 = \mathcal{T}_t(\M)\xi$, then it satisfies 
\begin{equation*}
    \frac{v_t^\T \A v_t}{\|v_t\|^2} \le -\frac{3}{4}\delta \le -\frac{5}{8}\delta
\end{equation*}
Therefore, denoting by $z_t \triangleq x_{t+1} - x_0$, we have 
\begin{equation*}
    \frac{\|z_t - v_t\|}{\|z_t-v_t\| + \|v_t\|} \le \frac{\|z_t - v_t\|}{\|z_t\|} \le  \frac{\|x_{t+1}^* - x_{t+1}\|}{r} \le \frac{\delta}{100\ell} \le \frac{1}{16},
\end{equation*}
Finally, we have 
\begin{align*}
    \frac{z_t^\T \A z_t}{\|z_t\|^2} & = \frac{\|v_t\|^2}{\|z_t\|^2} \cdot \frac{z_t^\T \A z_t}{\|v_t\|^2} 
    \leq \frac{\left\|v_{t}\right\|^{2}}{\left\|z_{t}\right\|^{2}} \cdot \frac{v_{t}^{\top} \mathbf{A} v_{t}+4 \ell \left\|z_t - v_t\right\|\left\|v_{t}\right\|}{\left\|v_{t}\right\|^{2}} \\
    & \leq \frac{\left\|v_{t}\right\|^{2}}{\left\|z_{t}\right\|^{2}} \cdot\left(\frac{v_{t}^{\top} \mathbf{A} v_{t}}{\left\|v_{t}\right\|^{2}}+\frac{4 \ell\left\|z_{t} - v_t\right\|}{\left\|v_{t}\right\|}\right) \\
    & \le \frac{\left\|v_{t}\right\|^{2}}{\left\|z_{t}\right\|^{2}} \left( -\frac{5}{8}\delta + \frac{1}{25} \delta \right) \le (1- \frac{\|z_t - v_t\|}{\|z_t\|}) \left( -\frac{5}{8}\delta + \frac{1}{25} \delta \right)\\
    &\le \frac{15}{16} \left( -\frac{5}{8}\delta + \frac{1}{25} \delta \right) \le -\frac{1}{2}\delta.
\end{align*}
\end{proof}

\begin{definition}
\label{def: chebyshev-polynomial}
    Let $\mathcal{T}_{t}(x)$  be the $t$-th Chebyshev polynomial of the first kind and $\mathcal{U}_{t}(x)$  be the  $t$-th Chebyshev polynomial of the second kind, defined as:
\begin{align*}
    &\mathcal{T}_{0}(x):=1, \quad \mathcal{T}_{1}(x):=x, \quad \mathcal{T}_{n+1}(x):=2 x \cdot \mathcal{T}_{n}(x)-\mathcal{T}_{n-1}(x) \\
    &\mathcal{U}_{0}(x) := 1, \quad \mathcal{U}_{1}(x) := 2 x, \quad \mathcal{U}_{n+1}(x) := 2 x \cdot \mathcal{U}_{n}(x)-\mathcal{U}_{n-1}(x)
\end{align*}

then $\mathcal{U}_{n}(x)$ satisfies:  $\frac{d}{d x} \mathcal{T}_{n}(x)=n \mathcal{U}_{n-1}(x)$ and:
\begin{align*}
    \mathcal{T}_{n}(x) &=
    \left\{
    \begin{array}{ll}
    \cos (n \arccos (x)) \in[-1,1] & \text { if } x \in[-1,1] \\
    \frac{1}{2}\left[\left(x-\sqrt{x^{2}-1}\right)^{n}+\left(x+\sqrt{x^{2}-1}\right)^{n}\right] & \text { if } x>1
    \end{array}\right. \\
    \mathcal{U}_{n}(x) &=
    \left\{\begin{array}{ll}
    \in[-t, t] & \text { if } x \in[-1,1] \\
    \frac{1}{2 \sqrt{x^{2}-1}}\left[\left(x+\sqrt{x^{2}-1}\right)^{n+1}-\left(x-\sqrt{x^{2}-1}\right)^{n+1}\right] & \text { if } x>1
    \end{array}\right.
\end{align*}

\end{definition}

\begin{definition}[Inexact backward recurrence, \cite{allen2018neon2}]
\label{def: inexact-backward-recurrence}
    Suppose we want to compute 
\begin{equation*}
    \vec{s}_N \triangleq \sum_{k=0}^{N} \mathcal{T}_k (\M) \vec{c}_k, \quad \text{where}\quad \M \in \R^{d\times d}\quad \text{is symmetric and each} \quad \vec{c}_k \in \R^d.
\end{equation*}
Let $\mathcal{M}$ be an approximate algorithm that satisfies $\|\mathcal{M}(u) - \M u\| \le \epsilon\|u\|$ for every $u\in \R^d$. Then, define inexact backward recurrence to be 
\begin{equation*}
    \hat{b}_{N+1} = 0, \hat{b}_N = \vec{c}_N, \quad \text{and}\quad \forall r\in\{N-1,\dots,0\}: \hat{b}_r \triangleq 2\mathcal{M}(\hat{b}_{r+1}) - \hat{b}_{r+2} + \vec{c}_r \in \R^d,
\end{equation*}
and define the output as $\hat{s}_N \triangleq \hat{b}_0 - \mathcal{M}(\hat{b}_1)$.
If $\epsilon = 0$, then $\hat{s}_N = \vec{s}_N$.
\end{definition}

\begin{lemma}[Stable computation of Chebyshev Polynomials, \cite{allen2018neon2}]
\label{lemma: stable-computation-of-Chebyshev-Polynomials}

For every $N\in\mathbb{N}^*$, suppose the eigenvalues of $\M$ are in $[a,b]$ and suppose there are parameters $C_U \ge 1, C_T \ge 1, \gamma \ge 1, C_c \ge0$ satisfying
\begin{equation*}
    \forall k\in\{0,1,\dots,N\}: \left\{ \gamma^k \|\vec{c}_k\| \le  C_c \quad \textrm{and} \quad \forall x\in[a,b]: |\mathcal{T}_k(x)|\le C_T \gamma^k, |\mathcal{U}_k(x)|\le C_U \gamma^k \right\}.
\end{equation*}
Then, if $\epsilon\le \frac{1}{4N C_U}$, we have
\begin{equation*}
    \|\hat{s}_N - \vec{s}_N\| \le \epsilon \cdot 2(1+2NC_T)NC_U C_c.
\end{equation*}
\end{lemma}

\section{Proof of Results of ZO-GD and ZO-SGD}

\subsection{Proof of Theorem~\ref{thm: ZO-SGD} (\texorpdfstring{\textbf{Option \uppercase\expandafter{\romannumeral1}}}{Lg})}

If we update $x_{t+1} = x_t - \frac{\eta}{|S|}\sum_{i\in S}\hat{\nabla}_{coord} f_i (x_t)$, then according to the smoothness of $f(\cdot)$ we have
\begin{align*}
    & f(x_t) - \E_S [f(x_{t+1})] 
    \ge  \E_S \left[ \<\nabla f(x_t),x_t-x_{t+1}\> - \frac{\ell}{2}\|x_t - x_{t+1}\|^2 \right] \\
    = & \eta\<\nabla f(x_t), \hat{\nabla}_{coord} f(x)\> - \frac{\eta^2 \ell}{2} \E_S \left[ \left\| \frac{1}{|S|} \sum_{i \in S} \hat{\nabla}_{coord} f_i(x_t) \right\|^2 \right] \\
    \overset{\textrm{\ding{172}}}{\ge} & \frac{\eta }{2} (\|\nabla f(x_t)\|^2 - \|\nabla f(x_t) - \hat{\nabla}_{coord} f(x_t)\|^2) - \frac{\eta^2 \ell}{2} \E_S \left[ \left\| \frac{1}{|S|} \sum_{i \in S} \hat{\nabla}_{coord} f_i(x_t) \right\|^2 \right] \\
    \overset{\textrm{\ding{173}}}{\ge} & \frac{\eta }{2} (\|\nabla f(x_t)\|^2 - \frac{\rho^2 d \mu^4}{36}) - \eta^2 \ell \left( \|\nabla f(x)\|^2  + \E_S \left[ \left\| \nabla f(x) - \frac{1}{|S|} \sum_{i \in S} \hat{\nabla}_{coord} f_i(x_t) \right\|^2 \right] \right) \\
    = & \frac{\eta}{2}( (1-2\eta \ell)\|\nabla f(x_t)\|^2 - \frac{\rho^2 d \mu^4}{36}) - \eta^2 \ell \E_S \left[ \left\| \nabla f(x) - \frac{1}{|S|} \sum_{i \in S} \hat{\nabla}_{coord} f_i(x_t) \right\|^2 \right] \\
    \overset{\textrm{\ding{174}}}{\ge} & \frac{\eta}{2}( (1-2\eta \ell)\|\nabla f(x_t)\|^2 - \frac{\rho^2 d \mu^4}{36}) -2\eta^2 \ell \E_S \left[ \left\| \nabla f(x) - \frac{1}{|S|} \sum_{i \in S} \nabla f_i(x_t) \right\|^2 \right. \\
    & \left. + \left\|\frac{1}{|S|} \sum_{i \in S} \left( \nabla f_i(x_t)- \hat{\nabla}_{coord} f_i(x_t) \right) \right\|^2 \right] \\
    \overset{\textrm{\ding{175}}}{\ge} & \frac{\eta}{2}( (1-2\eta \ell)\|\nabla f(x_t)\|^2 - \frac{\rho^2 d \mu^4}{36}) - 2\eta^2 \ell (\frac{\sigma^2}{B} + \frac{\rho^2 d \mu^4}{36}) \\
    = &  \frac{\eta - 2\eta^2 \ell}{2}\|\nabla f(x_t)\|^2 - (\frac{\eta}{2} + 2\eta^2 \ell)\frac{\rho^2 d \mu^4}{36} - \frac{2\eta^2 \ell \sigma^2}{B},
\end{align*}
where inequality \ding{172} holds since $-2\<a,b\> \le \|a-b\|^2 - \|a\|^2$; \ding{173} and \ding{174} holds since $\|a+b\|^2 \le 2 (\|a\|^2 + \|b\|^2)$ and Lemma \ref{lemma: coord-square-bound}; inequality \ding{175} holds since Lemma~\ref{lemma: variance reduced}
and Lemma~\ref{lemma: coord-square-bound}. With the choice of $\eta = \frac{1}{4\ell}, \mu \le \sqrt{\frac{3\epsilon}{4\rho \sqrt{d}}}$ and $B = \max\{\frac{32\sigma^2}{\epsilon^2},1\}$ we have 
\begin{equation*}
    f(x_t) - \E_S [f(x_{t+1})] \ge \frac{1}{16 \ell}(\|\nabla f(x)\|^2 - \frac{\epsilon^2 }{8})
\end{equation*}
Thus as long as Line 5 of Algorithm~\ref{alg: ZO-SGD} is reached, we have $f(x_t) - \E f(x_{t+1}) \ge \Omega(\frac{\epsilon^2}{\ell})$. On the other hand, whenever line 10 is reached, we have $v^\T \nabla f(x_t) v \le -\frac{\delta}{2}$. By Lemma~\ref{lemma: negative-curvature-reduction}, we have $f(x_t) - \E f(x_{t+1}) \ge \Omega(\frac{\delta^3}{\rho^2})$.

Then we choose $K=\mathcal{O} \left(\frac{\rho^2 \Delta_f}{\delta^3} + \frac{\ell \Delta_f}{\epsilon^2} \right)$, then the algorithm must terminate. As for the total query complexity, we note that each iteration of Algorithm~\ref{alg: ZO-SGD} needs $\tilde{\mathcal{O}}(B) = \tilde{\mathcal{O}}(\frac{\sigma^2}{\epsilon^2} + 1)$ stochastic gradient estimators in Line 3 and Line 5, totaling $\tilde{\mathcal{O}}(d(\frac{\sigma^2}{\epsilon^2}+1)K)$ function queries, as well as $\tilde{\mathcal{O}}(\frac{\ell^2}{\delta^2})$ stochastic gradient estimators computations with no more than $\mathcal{O}(\frac{\rho^2 \Delta_f}{\delta^3})$ times. Therefore, the total function query complexity is 
\begin{equation*}
    \tilde{\mathcal{O}} \left( d(\frac{\sigma^2}{\epsilon^2}+1)K + d \frac{\ell^2}{\delta^2} \frac{\rho^2 \Delta_f}{\delta^3} \right) 
    = 
    \tilde{\mathcal{O}} \left( d(\frac{\sigma^2}{\epsilon^2}+1) (\frac{\rho^2 \Delta_f}{\delta^3} + \frac{\ell \Delta_f}{\epsilon^2}) + d \frac{\ell^2 \rho^2 \Delta_f}{\delta^5} \right)
\end{equation*}

\subsection{Proof of Theorem~\ref{thm: ZO-SGD} (\texorpdfstring{\textbf{Option \uppercase\expandafter{\romannumeral2}}}{Lg})}

\begin{lemma}
\label{lemma: stochastic-RandGradEst-Error}
For any $x\in \R^d$, we have
\begin{equation*}
    \E \|\frac{1}{|S|}\sum_{i\in S} \hat{\nabla }_{rand} f_i (x_t) - \hat{\nabla}_{rand} f(x_t)\|^2 \le \frac{2d}{|S|} \|\nabla f(x_t)\|^2 + \frac{2d \sigma^2}{|S|} + \frac{\rho^2 d^2 \mu^4}{36|S|}
\end{equation*}
\end{lemma}

\begin{proof}
    Let $z_i = \hat{\nabla }_{rand} f_i (x_t) - \hat{\nabla}_{rand} f(x_t)$ and $I_i = I(i\in S)$, where $I(\cdot)$ is the indicator function. Then we have $\E_i (I_i^2) \frac{|S|}{n}$ and $\E_i (I_i I_j) = \frac{\binom{|S|}{2}}{\binom{n}{2}} = \frac{|S|(|S|-1)}{n(n-1)}, i \neq j$. Then we have \begin{align*}
        &\E \|\frac{1}{|S|}\sum_{i\in S} \hat{\nabla }_{rand} f_i (x_t) - \hat{\nabla}_{rand} f(x_t)\|^2 =  \E \| \frac{1}{|S|}\sum_{i \in S} z_i \|^2 = \frac{1}{|S|^2} \E \|\sum_{i=1}^n z_i I_i\|^2 \\
        =& \frac{1}{|S|^2} \left( \sum_{i=1}^n \E I_i^2 \|z_i\|^2 + \sum_{i\neq j} \E I_i I_j \<z_i,z_j\> \right) = \frac{1}{|S|^2} \E_u \left( \frac{|S|}{n} \sum_{i=1}^n \|z_i\|^2 + \frac{|S|(|S|-1)}{n(n-1)} \sum_{i \neq j } \<z_i,z_j\> \right) \\
        = & \frac{1}{|S|^2} \E_u \left( \left( \frac{|S|}{n} - \frac{|S|(|S|-1)}{n(n-1)} \right) \sum_{i=1}^n \|z_i\|^2 + \frac{|S|(|S|-1)}{n(n-1)} \|\sum_{i=1}^n z_i\|^2 \right) \\
        \overset{\textrm{\ding{172}}}{=} & \E_u \frac{n - |S|}{n(n-1)|S|} \sum_{i=1}^n \|z_i\|^2 \le \frac{1}{|S|} \E_u\frac{1}{n} \sum_{i=1}^n \|z_i\|^2 = \frac{1}{|S|} \E_u \E_i \|\hat{\nabla}_{rand} f_i (x_t) - \hat{\nabla}_{rand} f(x_t)\|^2 \\
        \overset{\textrm{\ding{173}}}{\le} & \frac{1}{|S|} \E_u \E_i \|\hat{\nabla}_{rand} f_i (x_t)\|^2 \overset{\textrm{Lemma~\ref{lemma: RandGradEst}}}{\le}\frac{1}{|S|} \E_i \left( d \|\nabla f_i (x_t)\|^2 + \frac{\rho^2 d^2 \mu^4}{36}\right) \\
        \le & \frac{d}{|S|} \E_i \left( 2\|\nabla f(x_t)\|^2 + 2 \|\nabla f(x_t) - \nabla f_i (x_t)\|^2 \right) + \frac{\rho^2 d^2 \mu^4}{36 |S|} \\
        \le & \frac{2d}{|S|} \|\nabla f(x_t)\|^2 + \frac{2d \sigma^2}{|S|} + \frac{\rho^2 d^2 \mu^4}{36|S|}
    \end{align*}
\end{proof}

    If we update $x_{t+1} = x_t - \frac{\eta}{|S|} \sum_{i \in S} \hat{\nabla}_{rand} f_i (x_t)$, then according to the smoothness of $f_{\mu} (\cdot)$, we have
    \begin{align*}
        & f_{\mu} (x_t) - \E [f_{\mu} (x_{t+1})] \\
        \ge & \E \left[\<\nabla f_{\mu} (x_t), x_t - x_{t+1}\> - \frac{\ell}{2} \|x_t - x_{t+1}\|^2\right] \\
        = &  \E\<\nabla f_{\mu} (x_t), \frac{\eta}{|S|}\sum_{i\in S}\hat{\nabla}_{rand} f_i(x_t)\> - \frac{\eta^2\ell}{2} \E \left\|\frac{1}{|S|}\sum_{i\in S} \hat{\nabla}_{rand}  f_i(x_t) \right\|^2 \\
        \overset{\textrm{\ding{172}}}{\ge} & \eta \|\nabla f_{\mu} (x_t)\|^2 - \frac{3 \eta^2 \ell}{2} \E \left(\left\| \frac{1}{|S|}\sum_{i\in S} \hat{\nabla}_{rand}  f_i(x_t) - \hat{\nabla}_{rand} f(x_t) \right\|^2 + \left\| \hat{\nabla }_{rand} f(x_t) - \nabla f_\mu (x_t) \right\|^2 \right. \\
        & \left.+ \|\nabla f_\mu (x_t)\|^2 \right) \\
        \overset{\textrm{\ding{173}}}{\ge} & \eta\left(1 - \frac{3\eta\ell}{2}\right) \|\nabla f_\mu(x_t)\|^2 - \frac{3\eta^2 \ell}{2} \left( \frac{2d}{|S|} \|\nabla f(x_t)\|^2 + \frac{2d \sigma^2}{|S|} + \frac{\rho^2 d^2 \mu^4}{36|S|} + d \|\nabla f(x_t)\|^2 + \frac{\rho^2 d^2 \mu^4}{36} \right) \\
        \ge& \eta\left(1 - \frac{3\eta\ell}{2}\right) \|\nabla f_\mu(x_t)\|^2 - \eta^2 \ell \left(  \frac{9 d}{2} \|\nabla f(x_t)\|^2 + \frac{3d \sigma^2}{|S|} + \frac{\rho^2 d^2 \mu^4}{24}\right) \\
        \ge & \eta\left(1 - \frac{3\eta\ell}{2}\right) \left(\frac{1}{2} \|\nabla f (x_t)\|^2 - \|\nabla f(x_t) - \nabla f_\mu (x_t)\|^2  \right)- \eta^2 \ell \left(  \frac{9 d}{2} \|\nabla f(x_t)\|^2 + \frac{3d \sigma^2}{|S|} + \frac{\rho^2 d^2 \mu^4}{24}\right) \\
        \ge & \eta\left(1 - \frac{3\eta\ell}{2}\right) \left(\frac{1}{2} \|\nabla f (x_t)\|^2 - \frac{\rho^2 d^2 \mu^4}{36}  \right)- \eta^2 \ell \left(  \frac{9 d}{2} \|\nabla f(x_t)\|^2 + \frac{3d \sigma^2}{|S|} + \frac{\rho^2 d^2 \mu^4}{24}\right) \\
        \ge & \eta(\frac{1}{2} - 8d\eta \ell) \|\nabla f(x_t)\|^2 - 3d \eta^2 \ell \frac{\sigma^2}{|S|} - \eta \frac{\rho^2 d^2 \mu^4}{36}
    \end{align*}
where \ding{172} is due to $\|a+b+c\|^2 \le 3 (\|a\|^2 + \|b\|^2 + \|c\|^2)$; \ding{173} is due to Lemma~\ref{lemma: RandGradEst} and Lemma~\ref{lemma: stochastic-RandGradEst-Error}.

Since $|f(x) - f_{\mu}(x)|\le  \frac{\ell \mu^2}{2}$, we have
\begin{align*}
    &f(x_t) - \E[f(x_{t+1})] \\
    \ge & f_{\mu}(x_t) - \E[f_{\mu}(x_{t+1})] - \ell \mu^2\\
    \ge & \eta(\frac{1}{2} - 8d\eta \ell) \|\nabla f(x_t)\|^2 - 3d \eta^2 \ell \frac{\sigma^2}{|S|} - \eta \frac{\rho^2 d^2 \mu^4}{36} - \ell \mu^2
\end{align*}
With the choice of $\eta = \frac{1}{32 d\ell}, \mu = \min \left\{\sqrt{\frac{3\epsilon}{4\rho d}}, \frac{\epsilon}{32 \sqrt{d} \ell} \right\}, B = \max\{ \frac{8 \sigma^2}{\epsilon^2}, 1\}$, we have 
\begin{align*}
    &f(x_t) - \E[f(x_{t+1})] \\
    \ge & \frac{1}{128d\ell} \|\nabla f(x_t)\|^2 - \frac{1}{256 d \ell} \frac{\sigma^2}{B} - \frac{1}{128d \ell} \frac{\rho^2 d^2 \mu^4}{9} - \ell \mu^2 \\
    = & \frac{1}{128d\ell} \left( \|\nabla f(x_t)\|^2 -\frac{ \sigma^2}{2B} - \frac{\rho^2 d^2 \mu^4}{9} - 128 d \ell^2 \mu^2 \right) \\
    \ge & \frac{1}{128d\ell} \left( \|\nabla f(x_t)\|^2 - \frac{\epsilon^2}{8} \right)
\end{align*}

Thus as long as Line 6 of Algorithm~\ref{alg: ZO-SGD} is reached, we have $f(x_t) - \E f(x_{t+1}) \ge \Omega(\frac{\epsilon^2}{d \ell})$. On the other hand, whenever line 10 is reached, we have $v^\T \nabla f(x_t) v \le -\frac{\delta}{2}$. By Lemma~\ref{lemma: negative-curvature-reduction}, we have $f(x_t) - \E f(x_{t+1}) \ge \Omega(\frac{\delta^3}{\rho^2})$.

Then we choose $K=\mathcal{O} \left(\frac{\rho^2 \Delta_f}{\delta^3} + \frac{d\ell \Delta_f}{\epsilon^2} \right)$, then the algorithm must terminate. As for the total query complexity, we note that each iteration of Algorithm~\ref{alg: ZO-SGD} needs $\tilde{\mathcal{O}}(B) = \tilde{\mathcal{O}}(\frac{\sigma^2}{\epsilon^2} + 1)$ stochastic gradient estimators in Line 6 and $\tilde{\mathcal{O}}(\frac{\sigma^2}{\epsilon^2} + 1)$ deterministic coordinate-wise gradient estimators Line 3, totaling $\tilde{\mathcal{O}}(d(\frac{\sigma^2}{\epsilon^2}+1)K)$ function queries, as well as $\tilde{\mathcal{O}}(\frac{\ell^2}{\delta^2})$ stochastic gradient estimators computations with no more than $\mathcal{O}(\frac{\rho^2 \Delta_f}{\delta^3})$ times. Therefore, the total function query complexity is 
\begin{equation*}
    \tilde{\mathcal{O}} \left( (d\frac{\sigma^2}{\epsilon^2}+d)K + d \frac{\ell^2}{\delta^2} \frac{\rho^2 \Delta_f}{\delta^3} \right) 
    = 
    \tilde{\mathcal{O}} \left( d(\frac{\sigma^2}{\epsilon^2}+1) (\frac{\rho^2 \Delta_f}{\delta^3} + \frac{d \ell \Delta_f}{\epsilon^2}) + d \frac{\ell^2 \rho^2 \Delta_f}{\delta^5} \right)
\end{equation*}

\section{Applying Zeroth-Order Negative Curvature Finding to ZO-SCSG}

In this section, we first propose a zeroth-order variant of the SCSG \cite{lei2017non} method in Algorithm~\ref{alg: ZO-SCSG}. At the beginning of the $j$-th epoch, we estimate the gradient $\nabla f_{\mathcal{I}_j} (\tilde{x}_{j-1})$ by \ref{eq: CoordGradEst} over a batch sampling set $\mathcal{I}_j$ with size $B$. In the inner loop iterations, the stochastic gradient estimator $v_{k-1}^j$ is either constructed by \ref{eq: CoordGradEst} or by \ref{eq: RandGradEst} 
over a mini-batch sampling set $\mathcal{I}_{k-1}^j$ with size $b$. After running ZO-SCSG for one epoch, we have the following lemma:

\begin{algorithm}[htbp]
	\caption{ZO-SCSG}
	\label{alg: ZO-SCSG}
	\renewcommand{\algorithmicrequire}{\textbf{Input:}} 
	\renewcommand{\algorithmicensure}{\textbf{Output:}}
	\begin{algorithmic}[1]
		\Require Number of stages $T$, initial point $\tilde{x}_0$, batch size $B$, mini-bath size $b$, learning rate $\eta>0$.
		\For{$j = 1,\dots, T$}
		\State Uniformly randomly sample a batch $\mathcal{I}_j \subset [n]$ with $|\mathcal{I}_j| = B$
		\State $v_j = \hat{\nabla}_{coord} f_{\mathcal{I}_j} (\tilde{x}_{j-1})$
		\State $x_0^j = \tilde{x}_{j-1}$
		\State \textbf{Option \uppercase\expandafter{\romannumeral1}: } $N_j \sim \text{Geom}( \frac{B}{B+b})$
		\textbf{Option \uppercase\expandafter{\romannumeral2 }: } $N_j \sim \text{Geom}(\frac{B}{B+b/d})$
		\For{$k = 1, \dots, N_j$}
		\State Randomly pick $\mathcal{I}_{k-1}^j \subset [n]$ with size $b$
		\State \textbf{Option \uppercase\expandafter{\romannumeral1}: } $v_{k-1}^j = \hat{\nabla}_{coord} f_{\mathcal{I}_{k-1}^j} (x_{k-1}^j) - \hat{\nabla}_{coord} f_{\mathcal{I}_{k-1}^j} (x_0^j) + v_j$
		\State \textbf{Option \uppercase\expandafter{\romannumeral2}: } $v_{k-1}^j = \hat{\nabla}_{rand} f_{\mathcal{I}_{k-1}^j} (x_{k-1}^j) - \hat{\nabla}_{rand} f_{\mathcal{I}_{k-1}^j} (x_0^j) + v_j$
		\State $x_k^j = x_{k-1}^j - \eta v_{k-1}^j$
		\EndFor
		\State $\tilde{x}_j = x_{N_j}^j$
	    \EndFor
	\end{algorithmic}
\end{algorithm}

\begin{algorithm}[htbp]
	\caption{ZO-SCSG-NCF}
	\label{alg: ZO-SCSG-NCF}
	\renewcommand{\algorithmicrequire}{\textbf{Input:}} 
	\renewcommand{\algorithmicensure}{\textbf{Output:}}
	\begin{algorithmic}[1]
		\Require Function $f$, starting point $x_0$, batch size $B$, mini-batch size $b$, $K$, $\epsilon>0$ and $\delta>0$.
		\If{$b > B$}
		\Return \hyperref[alg: ZO-SGD]{ZO-SGD}($f, x_0, \frac{2}{3}, \epsilon, \delta$)
		\EndIf
		\For{$t = 0,\dots,K-1$}
		\State uniformly randomly choose a set $\mathcal{B}$ with batch size $\mathcal{O}(\frac{\sigma^2}{\epsilon^2} \log K)$
		\If{$\|\hat{\nabla}_{coord} f_{\mathcal{B}}(x_t)\|\ge \frac{3\epsilon}{4}$ }
        \State $x_{t+1} \gets$ apply \hyperref[alg: ZO-SCSG]{ZO-SCSG} on $x_t$ for one epoch with batch size $B$ and mini-batch size $b$
	    \Else 
	    \State $v \gets$ \hyperref[alg: ZO-NCF-Online]{ZO-NCF-Online} ($f, x_t, \delta, \frac{1}{20K}$)
	    \If{$v=\bot$}
	    \Return $x_t$
	    \Else
	    \quad $x_{t+1} = x_t \pm \frac{\delta}{\rho} v$
	    \EndIf
	    \EndIf
	    \EndFor
	\end{algorithmic}
\end{algorithm}

\begin{lemma}[One epoch analysis]
Under Assumption~\ref{assum: basic}, \textbf{Option \uppercase\expandafter{\romannumeral1}:}
\label{lemma: ZO-SCSG}
    Let $\eta \ell = \gamma \left(\frac{b}{B}\right)^{\frac{2}{3}}$. Suppose $ \gamma \le \frac{1}{4}$, $B \ge 8b$ and $b \ge 1$, then after running \hyperref[alg: ZO-SCSG]{ZO-SCSG} for one epoch, we have 
    \begin{align}
        \left( \frac{B}{b} \right)^{\frac{1}{3}} \E\|\nabla f(\tilde{x}_j)\|^2 \le  \frac{4\ell}{\gamma} \left(f(\tilde{x}_{j-1}) - f(\tilde{x}_j)\right) + \frac{30\sigma^2 }{b^{\frac{1}{3}} B^{\frac{2}{3}}} + c \left(\frac{B}{b}\right)^{\frac{1}{3 }} \ell^2 d \mu^2,
    \end{align}
    where $c$ is a sufficiently large constant. 
    \textbf{Option \uppercase\expandafter{\romannumeral2}:}
    Let $\eta \ell = \gamma \left(\frac{b/d}{B}\right)^{\frac{2}{3}}$. Suppose $ \gamma \le \frac{1}{8}$, $B \ge 8b/d$ and $b \ge d$, then after running \hyperref[alg: ZO-SCSG]{ZO-SCSG} for one epoch, we have 
    \begin{align}
    \left( \frac{B}{b/d} \right)^{\frac{1}{3}} \E\|\nabla f(\tilde{x}_j)\|^2 \le \frac{8\ell}{\gamma} \E \left( f(\tilde{x}_{j-1}) - f(\tilde{x}_j) \right) + \frac{72 \sigma^2}{(b/d)^{\frac{1}{3}} B^{\frac{2}{3}}} + c \left( \frac{B}{b/d} \right)^{\frac{1}{3}} \ell^2 d^2 \mu^2,
    \end{align}
    where $c$ is a sufficiently large constant.
\end{lemma}

\begin{remark}
The epoch size $N_j$ obeys the Geometric distribution, \emph{i.e.}, $N_j \sim \text{Geom}(\frac{B}{B+b})$ in  \textbf{Option \uppercase\expandafter{\romannumeral1}} and $N_j \sim \text{Geom}(\frac{B}{B+b/d})$ in  \textbf{Option \uppercase\expandafter{\romannumeral2}}. Since in expectation we have $\E(N_j)_{N_j\sim \text{Geom}(\theta)} = \frac{\theta}{1-\theta}$ \cite{lei2017non}, then for both \textbf{Option \uppercase\expandafter{\romannumeral1}} and \textbf{Option \uppercase\expandafter{\romannumeral2}}, the function query complexity in each epoch is $\mathcal{O}(d\cdot B)$.
\end{remark}

\begin{theorem}
Under Assumption~\ref{assum: basic}, if we set $\mu_1 = \sqrt{\frac{3\epsilon}{4\rho\sqrt{d}}}$ and other parameters as follows, 
\begin{align*}
    \textbf{Option \uppercase\expandafter{\romannumeral1}:}  B & =\max\{\frac{480\sigma^2}{\epsilon^2},1\}, b=\max\{1,\Theta(\frac{(\epsilon^2 + \sigma^2)\epsilon^4 \rho^6}{\delta^9 \ell^3})\}, K=\Theta(\frac{\ell b^{\frac{1}{3}}\Delta_f}{\epsilon^2 B^{\frac{1}{3}}}), \mu_2 = \frac{\epsilon}{4\sqrt{c d}\ell}; \\
    \textbf{Option \uppercase\expandafter{\romannumeral2}:} B & = \max\{1, \frac{1152\sigma^2}{\epsilon^2}\}, b=\max\{1,d\Theta(\frac{(\epsilon^2 + \sigma^2)\epsilon^4 \rho^6}{\delta^9 \ell^3})\}, K=\Theta(\frac{\ell (b/d)^{\frac{1}{3}}\Delta_f}{\epsilon^2 B^{\frac{1}{3}}}), \mu_2 = \frac{\epsilon}{4\sqrt{c} d \ell},
\end{align*}
where $\mu_1$ and $\mu_2$ are only use in Line 4 and Line 5 of Algorithm~\ref{alg: ZO-SCSG-NCF}, respectively. With probability at least $\frac{2}{3}$, for both \textbf{Option \uppercase\expandafter{\romannumeral1}} and \textbf{Option \uppercase\expandafter{\romannumeral2}}, Algorithm~\ref{alg: ZO-SCSG-NCF} outputs an $(\epsilon, \delta)$-approximate local minimum in function query complexity 
\begin{equation*}
    \tilde{\mathcal{O}}( d ( \frac{\ell  \Delta_f }{\epsilon^\frac{4}{3} \sigma^{\frac{2}{3}} } + \frac{\rho^2 \Delta_f}{\delta^3} ) ( \frac{\sigma^2 }{\epsilon^2} + \frac{\ell^2}{\delta^2} ) + d \frac{\ell \Delta_f}{\epsilon^2} \frac{\ell^2}{\delta^2}  ).
\end{equation*}

\end{theorem}

\begin{remark}
The problem described in Remark~\ref{remark: ZO-SGD} doesn't exist in ZO-SCSG-NCF as we only evaluate the magnitude of the gradient after each epoch (\emph{i.e.}, Line 4 in Algorithm~\ref{alg: ZO-SCSG-NCF}), and the function query complexity is almost the same in the inner loop for both \textbf{Option \uppercase\expandafter{\romannumeral1}} and \textbf{Option \uppercase\expandafter{\romannumeral2}}. We can boost the confidence in Theorem~\ref{thm: ZO-SCSG-NCF} from $2/3$ to $1-p$ by running $\log 1/p$ copies of \hyperref[alg: ZO-SCSG-NCF]{ZO-SCSG-NCF}.
\end{remark}


\subsection{One Epoch Analysis of  \texorpdfstring{\hyperref[alg: ZO-SCSG]{ZO-SCSG} (\textbf{Option \uppercase\expandafter{\romannumeral1}})}{Lg}}

\begin{lemma}[\cite{lei2017non}]
    Let $N\sim \text{Geom}(\gamma)$ for $\gamma >0$. Then for any sequence $D_0, D_1, \dots$ with $\E|D_N|<\infty$
    \begin{equation*}
        \E(D_N - D_{N+1}) = \left( \frac{1}{\gamma} - 1\right)(D_0 - \E D_N).
    \end{equation*}
\end{lemma}
\begin{proof}
    Then proof directly follows from Lemma A.2 in \cite{lei2017non}.
\end{proof}

\begin{lemma}
\label{lemma: ZO-SCSG-Coord-1}
Suppose $\eta\ell <1$, then under Assumption~\ref{assum: basic},
\begin{align*}
    & \eta(1-\ell\eta) B \E\|\nabla f(\tilde{x}_j)\|^2 + \eta B \E\<\hat{e}_{N_j}^j, \nabla f(\tilde{x}_j)\> \\
    \le & b \left(f(\tilde{x}_{j-1}) - \E f(\tilde{x}_j)\right) + \frac{\ell^3 \eta^2 B}{b} \E\|\tilde{x}_j - \tilde{x}_{j-1}\|^2 + \frac{4\ell^3 \eta^2 d \mu^2 B}{b} + \ell \eta^2 B \E \|\hat{e}_{N_j}^j\|^2
\end{align*}

\end{lemma}

\begin{proof}
By Lemma~\ref{lemma: Lipschitz}, we have 
\begin{equation*}
    f(x_{k+1}^j) \le f(x_k^j) - \<x_{k+1}^j - x_k^j, \nabla f(x_k^j)\> + \frac{\ell}{2} \|x_{k+1}^j - x_k^j\|^2 \le f(x_k^j) - \eta \<v_k^j, \nabla f(x_k^j)\> + \frac{\ell \eta^2 }{2} \|v_k^j\|^2
\end{equation*}
Define the following notation, 
\begin{align*}
    \hat{e}_j = & v_j - \hat{\nabla}_{coord} f(x_0^j) \\
    \hat{e}_k^j = & \hat{\nabla}_{coord} f(x_k^j) - \nabla f(x_k^j) +  \hat{e}_j
\end{align*}
Then we have 
\begin{equation*}
    \E_{\mathcal{I}_k^j} v_k^j = \E_{\mathcal{I}_k^j} \left(\hat{\nabla}_{coord} f_{\mathcal{I}_k^j}(x_k^j) - \hat{\nabla}_{coord} f_{\mathcal{I}_k^j}(x_0^j) + v_j \right) = \nabla f(x_k^j) + \hat{e}_k^j
\end{equation*}
Taking expectation over the above inequality we have 
\begin{align*}
    & \E_{\mathcal{I}_k^j} f(x_{k+1}^j) \\
    \le &  f(x_k^j) - \eta \<\E_{\mathcal{I}_k^j} v_k^j, \nabla f(x_k^j)\> + \frac{\ell \eta^2 }{2} \E_{\mathcal{I}_k^j}\|v_k^j\|^2 \\
    = & f(x_k^j) - \eta \<\nabla f(x_k^j) + \hat{e}_k^j, \nabla f(x_k^j)\> + \frac{\ell \eta^2}{2} \E_{\mathcal{I}_k^j} \|v_k^j\|^2 \\
    = & f(x_k^j) - \eta \|\nabla f(x_k^j)\|^2  - \eta\<\hat{e}_k^j, \nabla f(x_k^j)\> + \frac{\ell \eta^2}{2} \E_{\mathcal{I}_k^j} \|v_k^j\|^2 
\end{align*}
Then we bound the term $\E_{\mathcal{I}_k^j} \|v_k^j\|^2$ by using the fact that $\E\|a\|^2 = \E\|a-\E a\|^2 + \|\E a\|^2$.
\begin{align*}
    & \E_{\mathcal{I}_k^j} \|v_k^j\|^2 = \E_{\mathcal{I}_k^j} \|v_k^j - \E_{\mathcal{I}_k^j} v_k^j\|^2 + \|\E_{\mathcal{I}_k^j} v_k^j\|^2 \\
    =& \E_{\mathcal{I}_k^j} \|\hat{\nabla}_{coord} f_{\mathcal{I}_k^j} (x_k^j) - \hat{\nabla}_{coord} f_{\mathcal{I}_k^j} (x_0^j) - \left( \hat{\nabla}_{coord} f(x_k^j) - \hat{\nabla}_{coord} f(x_0^j) \right) \|^2 + \|\nabla f(x_k^j) + \hat{e}_k^j\|^2 \\
    \le & \E_{\mathcal{I}_k^j} \|\hat{\nabla}_{coord} f_{\mathcal{I}_k^j} (x_k^j) - \hat{\nabla}_{coord} f_{\mathcal{I}_k^j} (x_0^j) - \left( \hat{\nabla}_{coord} f(x_k^j) - \hat{\nabla}_{coord} f(x_0^j) \right) \|^2  + 2\|\nabla f(x_k^j)\|^2 + 2\|\hat{e}_k^j\|^2
\end{align*}
By Lemma~\ref{lemma: variance reduced},
\begin{align*}
    & \E_{\mathcal{I}_k^j} \|\hat{\nabla}_{coord} f_{\mathcal{I}_k^j} (x_k^j) - \hat{\nabla}_{coord} f_{\mathcal{I}_k^j} (x_0^j) - \left( \hat{\nabla}_{coord} f(x_k^j) - \hat{\nabla}_{coord} f(x_0^j) \right) \|^2 \\
    \le & \frac{1}{b} \cdot \frac{1}{n}\sum_{i=1}^n \|\hat{\nabla}_{coord} f_i (x_k^j) - \hat{\nabla}_{coord} f_i (x_0^j) - \left( \hat{\nabla}_{coord} f(x_k^j) - \hat{\nabla}_{coord} f(x_0^j) \right)\|^2 \\
    = & \frac{1}{b} \left( \frac{1}{n} \sum_{i=1}^n \| \hat{\nabla}_{coord} f_i (x_k^j) - \hat{\nabla}_{coord} f_i (x_0^j) \|^2 - \|\hat{\nabla}_{coord} f(x_k^j) - \hat{\nabla}_{coord} f(x_0^j)\|^2 \right) \\
    \le & \frac{1}{b} \cdot \frac{1}{n} \sum_{i=1}^n \| \hat{\nabla}_{coord} f_i (x_k^j) - \hat{\nabla}_{coord} f_i (x_0^j) \|^2 \\
    \le & \frac{1}{b} \cdot \frac{1}{n} \sum_{i=1}^n \left(2\|\nabla f_i(x_k^j) - \nabla f_i(x_0^j)\|^2 + 2\| \hat{\nabla}_{coord} f_i (x_k^j) - \hat{\nabla}_{coord} f_i (x_0^j) - \left(\nabla f_i(x_k^j) - \nabla f_i(x_0^j)\right) \|^2\right) \\
    \le & \frac{1}{b} \cdot \frac{1}{n} \sum_{i=1}^n \left( 2\|\nabla f_i(x_k^j) - \nabla f_i(x_0^j)\|^2 + 4 \|\hat{\nabla}_{coord} f_i (x_k^j) - \nabla f_i (x_k^j)\|^2 + 4\| \hat{\nabla}_{coord} f_i (x_0^j) - \nabla f_i(x_0^j) \|^2 \right) \\
    \le & \frac{1}{b} \cdot \frac{1}{n} \sum_{i=1}^n \left( 2\|\nabla f_i(x_k^j) - \nabla f_i(x_0^j)\|^2 + 8\ell^2 d \mu^2 \right) \\
    \le & \frac{1}{b} \cdot \frac{1}{n} \sum_{i=1}^n \left( 2 \ell^2 \|x_k^j - x_0^j\|^2 + 8\ell^2 d \mu^2 \right) \\
    = & \frac{2\ell^2}{b}\|x_k^j - x_0^j\|^2 + \frac{8 \ell^2 d \mu^2}{b}
\end{align*}
Therefore,
\begin{equation*}
    \E_{\mathcal{I}_k^j} \|v_k^j\|^2 \le \frac{2\ell^2}{b}\|x_k^j - x_0^j\|^2 + \frac{8}{b}  \ell^2 d \mu^2 + 2\|\nabla f(x_k^j)\|^2 + 2 \|\hat{e}_k^j\|^2
\end{equation*}
So we have 
\begin{align*}
    & \E_{\mathcal{I}_k^j} f(x_{k+1}^j) \\
    \le & f(x_k^j) - \eta \|\nabla f(x_k^j)\|^2  - \eta\<\hat{e}_k^j, \nabla f(x_k^j)\> + \frac{\ell \eta^2}{2} \left( \frac{2\ell^2}{b}\|x_k^j - x_0^j\|^2 + \frac{8}{b}  \ell^2 d \mu^2 + 2\|\nabla f(x_k^j)\|^2 + 2 \|\hat{e}_k^j\|^2 \right) \\
    = & f(x_k^j) - \eta(1-\ell\eta)\|\nabla f(x_k^j)\|^2 - \eta\<\hat{e}_k^j, \nabla f(x_k^j)\> + \frac{\ell^3 \eta^2}{b}\|x_k^j - x_0^j\|^2 + \frac{4\ell^3 \eta^2 d \mu^2}{b} + \ell \eta^2 \|\hat{e}_k^j\|^2
\end{align*}
Let $\E_j$ denotes the expectation over $\mathcal{I}_0^k, \mathcal{I}_1^k, \dots$, given $N_j$. Since $\mathcal{I}_{k+1}^j, \mathcal{I}_{k+2}^j, \dots$ are independent of $x_k^j$, the above inequality implies that
\begin{align*}
    & \eta(1-\ell\eta) \E_j\|\nabla f(x_k^j)\|^2 + \eta \E_j\<\hat{e}_k^j, \nabla f(x_k^j)\> \\
    \le & \E_j f(x_k^j) - \E_j f(x_{k+1}^j) + \frac{\ell^3 \eta^2}{b} \E_j\|x_k^j - x_0^j\|^2 + \frac{4\ell^3 \eta^2 d \mu^2}{b} + \ell \eta^2 \E_j \|\hat{e}_k^j\|^2
\end{align*}
Let $k = N_j$, by taking expectation to $N_j$ and using Fubini's theorem, we have
\begin{align*}
    & \eta(1-\ell\eta) \E_{N_j}\E_j\|\nabla f(x_{N_j}^j)\|^2 + \eta \E_{N_j}\E_j\<\hat{e}_{N_j}^j, \nabla f(x_{N_j}^j)\> \\
    \le & \E_{N_j} \left(\E_j f(x_{N_j}^j) - \E_j f(x_{N_j+1}^j)\right) + \frac{\ell^3 \eta^2}{b} \E_{N_j}\E_j\|x_{N_j}^j - x_0^j\|^2 + \frac{4\ell^3 \eta^2 d \mu^2}{b} + \ell \eta^2 \E_{N_j}\E_j \|\hat{e}_{N_j}^j\|^2 \\
    = & \frac{b}{B} \left(f(x_0^j) - \E_{N_j} \E_j f(x_{N_j}^j)\right) + \frac{\ell^3 \eta^2}{b} \E_j \E_{N_j}\|x_{N_j}^j - x_0^j\|^2 + \frac{4\ell^3 \eta^2 d \mu^2}{b} + \ell \eta^2 \E_{N_j}\E_j \|\hat{e}_{N_j}^j\|^2
\end{align*}
Substituting $X_{N_j}^j, x_0^j$ by $\tilde{x}_{j}, \tilde{x}_{j-1}$ and take a further expectation to the past randomness, we get
\begin{align*}
    & \eta(1-\ell\eta) \E\|\nabla f(\tilde{x}_j)\|^2 + \eta \E\<\hat{e}_{N_j}^j, \nabla f(\tilde{x}_j)\> \\
    \le & \frac{b}{B} \left(f(\tilde{x}_{j-1}) - \E f(\tilde{x}_j)\right) + \frac{\ell^3 \eta^2}{b} \E\|\tilde{x}_j - \tilde{x}_{j-1}\|^2 + \frac{4\ell^3 \eta^2 d \mu^2}{b} + \ell \eta^2 \E \|\hat{e}_{N_j}^j\|^2
\end{align*}
Multiplying both sides by $B$, we have 
\begin{align*}
    & \eta(1-\ell\eta) B \E\|\nabla f(\tilde{x}_j)\|^2 + \eta B \E\<\hat{e}_{N_j}^j, \nabla f(\tilde{x}_j)\> \\
    \le & b \left(f(\tilde{x}_{j-1}) - \E f(\tilde{x}_j)\right) + \frac{\ell^3 \eta^2 B}{b} \E\|\tilde{x}_j - \tilde{x}_{j-1}\|^2 + \frac{4\ell^3 \eta^2 d \mu^2 B}{b} + \ell \eta^2 B \E \|\hat{e}_{N_j}^j\|^2
\end{align*}

\end{proof}

\begin{lemma}
\label{lemma: ZO-SCSG-Coord-2}
Suppose $2\eta^2\ell^2 B < b^2$, under Assumption~\ref{assum: basic}
\begin{align*}
    & \left( \frac{b}{B} - \frac{2\eta^2 \ell^2}{b}  \right) \E\|\tilde{x}_j - \tilde{x}_{j-1}\|^2 + 2 \eta \E \<\hat{e}_{N_j}^j, \tilde{x}_j - \tilde{x}_{j-1}\> \\
    \le & - 2 \eta \E \<\nabla f(\tilde{x}_j), \tilde{x}_j - \tilde{x}_{j-1}\>  +\frac{8\eta^2 \ell^2 d \mu^2}{b}  + 2\eta^2 \E \|\nabla f(\tilde{x}_j)\|^2 + 2\eta^2 \E \|\hat{e}_{N_j}^j\|^2
\end{align*}
\end{lemma}

\begin{proof}
Since $x_{k+1}^j = x_k^j - \eta v_k^j$, we have 
\begin{align*}
    & \E_{\mathcal{I}_k^j} \|x_{k+1}^j - x_0^j\|^2 \\
    = & \|x_k^j - x_0^j\|^2 - 2 \eta \<\E_{\mathcal{I}_k^j} v_k^j, x_k^j - x_0^j\> + \eta^2 \E_{\mathcal{I}_k^j} \|v_k^j\|^2 \\
    = & \|x_k^j - x_0^j\|^2 - 2 \eta \<\nabla f(x_k^j), x_k^j - x_0^j\> - 2 \eta \<\hat{e}_k^j, x_k^j - x_0^j\> + \eta^2 \E_{\mathcal{I}_k^j} \|v_k^j\|^2 \\
    \le & \|x_k^j - x_0^j\|^2 - 2 \eta \<\nabla f(x_k^j), x_k^j - x_0^j\> - 2 \eta \<\hat{e}_k^j, x_k^j - x_0^j\> \\
    & + \eta^2 \left( \frac{2\ell^2}{b}\|x_k^j - x_0^j\|^2 + \frac{8}{b}  \ell^2 d \mu^2 + 2\|\nabla f(x_k^j)\|^2 + 2 \|\hat{e}_k^j\|^2 \right)  \\
    = & \left(1 + \frac{2\eta^2 \ell^2}{b}  \right)\|x_k^j - x_0^j\|^2 - 2 \eta \<\nabla f(x_k^j), x_k^j - x_0^j\> - 2 \eta \<\hat{e}_k^j, x_k^j - x_0^j\>  \\
    & +\frac{8\eta^2 \ell^2 d \mu^2}{b} + 2\eta^2 \|\nabla f(x_k^j)\|^2 + 2\eta^2 \|\hat{e}_k^j\|^2
\end{align*}
Using the notation $\E_j$ we have
\begin{align*}
    & 2 \eta \E_j \<\nabla f(x_k^j), x_k^j - x_0^j\> + 2 \eta \E_j \<\hat{e}_k^j, x_k^j - x_0^j\> \\
    \le & \left(1 + \frac{2\eta^2 \ell^2}{b}  \right) \E_j\|x_k^j - x_0^j\|^2 - \E_j \|x_{k+1}^j - x_0^j\|^2 +\frac{8\eta^2 \ell^2 d \mu^2}{b} + 2\eta^2 \|\nabla f(x_k^j)\|^2 + 2\eta^2 \E_j \|\hat{e}_k^j\|^2
\end{align*}
Let $k = N_j$, by taking expectation with respect to $N_j$ and using Fubini's theorem, we have 
\begin{align*}
    & 2 \eta \E_{N_j}\E_j \<\nabla f(x_{N_j}^j), x_{N_j}^j - x_0^j\> + 2 \eta \E_{N_j}\E_j \<\hat{e}_{N_j}^j, x_{N_j}^j - x_0^j\> \\
    \le & \left(1 + \frac{2\eta^2 \ell^2}{b}  \right) \E_{N_j} \E_j\|x_{N_j}^j - x_0^j\|^2 - \E_{N_j} \E_j \|x_{{N_j}+1}^j - x_0^j\|^2 +\frac{8\eta^2 \ell^2 d \mu^2}{b} \\
    & + 2\eta^2 \E_{N_j} \|\nabla f(x_{N_j}^j)\|^2 + 2\eta^2 \E_{N_j} \E_j \|\hat{e}_{N_j}^j\|^2 \\
    = & \left(- \frac{b}{B} + \frac{2\eta^2 \ell^2}{b}  \right) \E_{N_j} \E_j\|x_{N_j}^j - x_0^j\|^2  +\frac{8\eta^2 \ell^2 d \mu^2}{b}  + 2\eta^2 \E_{N_j} \|\nabla f(x_{N_j}^j)\|^2 + 2\eta^2 \E_{N_j} \E_j \|\hat{e}_{N_j}^j\|^2
\end{align*}
Substituting $x_{N_j}^j, x_0^j$ by $\tilde{x}_{j}, \tilde{x}_{j-1}$ and take a further expectation to the past randomness, we get
\begin{align*}
    & 2 \eta \E \<\nabla f(\tilde{x}_j), \tilde{x}_j - \tilde{x}_{j-1}\> + 2 \eta \E \<\hat{e}_{N_j}^j, \tilde{x}_j - \tilde{x}_{j-1}\> \\
    \le & \left(- \frac{b}{B} + \frac{2\eta^2 \ell^2}{b}  \right) \E\|\tilde{x}_j - \tilde{x}_{j-1}\|^2  +\frac{8\eta^2 \ell^2 d \mu^2}{b}  + 2\eta^2 \E \|\nabla f(\tilde{x}_j)\|^2 + 2\eta^2 \E \|\hat{e}_{N_j}^j\|^2
\end{align*}
Swapping the order we get
\begin{align*}
    & \left( \frac{b}{B} - \frac{2\eta^2 \ell^2}{b}  \right) \E\|\tilde{x}_j - \tilde{x}_{j-1}\|^2 + 2 \eta \E \<\hat{e}_{N_j}^j, \tilde{x}_j - \tilde{x}_{j-1}\> \\
    \le & - 2 \eta \E \<\nabla f(\tilde{x}_j), \tilde{x}_j - \tilde{x}_{j-1}\>  +\frac{8\eta^2 \ell^2 d \mu^2}{b}  + 2\eta^2 \E \|\nabla f(\tilde{x}_j)\|^2 + 2\eta^2 \E \|\hat{e}_{N_j}^j\|^2
\end{align*}

\end{proof}

\begin{lemma}
\label{lemma: ZO-SCSG-Coord-3}
\begin{align*}
    \frac{b}{B}\E \<\hat{e}_k^j, \tilde{x}_j - \tilde{x}_{j-1}\> = -\eta \E  \<\hat{e}_{N_j}^j,\nabla f(\tilde{x}_j)\> - \eta  \E \|\hat{e}_{N_j}^j\|^2
\end{align*}
\end{lemma}

\begin{proof}
Let $M_k^j = \<\hat{e}_k^j, x_k^j - x_0^j\>$. Then we have 
\begin{equation*}
    \E_{N_j} \<\hat{e}_k^j, \tilde{x}_j - \tilde{x}_{j-1}\> = \E_{N_j} M_{N_j}^j
\end{equation*}
Since $N_j$ is independent of $x_0^j, \hat{e}_k^j$, we have 
\begin{equation*}
    \E \<\hat{e}_k^j, \tilde{x}_j - \tilde{x}_{j-1}\> = \E M_{N_j}^j
\end{equation*}
Also we have $M_0^j = 0$. On the other hand,
\begin{align*}
    \E_{\mathcal{I}_k^j} \left( M_{k+1}^j - M_k^j \right) &= \E_{\mathcal{I}_k^j} \<\hat{e}_k^j, x_{k+1}^j - x_k^j\> = -\eta \<\hat{e}_k^j,  \E_{\mathcal{I}_k^j} v_k^j\> \\
    & = -\eta \<\hat{e}_k^j, \nabla f(x_k^j)\> - \eta \|\hat{e}_k^j\|^2
\end{align*}
Using the same notation $\E_j$ as in the proof in Lemma~\ref{lemma: ZO-SCSG-Coord-1} and Lemma~\ref{lemma: ZO-SCSG-Coord-2}, we have 
\begin{align*}
    \E_j \left( M_{k+1}^j - M_k^j \right) = -\eta \<\hat{e}_k^j, \E_j \nabla f(x_k^j)\> - \eta \E_j \|\hat{e}_k^j\|^2
\end{align*}
Let $k = N_j$, by taking the expectation with respect to $N_j$ and using Fubini's theorem, we have , we have 
\begin{align*}
    \frac{b}{B} \E_{N_j} \E_j M_{N_j}^j  = -\eta \<\hat{e}_{N_j}^j, \E_{N_j}\E_j \nabla f(x_{N_j}^j)\> - \eta \E_{N_j}\E_j \|\hat{e}_{N_j}^j\|^2
\end{align*}
Substituting $x_{N_j}^j, x_0^j$ by $\tilde{x}_{j}, \tilde{x}_{j-1}$ and take a further expectation to the past randomness, we get
\begin{align*}
    \frac{b}{B}\E \<\hat{e}_k^j, \tilde{x}_j - \tilde{x}_{j-1}\> = -\eta \E  \<\hat{e}_{N_j}^j,\nabla f(\tilde{x}_j)\> - \eta  \E \|\hat{e}_{N_j}^j\|^2
\end{align*}

\end{proof}

\begin{lemma}[\cite{ji2019improved}]
\label{lemma: ZO-SCSG-Coord-4}
Define $\hat{e}_j = v_j - \hat{\nabla}_{coord} f(x_0^j) = \hat{\nabla}_{coord} f_{\mathcal{I}_j} (\tilde{x}_{j-1}) - \hat{\nabla}_{coord} f(\tilde{x}_{j-1})$, we have
\begin{equation*}
    \|\hat{e}_j\|^2 \le \frac{3(2\ell^2 d \mu^2 + \sigma^2) }{B}
\end{equation*}
\end{lemma}
\begin{proof}
The proof directly follows from Lemma 4 in \cite{ji2019improved}.
\end{proof}

\begin{proof}[\textbf{Proof of Lemma~\ref{lemma: ZO-SCSG} (Option \uppercase\expandafter{\romannumeral1})}]
Multiplying Lemma~\ref{lemma: ZO-SCSG-Coord-1} by $2$, Lemma~\ref{lemma: ZO-SCSG-Coord-2} by $\frac{b}{\eta}$ and summing them up, we have 
\begin{align*}
    & 2\eta B \left( 1-\eta \ell - \frac{b}{B} \right) \E\|\nabla f(\tilde{x}_j)\|^2 + \frac{b^3 - 2\eta^2 \ell^2 b B - 2\ell^3 \eta^3 B^2}{\eta b B} \E \|\tilde{x}_j - \tilde{x}_{j-1}\|^2 \\
    & + 2\eta B \E\<\hat{e}_{N_j}^j, \nabla f(\tilde{x}_j)\> + 2 b \<\hat{e}_{N_j}^j, \tilde{x}_j - \tilde{x}_{j-1}\> \\
    \le & -2b \E\<\nabla f(\tilde{x}_j), \tilde{x}_j - \tilde{x}_{j-1}\> + 2b \left( f(\tilde{x}_{j-1}) - f(\tilde{x}_j) \right) + ( \ell \frac{B}{b} + \frac{1}{\eta}) 8\eta^2 \ell^2 d \mu^2 + (2\eta^2 \ell B + 2\eta b) \E\|\hat{e}_{N_j}^j\|^2
\end{align*}
By Lemma~\ref{lemma: ZO-SCSG-Coord-3}, 
\begin{align*}
    2\eta B \E\<\hat{e}_{N_j}^j, \nabla f(\tilde{x}_j)\> + 2 b \<\hat{e}_{N_j}^j, \tilde{x}_j - \tilde{x}_{j-1}\> = - 2\eta B \E\|\hat{e}_{N_j}^j\|^2
\end{align*}
So the above inequality can simplified as
\begin{align*}
    & 2\eta B \left( 1-\eta \ell - \frac{b}{B} \right) \E\|\nabla f(\tilde{x}_j)\|^2 + \frac{b^3 - 2\eta^2 \ell^2 b B - 2\ell^3 \eta^3 B^2}{\eta b B} \E \|\tilde{x}_j - \tilde{x}_{j-1}\|^2 \\
    \le & -2b \E\<\nabla f(\tilde{x}_j), \tilde{x}_j - \tilde{x}_{j-1}\> + 2b \left( f(\tilde{x}_{j-1}) - f(\tilde{x}_j) \right) + ( \ell \frac{B}{b} + \frac{1}{\eta}) 8\eta^2 \ell^2 d \mu^2 \\
    &+ (2\eta^2 \ell B + 2\eta b + 2\eta B) \E\|\hat{e}_{N_j}^j\|^2
\end{align*}
Using the fact that $2\<a,b\>\le \beta\|a\|^2 + \frac{1}{\beta}\|b\|^2$ for any $\beta>0$, we have
\begin{align*}
    &-2b \E\<\nabla f(\tilde{x}_j), \tilde{x}_j - \tilde{x}_{j-1}\> \\
    \le & \frac{\eta b B}{b^3 - 2\eta^2 \ell^2 b B - 2\ell^3 \eta^3 B^2} b^2 \E \|\nabla f(\tilde{x}_j)\|^2 + \frac{b^3 - 2\eta^2 \ell^2 b B - 2\ell^3 \eta^3 B^2}{\eta b B} \E \|\tilde{x}_j - \tilde{x}_{j-1}\|^2
\end{align*}
Then we conclude that
\begin{align*}
    &\eta B \left( 2- 2\eta\ell - 2\frac{b}{B} - \frac{b^3}{b^3 - 2\eta^2 \ell^2 b B - 2\ell^3 \eta^3 B^2} \right) \E\|\nabla f(\tilde{x}_j)\|^2 \\
    \le & 2b \left( f(\tilde{x}_{j-1}) - f(\tilde{x}_j) \right) + ( \ell \frac{B}{b} + \frac{1}{\eta}) 8\eta^2 \ell^2 d \mu^2 + (2\eta^2 \ell B + 2\eta b + 2\eta B) \E\|\hat{e}_{N_j}^j\|^2
\end{align*}
Multiplying both sides by $\frac{\ell}{b}$, we have 
\begin{align*}
    & \eta \ell \frac{B}{b} \left( 2- 2\eta\ell - 2\frac{b}{B} - \frac{b^3}{b^3 - 2\eta^2 \ell^2 b B - 2\ell^3 \eta^3 B^2} \right) \E\|\nabla f(\tilde{x}_j)\|^2 \\
    \le & 2\ell \left( f(\tilde{x}_{j-1}) - f(\tilde{x}_j) \right) + \left(\frac{\ell^2 B}{b^2} + \frac{\ell }{b \eta}\right) 8\eta^2 \ell^2 d\mu^2 + (2\eta^2 \ell^2 \frac{B}{b} + 2\eta \ell + 2\eta \ell \frac{B}{b}) \E\|\hat{e}_{N_j}^j\|^2
\end{align*}
Let $\eta \ell = \gamma \left( \frac{b}{B} \right)^{\frac{2}{3}}$ and $b \ge 1, \frac{B}{b} \ge 8 \ge \frac{8}{b}$, we have 
\begin{align*}
    b^3 - 2\eta^2 \ell^2 b B - 2\ell^3 \eta^3 B^2 = &  b^3 \left(  1 - 2\gamma^2 b (\frac{b}{B})^{\frac{1}{3}} - 2 \gamma^3 b^{-1}\right) \\
    \ge & b^3 (1 - \gamma^2 - 2 \gamma^3) 
\end{align*}
Then the above inequality can be simplified as 
\begin{align*}
    & \gamma \left( \frac{B}{b} \right)^{\frac{1}{3}} \left(2 - 2\gamma (\frac{b}{B})^{\frac{2}{3}} - 2 \frac{b}{B} - \frac{1}{1 - \gamma^2 -2\gamma^3} \right)\E\|\nabla f(\tilde{x}_j)\|^2 \\
    \le & 2\ell \left( f(\tilde{x}_{j-1}) - f(\tilde{x}_j) \right) + 2\gamma \left( 1 + \gamma \left(\frac{b}{B}\right)^{\frac{2}{3}} + \frac{b}{B} \right) \left(\frac{B}{b}\right)^{\frac{1}{3}} \E\|\hat{e}_{N_j}^j\|^2  + \gamma \left( \gamma (\frac{b}{B})^{\frac{1}{3}} + (\frac{b}{B})^{\frac{2}{3}} \right) \frac{8\ell^2 d \mu^2}{b}
\end{align*}
Since $B \ge 8b, \gamma \le \frac{1}{4}$, we have 
\begin{align*}
    2 - 2\gamma (\frac{b}{B})^{\frac{2}{3}} - 2 \frac{b}{B} - \frac{1}{1 - \gamma^2 -2\gamma^3} \ge & 2-\frac{\gamma}{2} - \frac{1}{4} - \frac{1}{1 - \gamma^2 -2\gamma^3} \ge 0.5 \\
    1 + \gamma \left(\frac{b}{B}\right)^{\frac{2}{3}} + \frac{b}{B} \le & 1 + \frac{\gamma}{4} + \frac{1}{8} \le \frac{19}{16} \\
    \gamma (\frac{b}{B})^{\frac{1}{3}} + (\frac{b}{B})^{\frac{2}{3}} \le & \frac{1}{2}\gamma + \frac{1}{4} \le \frac{3}{8}
\end{align*}
Thus we have
\begin{align*}
    \left( \frac{B}{b} \right)^{\frac{1}{3}} \E\|\nabla f(\tilde{x}_j)\|^2 \le  \frac{4\ell}{\gamma} \left(f(\tilde{x}_{j-1}) - f(\tilde{x}_j)\right) + 5 \left( \frac{B}{b} \right)^{\frac{1}{3}} \E\|\hat{e}_{N_j}^j\|^2 + \frac{6\ell^2 d \mu^2 }{b}
\end{align*}
Using Lemma~\ref{lemma: ZO-SCSG-Coord-4}, we have 
\begin{align*}
    \E\|\hat{e}_{N_j}^j\|^2 = & \E\|\nabla f(x_{N_j}^j) - \hat{\nabla}_{coord} f(x_{N_j}^j) + \hat{e}_j \|^2 \\
    \le & 2 \|\nabla f(x_{N_j}^j) - \hat{\nabla}_{coord} f(x_{N_j}^j)\|^2 + 2\E\|\hat{e}_j\|^2 \\
    \le & 2 \ell^2 d \mu^2 + 2 \left(\frac{3(2\ell^2 d \mu^2 + \sigma^2) }{B} \right)
\end{align*}
Thus we obtain
\begin{align*}
    \left( \frac{B}{b} \right)^{\frac{1}{3}} \E\|\nabla f(\tilde{x}_j)\|^2 \le  \frac{4\ell}{\gamma} \left(f(\tilde{x}_{j-1}) - f(\tilde{x}_j)\right) + \frac{30\sigma^2 }{b^{\frac{1}{3}} B^{\frac{2}{3}}} + c \left(\frac{B}{b}\right)^{\frac{1}{3 }} \ell^2 d \mu^2
\end{align*}
where $c$ is a sufficient large constant. Telescope the sum in $j=1, \dots, T$, and using the definition of $\tilde{x}_T^*$, we finally get
\begin{equation*}
    \E \|\tilde{x}_T^*\|^2 = \frac{1}{T} \sum_{j=1}^T \E \|\nabla f(\tilde{x}_j)\|^2 \le \frac{\frac{4\ell}{\gamma}}{T \left( \frac{B}{b/d}\right)^{\frac{1}{3}}} \cdot \E\left( f(\tilde{x}_0 ) - f(\tilde{x}_T) \right) + \frac{30 \sigma^2}{B} + c \ell^2 d \mu^2
\end{equation*}

\end{proof}


\subsection{Proof of Second-Order Stationary Point (\texorpdfstring{\textbf{Option \uppercase\expandafter{\romannumeral1}}}{Lg})}

\begin{proof}[\textbf{Proof of Theorem~\ref{thm: ZO-SCSG-NCF}}]
Let $N_1$ and $N_2$ be the number of times we reach Line 7 and 9 of Algorithm~\ref{alg: ZO-SCSG-NCF}. From Lemma~\ref{lemma: ZO-SCSG} of ZO-SCSG we know that for one epoch with size $B = \max\{1, \frac{480 \sigma^2}{\epsilon^2}\} $, mini-batch size $b\ge 1$ and the smoothing parameter $\mu = \frac{\epsilon}{4\sqrt{cd}\ell}$, we have 
\begin{equation*}
    \E \|x_{t+1}\|^2 \le \frac{4 \ell }{\gamma} \left(\frac{b}{B}\right)^{\frac{1}{3}} \E \left( f(x_t) - f(x_{t+1}) \right) + \frac{\epsilon^2}{8}
\end{equation*}
Then, if $\|\nabla f(x_{t+1})\| \ge 
\frac{\epsilon}{2}$, we have $x_{t+1} = x_{t+1}$; if $v = \bot$, we set $x_{t+1} = x_{t+1}$ for  if $v \neq \bot$, we have $f(x_{t+1}) - \E f(x_{t+1}) \ge \frac{\delta^3}{12 \rho^2}$ (here the expectation is taken on the randomness of sign of $v$). Thus we have 
\begin{equation*}
    \frac{\gamma B^{\frac{1}{3}}}{ 4\ell b^{\frac{1}{3}} } \E \left[\sum_{t=0}^{K-1} \left( \|\nabla f(x_{t+1} )\|^2 - \frac{\epsilon^2}{8} \right) \right] + \frac{\delta^3 }{12 \rho^2 } \E [N_2] \le \Delta_f
\end{equation*}
On one hand, since we have chosen $K$ such that $K \ge \Omega\left( \frac{\ell b^{\frac{1}{3}} \Delta_f}{\epsilon^2 B^{\frac{1}{3}}}\right) = \Omega \left( \frac{\ell b^{\frac{1}{3}} \Delta_f}{\epsilon^2 (1+\frac{\sigma^2}{\epsilon^2})^{\frac{1}{3}}} \right)$, then by Markov's inequality, with probability at least $\frac{5}{6}$, it satisfies $\sum_{t=0}^{K-1} \|\nabla f(x_{t+1})\|^2 \le \frac{\epsilon}{4} K$. As a sequence, at least half of the indices $t=0, \dots, K-1$ will satisfy $\|\nabla f(x_{t+1})\| \le \frac{\epsilon}{2}$, which means that $N_1 \ge \frac{K}{2}$.

On the other hand, we have $\frac{\delta^3}{12 \rho^2} \E[N_2] \le \Delta_f + \frac{K \gamma B^{\frac{1}{3}} \epsilon^2}{32 \ell b^{\frac{1}{3}}}$. Since $K \ge \Omega\left( \frac{\ell b^{\frac{1}{3}} \Delta_f}{\epsilon^2 B^{\frac{1}{3}}}\right) = \Omega \left( \frac{\ell b^{\frac{1}{3}} \Delta_f}{\epsilon^2 (1+\frac{\sigma^2}{\epsilon^2})^{\frac{1}{3}}} \right)$, we have $ \E[N_2] \le \frac{K \gamma B^{\frac{1}{3}} \epsilon^2 \rho^2}{\ell \delta^3 b^{\frac{1}{3}}}$. As long as $B \le \mathcal{O}\left( \frac{\ell^3 \delta^9 b}{\epsilon^6 \rho^6} \right)$, or equivalently $b \ge \Omega\left( \frac{B\epsilon^6 \rho^6}{\delta^9 \ell^3} \right)$, we have $\E[N_2] \le \frac{K}{12}$. Therefore, with provability at least $\frac{5}{6}$, it satisfies $N_2 \le 
\frac{k}{2}$.

Since $N_1 \ge N_2$, this means with probability at least $\frac{2}{3}$ the algorithm must terminate and output some $x_{t+1}$ in an iteration.

Finally, the per-iteration complexity of Algorithm~\ref{alg: ZO-SCSG-NCF} is dominated by $\tilde{\mathcal{O}}(B)$ stochastic gradient estimators per iteration for both \hyperref[alg: ZO-SCSG]{ZO-SCSG} and estimating $\|\nabla f(x_{t+1})\|$, as well as $\tilde{\mathcal{O}}(\frac{\ell^2}{\delta^2})$ invoking Algorithm~\ref{alg: ZO-NCF-Online}. The total function query complexity is 
\begin{align*}
    \tilde{\mathcal{O}}\left( d K \left( B+ \frac{\ell^2}{\delta^2}  \right) \right) &= \tilde{\mathcal{O}}\left( d \left( \frac{\ell b^{\frac{1}{3}} \Delta_f }{\epsilon^2 B^{\frac{1}{3}}} \right) \left( B+ \frac{\ell^2}{\delta^2} \right) \right) \\
    &= \tilde{\mathcal{O}}\left( d \left( \frac{\ell \max\{1, \frac{B^{\frac{1}{3}}\epsilon^2 \rho^2}{\delta^3 \ell}\} \Delta_f }{\epsilon^2 B^{\frac{1}{3}}} \right) \left( B+ \frac{\ell^2}{\delta^2} \right) \right) \\
    &= \tilde{\mathcal{O}}\left( d \left( \frac{\ell (1 + \frac{B^{\frac{1}{3}}\epsilon^2 \rho^2}{\delta^3 \ell}) \Delta_f }{\epsilon^2 B^{\frac{1}{3}}} \right) \left( B+ \frac{\ell^2}{\delta^2} \right) \right) \\
    &= \tilde{\mathcal{O}}\left( d \left( \frac{\ell  \Delta_f }{\epsilon^2 B^{\frac{1}{3}}} + \frac{\rho^2 \Delta_f}{\delta^3} \right) \left( B+ \frac{\ell^2}{\delta^2} \right) \right) \\
    &= \tilde{\mathcal{O}}\left( d \left( \frac{\ell  \Delta_f }{\epsilon^2 \max\{1, \frac{\sigma^2}{\epsilon^2}\}^{\frac{1}{3}}} + \frac{\rho^2 \Delta_f}{\delta^3} \right) \left( \max\{1, \frac{\sigma^2}{\epsilon^2}\} + \frac{\ell^2}{\delta^2} \right) \right) \\
    &= \tilde{\mathcal{O}}\left( d \left( \frac{\ell  \Delta_f }{\epsilon^\frac{4}{3} \sigma^{\frac{2}{3}} } + \frac{\rho^2 \Delta_f}{\delta^3} \right) \left( \frac{\sigma^2 }{\epsilon^2} + \frac{\ell^2}{\delta^2} \right) + d \frac{\ell \Delta_f}{\epsilon^2} \frac{\ell^2}{\delta^2}  \right)
\end{align*}

\end{proof}

\subsection{One Epoch Analysis of  \texorpdfstring{\hyperref[alg: ZO-SCSG]{ZO-SCSG} (\textbf{Option \uppercase\expandafter{\romannumeral2}})}{Lg}}

From Algorithm~\ref{alg: ZO-SCSG}, we know that all randomness in epoch $j$, iteration $k$ come from three part: $\left. 1\right)$ random selection of $\mathcal{I}_{k-1}^j$ in Line 7; $\left. 2 \right)$ random direction of $u$ in estimating the gradient in Line 9; $\left. 3\right)$ random generation of $N_j$.

\begin{lemma}
Under Assumption~\ref{assum: basic},
\begin{equation*}
    \E_u\E_{\mathcal{I}_k^j} \|v_k^j\|^2 \le \frac{3d \ell^2}{b} \|x_k^j - x_0^j\|^2 + \frac{3 \ell^2 d^2 \mu^2}{2 b} + 2\|\nabla f_\mu (x_k^j)\|^2 + 2 \|\hat{e}_j\|^2 
\end{equation*}

\end{lemma}

\begin{proof}
Define the following notation, 
\begin{equation*}
    \hat{e}_j = v_j - \nabla f_\mu(x_0^j)
\end{equation*}
Then we have 
\begin{equation*}
    \E_u \E_{\mathcal{I}_k^j} v_k^j = \E_u \E_{\mathcal{I}_k^j} \left(\hat{\nabla}_{rand} f_{\mathcal{I}_k^j}(x_k^j) - \hat{\nabla}_{rand} f_{\mathcal{I}_k^j}(x_0^j) + v_j \right) = \nabla f_\mu(x_k^j) + \hat{e}_j
\end{equation*}
Using the fact that $\E\|a\|^2 = \E\|a- \E a\|^2+ \|\E a\|^2$, we have 
\begin{align*}
    &\E_u\E_{\mathcal{I}_k^j} \|v_k^j\|^2 = \E_u\E_{\mathcal{I}_k^j} \|v_k^j - \E_u\E_{\mathcal{I}_k^j} v_k^j\|^2 + \|\E_u\E_{\mathcal{I}_k^j} v_k^j\|^2 \\
    = & \E_u\E_{\mathcal{I}_k^j} \| \hat{\nabla}_{rand} f_{\mathcal{I}_k^j} (x_k^j) - \hat{\nabla}_{rand} f_{\mathcal{I}_k^j} (x_0^j) - \left( \nabla f_\mu (x_k^j) - \nabla f_\mu(x_0^j) \right) \|^2 + \| \nabla f_\mu(x_k^j) + \hat{e}_j \|^2 \\
    \le & \E_u\E_{\mathcal{I}_k^j} \| \hat{\nabla}_{rand} f_{\mathcal{I}_k^j} (x_k^j) - \hat{\nabla}_{rand} f_{\mathcal{I}_k^j} (x_0^j) - \left( \nabla f_\mu (x_k^j) - \nabla f_\mu(x_0^j) \right) \|^2 + 2\|\nabla f_\mu (x_k^j)\|^2 + 2 \|\hat{e}_j\|^2.
\end{align*}
By Lemma~\ref{lemma: variance reduced}, we have
\begin{align*}
    & \E_u\E_{\mathcal{I}_k^j} \| \hat{\nabla}_{rand} f_{\mathcal{I}_k^j} (x_k^j) - \hat{\nabla}_{rand} f_{\mathcal{I}_k^j} (x_0^j) - \left( \nabla f_\mu (x_k^j) - \nabla f_\mu(x_0^j) \right) \|^2 \\
    = & \frac{1}{b^2}\E_u \E_{\mathcal{I}_k^j} \| \sum_{i \in \mathcal{I}_k^j} \left( \hat{\nabla}_{rand} f_i (x_k^j) - \hat{\nabla}_{rand} f_i (x_0^j) - \left( \nabla f_\mu (x_k^j) - \nabla f_\mu(x_0^j) \right) \right) \|^2 \\
    =& \frac{1}{b^2} \E_u\E_{\mathcal{I}_k^j} \sum_{i \in \mathcal{I}_k^j} \|\hat{\nabla}_{rand} f_i (x_k^j) - \hat{\nabla}_{rand} f_i (x_0^j) - \left( \nabla f_\mu (x_k^j) - \nabla f_\mu(x_0^j) \right)\|^2 \\
    & +\frac{1}{b^2} \E_u\E_{\mathcal{I}_k^j} \sum_{i \neq j} \left\langle \hat{\nabla}_{rand} f_i (x_k^j) - \hat{\nabla}_{rand} f_i (x_0^j) - \left( \nabla f_\mu (x_k^j) - \nabla f_\mu(x_0^j) \right),  \right.\\
    & \left. \hat{\nabla}_{rand} f_j (x_k^j) - \hat{\nabla}_{rand} f_j (x_0^j) - \left( \nabla f_\mu (x_k^j) - \nabla f_\mu(x_0^j) \right) \right\rangle \\
    = & \frac{1}{b^2} \E_u\E_{\mathcal{I}_k^j} \sum_{i \in \mathcal{I}_k^j} \|\hat{\nabla}_{rand} f_i (x_k^j) - \hat{\nabla}_{rand} f_i (x_0^j) - \left( \nabla f_\mu (x_k^j) - \nabla f_\mu(x_0^j) \right)\|^2 \\
    = & \frac{1}{b} \E_u \frac{1}{n} \sum_{i \in [n]} \|\hat{\nabla}_{rand} f_i (x_k^j) - \hat{\nabla}_{rand} f_i (x_0^j) - \left( \nabla f_\mu (x_k^j) - \nabla f_\mu(x_0^j) \right)\|^2 \\
    = & \frac{1}{b} \cdot  \E_u \left( \frac{1}{n} \sum_{i \in [n]} \|\hat{\nabla}_{rand} f_i (x_k^j) - \hat{\nabla}_{rand} f_i (x_0^j)\|^2 \right) - \frac{1}{b} \|\nabla f_\mu (x_k^j) - \nabla f_\mu(x_0^j)\|^2  \\
    \le & \frac{1}{b} \cdot  \E_u \left( \frac{1}{n} \sum_{i \in [n]} \|\hat{\nabla}_{rand} f_i (x_k^j) - \hat{\nabla}_{rand} f_i (x_0^j)\|^2 \right) \\ 
    \overset{\textrm{Lemma~\ref{lemma: RandGradEst}}}{\le} & \frac{3d \ell^2}{b} \|x_k^j - x_0^j\|^2 + \frac{3 \ell^2 d^2 \mu^2}{2 b}
\end{align*}
Therefore
\begin{equation*}
    \E_u\E_{\mathcal{I}_k^j} \|v_k^j\|^2 \le \frac{3d \ell^2}{b} \|x_k^j - x_0^j\|^2 + \frac{3 \ell^2 d^2 \mu^2}{2 b} + 2\|\nabla f_\mu (x_k^j)\|^2 + 2 \|\hat{e}_j\|^2 
\end{equation*}

\end{proof}

\begin{lemma}
\label{lemma: ZO-SCSG-rand-1}
Suppose $\eta\ell <1$, then under Assumption~\ref{assum: basic},
    \begin{align*}
    & \eta (1-\ell \eta) B \E \|\nabla f_\mu (\tilde{x}_j)\|^2 + \eta B \E\<\hat{e}_j, \nabla f_\mu (\tilde{x}_j)\> \\
    \le & b/d \E \left( f_\mu (\tilde{x}_{j-1}) - f_\mu(\tilde{x}_j) \right) + \frac{3d \ell^3 \eta^2 B}{2b}  \E\|\tilde{x}_j - \tilde{x}_{j-1}\|^2 + \frac{3 \ell^3 \eta^2 d^2 \mu^2 B}{4 b}  + \ell \eta^2 B \E\|\hat{e}_j\|^2
    \end{align*}
\end{lemma}

\begin{proof}
By Lemma~\ref{lemma: Lipschitz}, we have 
\begin{equation*}
    f_\mu(x_{k+1}^j) \le f_\mu(x_k^j) - \<x_{k+1}^j - x_k^j, \nabla f(x_k^j)\> + \frac{\ell}{2} \|x_{k+1}^j - x_k^j\|^2 \le f_\mu(x_k^j) - \eta \<v_k^j, \nabla f_\mu(x_k^j)\> + \frac{\ell \eta^2 }{2} \|v_k^j\|^2
\end{equation*}
Taking expectation over the above inequality we have 
\begin{align*}
    f_\mu(x_{k+1}^j) 
    \le & f_\mu(x_k^j) - \eta \<\nabla f_\mu (x_k^j) + \hat{e}_j, \nabla f_\mu(x_k^j)\> + \frac{\ell \eta^2 }{2} \E_u\E_{\mathcal{I}_k^j}\|v_k^j\|^2 \\
    = & f_\mu(x_k^j) -\eta \|\nabla f_\mu (x_k^j)\|^2 - \eta \<\hat{e}_j, \nabla f_\mu (x_k^j)\> + \frac{\ell \eta^2 }{2} \E_u\E_{\mathcal{I}_k^j}\|v_k^j\|^2 \\
    \le & f_\mu(x_k^j) -\eta (1-\ell \eta) \|\nabla f_\mu (x_k^j)\|^2 - \eta \<\hat{e}_j, \nabla f_\mu (x_k^j)\> \\
    & + \frac{3d \ell^3 \eta^2}{2b} \|x_k^j - x_0^j\|^2 + \frac{3 \ell^3 \eta^2 d^2 \mu^2}{4 b}  + \ell \eta^2 \|\hat{e}_j\|^2 
\end{align*}
Let $\E_j$ denote the expectation over $\mathcal{I}_0^k, \mathcal{I}_1^k, \dots,$ given $N_j$. Since $\mathcal{I}_{k+1}^j, \mathcal{I}_{k+2}^j, \dots$ are independent of $x_k^j$, the above inequality implies that
\begin{align*}
    & \eta (1-\ell \eta)\E_j \|\nabla f_\mu (x_k^j)\|^2 + \eta \E_j\<\hat{e}_j, \nabla f_\mu (x_k^j)\> \\
    \le & \E_j f_\mu(x_k^j) - \E_j f_\mu(x_{k+1}^j) + \frac{3d \ell^3 \eta^2}{2b} \E_j \|x_k^j - x_0^j\|^2 + \frac{3 \ell^3 \eta^2 d^2 \mu^2}{4 b}  + \ell \eta^2 \|\hat{e}_j\|^2
\end{align*}
Let $k=N_j$, by taking expectation to $N_j$ and using Fubini's theorem, we have 
\begin{align*}
    & \eta (1-\ell \eta)\E_{N_j}\E_j \|\nabla f_\mu (x_{N_j}^j)\|^2 + \eta \E_{N_j}\E_j\<\hat{e}_j, \nabla f_\mu (x_{N_j}^j)\> \\
    \le & \E_{N_j} \left( \E_j f_\mu(x_{N_j}^j) - \E_j f_\mu(x_{N_j+1}^j) \right) + \frac{3d \ell^3 \eta^2}{2b} \E_{N_j}\E_j \|x_{N_j}^j - x_0^j\|^2 + \frac{3 \ell^3 \eta^2 d^2 \mu^2}{4 b}  + \ell \eta^2 \|\hat{e}_j\|^2 \\
    = & \frac{b/d}{B} \E_{N_j}\left( f_\mu (x_0^j) - f_\mu(x_{N_j}^j) \right) + \frac{3d \ell^3 \eta^2}{2b} \E_j \E_{N_j} \|x_{N_j}^j - x_0^j\|^2 + \frac{3 \ell^3 \eta^2 d^2 \mu^2}{4 b}  + \ell \eta^2 \|\hat{e}_j\|^2
\end{align*}
then we have 
\begin{align*}
    & \eta (1-\ell \eta) B \E \|\nabla f_\mu (\tilde{x}_j)\|^2 + \eta B \E\<\hat{e}_j, \nabla f_\mu (\tilde{x}_j)\> \\
    \le & b/d \E \left( f_\mu (\tilde{x}_{j-1}) - f_\mu(\tilde{x}_j) \right) + \frac{3d \ell^3 \eta^2 B}{2b}  \E\|\tilde{x}_j - \tilde{x}_{j-1}\|^2 + \frac{3 \ell^3 \eta^2 d^2 \mu^2 B}{4 b}  + \ell \eta^2 B \E\|\hat{e}_j\|^2
    \end{align*}

\end{proof}

\begin{lemma}
\label{lemma: ZO-SCSG-rand-2}
Suppose $3 d^2 \ell^2 \eta^2 B< b^2$, then under Assumption~\ref{assum: basic},
\begin{align*} 
    & \left( \frac{b/d}{B} - \frac{3d \ell^2 \eta^2}{b}\right) \E \|\tilde{x}_j - \tilde{x}_{j-1}\|^2 + 2\eta \E \<\hat{e}_j, \tilde{x}_j - \tilde{x}_{j-1}\> \\
    \le & -2 \eta \E \<\nabla f_\mu (\tilde{x}_j), \tilde{x}_j - \tilde{x}_{j-1}\> + \frac{3 \ell^2 \eta^2 d^2 \mu^2}{2 b}  + 2 \eta^2\E\|\nabla f_\mu (\tilde{x}_j)\|^2 + 2 \eta^2 \E \|\hat{e}_j\|^2 
\end{align*}
\end{lemma}

\begin{proof}
Since $x_{k+1}^j = x_k^j - \eta v_k^j$, we have 
\begin{align*}
    & \E_u \E_{\mathcal{I}_k^j} \|x_{k+1}^j - x_0^j\|^2 \\
    = & \|x_k^j - x_0^j\|^2 - 2 \eta \<\E_u \E_{\mathcal{I}_k^j} v_k^j, x_k^j - x_0^j\> + \eta^2 \E_u \E_{\mathcal{I}_k^j} \|v_k^j\|^2 \\
    = & \|x_k^j - x_0^j\|^2 - 2 \eta \<\nabla f_\mu (x_k^j), x_k^j - x_0^j\>  - 2\eta \<\hat{e}_j, x_k^j - x_0^j\> + \eta^2 \E_u \E_{\mathcal{I}_k^j} \|v_k^j\|^2 \\
    \le & \left(1 + \frac{3d \ell^2 \eta^2}{b}\right)\|x_k^j - x_0^j\|^2 - 2 \eta \<\nabla f_\mu (x_k^j), x_k^j - x_0^j\>  - 2\eta \<\hat{e}_j, x_k^j - x_0^j\> +  \frac{3 \ell^2 \eta^2 d^2 \mu^2}{2 b} \\
    & + 2 \eta^2\|\nabla f_\mu (x_k^j)\|^2 + 2 \eta^2 \|\hat{e}_j\|^2 
\end{align*}
Using the notation $\E_j$ we have
\begin{align*}
    & 2 \eta \E_j \<\nabla f_\mu (x_k^j), x_k^j - x_0^j\>  + 2\eta \E_j \<\hat{e}_j, x_k^j - x_0^j\> \\
    \le & \left(1 + \frac{3d \ell^2 \eta^2}{b}\right) \E_j \|x_k^j - x_0^j\|^2 - \E_j \|x_{k+1}^j - x_0^j\|^2 + \frac{3 \ell^2 \eta^2 d^2 \mu^2}{2 b} + 2 \eta^2\|\nabla f_\mu (x_k^j)\|^2 + 2 \eta^2 \|\hat{e}_j\|^2 
\end{align*}
Let $k=N_j$, by taking expectation to $N_j$ and using Fubini's theorem, we have 
\begin{align*}
    & 2 \eta \E_{N_j} \E_j \<\nabla f_\mu (x_{N_j}^j), x_{N_j}^j - x_0^j\>  + 2\eta \E_{N_j} \E_j \<\hat{e}_j, x_{N_j}^j - x_0^j\> \\
    \le & \left(1 + \frac{3d \ell^2 \eta^2}{b}\right)\E_{N_j} \E_j\|x_{N_j}^j - x_0^j\|^2 - \E_{N_j} \E_j \|x_{{N_j}+1}^j - x_0^j\|^2 + \frac{3 \ell^2 \eta^2 d^2 \mu^2}{2 b} \\
    & + 2 \eta^2\E_{N_j}\|\nabla f_\mu (x_{N_j}^j)\|^2 + 2 \eta^2 \|\hat{e}_j\|^2 \\
    = & \left( - \frac{b/d}{B} + \frac{3d \ell^2 \eta^2}{b}\right)\E_{N_j} \E_j\|x_{N_j}^j - x_0^j\|^2 + \frac{3 \ell^2 \eta^2 d^2 \mu^2}{2 b}  + 2 \eta^2\E_{N_j}\|\nabla f_\mu (x_{N_j}^j)\|^2 + 2 \eta^2 \|\hat{e}_j\|^2 
\end{align*}
Substituting $x_{N_j}^j, x_0^j$ by $\tilde{x}_{j}, \tilde{x}_{j-1}$ and take a further expectation to the past randomness, we get
\begin{align*} 
    & \left( \frac{b/d}{B} - \frac{3d \ell^2 \eta^2}{b}\right) \E \|\tilde{x}_j - \tilde{x}_{j-1}\|^2 + 2\eta \E \<\hat{e}_j, \tilde{x}_j - \tilde{x}_{j-1}\> \\
    \le & -2 \eta \E \<\nabla f_\mu (\tilde{x}_j), \tilde{x}_j - \tilde{x}_{j-1}\> + \frac{3 \ell^2 \eta^2 d^2 \mu^2}{2 b}  + 2 \eta^2\E\|\nabla f_\mu (\tilde{x}_j)\|^2 + 2 \eta^2 \E \|\hat{e}_j\|^2 
\end{align*}

\end{proof}

\begin{lemma}
\label{lemma: ZO-SCSG-rand-3}
\begin{equation*}
    \frac{b/d}{B} \E \<\hat{e}_j, \tilde{x}_j - \tilde{x}_{j-1}\> = -\eta \E\<\hat{e}_j, \nabla f_\mu(\tilde{x}_j)\> - \eta \E \|\hat{e}_j\|^2
\end{equation*}
\end{lemma}

\begin{proof}
Let $M_k^j = \<\hat{e}_j, x_k^j - x_0^j\>$. Then we have 
\begin{equation*}
    \E_{N_j} \<\hat{e}_j, \tilde{x}_j - \tilde{x}_{j-1}\> = \E_{N_j} M_{N_j}^j
\end{equation*}
Since $N_j$ is independent of $x_0^j, \hat{e}_j$, we have 
\begin{equation*}
    \E \<\hat{e}_j, \tilde{x}_j - \tilde{x}_{j-1}\> = \E M_{N_j}^j
\end{equation*}
Also we have $M_0^j = 0$. On the other hand,
\begin{align*}
    \E_u \E_{\mathcal{I}_k^j} \left(M_{k+1}^j - M_k^j\right) = & \E_u \E_{\mathcal{I}_k^j} \<\hat{e}_j, x_{k+1}^j - x_k^j\> = -\eta \<\hat{e}_j, \E_u \E_{\mathcal{I}_k^j} v_k^j\> \\
    = & -\eta \<\hat{e}_j, \nabla f_\mu(x_k^j)\> - \eta \|\hat{e}_j\|^2
\end{align*}
Using the notation $E_j$, we have 
\begin{equation*}
    \E_j \left(M_{k+1}^j - M_k^j\right) =  -\eta \<\hat{e}_j, \E_j\nabla f_\mu(x_k^j)\> - \eta \|\hat{e}_j\|^2
\end{equation*}
Let $k = N_j$, by taking the expectation with respect to $N_j$ and using Fubini's theorem, we have 
\begin{equation*}
    \frac{b/d}{B} \E_{N_j} M_{N_j}^j = -\eta \<\hat{e}_j, \E_{N_j}\E_j\nabla f_\mu(x_{N_j}^j)\> - \eta \|\hat{e}_j\|^2
\end{equation*}
Substituting $x_{N_j}^j, x_0^j$ by $\tilde{x}_{j}, \tilde{x}_{j-1}$ and take a further expectation to the past randomness, we get
\begin{equation*}
    \frac{b/d}{B} \E \<\hat{e}_j, \tilde{x}_j - \tilde{x}_{j-1}\> = -\eta \E\<\hat{e}_j, \nabla f_\mu(\tilde{x}_j)\> - \eta \E \|\hat{e}_j\|^2
\end{equation*}

\end{proof}

\begin{proof}[\textbf{Proof of Lemma~\ref{lemma: ZO-SCSG} (Option \uppercase\expandafter{\romannumeral2})}]

Multiplying Lemma~\ref{lemma: ZO-SCSG-rand-1} by 2, Lemma~\ref{lemma: ZO-SCSG-rand-2} by $\frac{b/d}{\eta}$ and summing them up, we have
\begin{align*}
    & 2\eta B(1 - \eta \ell - \frac{b/d}{B})\E \|\nabla f_\mu (\tilde{x}_j)\|^2 + \frac{b^3/d^2 - 3 \ell^2 \eta^2 b B - 3d \ell^3 \eta^3 B^2}{\eta b B} \E \|\tilde{x}_j - \tilde{x}_{j-1}\|^2 \\
    & + 2 \eta B \E\<\hat{e}_j, \nabla f_\mu (\tilde{x}_j)\> + 2b/d \<\hat{e}_j, \tilde{x}_j - \tilde{x}_{j-1}\> \\
    \le & -2 b/d \E \<\nabla f_\mu(\tilde{x}_j), \tilde{x}_j - \tilde{x}_{j-1}\> + 2 b/d \E \left(f_\mu (\tilde{x}_{j-1}) - f_\mu (\tilde{x}_j) \right) + \frac{3\ell^2 \eta^2 d^2 \mu^2}{2 b}\left(B \ell + \frac{b/d}{\eta}\right) \\
    & + (2 \ell \eta^2 B + 2\eta b/d) \E\|\hat{e}_j\|^2
\end{align*}
By Lemma~\ref{lemma: ZO-SCSG-rand-3},
\begin{equation*}
    2 \eta B \E\<\hat{e}_j, \nabla f_\mu (\tilde{x}_j)\> + 2 b/d \E \<\hat{e}_j, \tilde{x}_j - \tilde{x}_{j-1}\>  =  - 2 \eta B \E \|\hat{e}_j\|^2
\end{equation*}
So the above inequality can be simplified as 
\begin{align*}
    & 2\eta B(1 - \eta \ell - \frac{b/d}{B})\E \|\nabla f_\mu (\tilde{x}_j)\|^2 + \frac{b^3/d^2 - 3 \ell^2 \eta^2 b B - 3d \ell^3 \eta^3 B^2}{\eta b B} \E \|\tilde{x}_j - \tilde{x}_{j-1}\|^2 \\
    \le & -2 b/d \E \<\nabla f_\mu(\tilde{x}_j), \tilde{x}_j - \tilde{x}_{j-1}\> + 2 b/d \E \left(f_\mu (\tilde{x}_{j-1}) - f_\mu (\tilde{x}_j) \right) + \frac{3\ell^2 \eta^2 d^2 \mu^2}{2 b}\left(B \ell+ \frac{b/d}{\eta}\right) \\
    & + (2 \ell \eta^2 B + 2\eta b/d + 2\eta B) \E\|\hat{e}_j\|^2
\end{align*}
Using the fact that $2\<a,b\>\le \beta\|a\|^2 + \frac{1}{\beta}\|b\|^2$ for any $\beta>0$, we have
\begin{align*}
    & -2 b/d \E \<\nabla f_\mu(\tilde{x}_j), \tilde{x}_j - \tilde{x}_{j-1}\> \\
    \le & \frac{\eta b B}{b^3/d^2 - 3 \ell^2 \eta^2 b B - 3d \ell^3 \eta^3 B^2} (\frac{b}{d})^2 \E \|\nabla f_\mu(\tilde{x}_j)\|^2 + \frac{b^3/d^2 - 3 \ell^2 \eta^2 b B - 3d \ell^3 \eta^3 B^2}{\eta b B} \E \|\tilde{x}_j - \tilde{x}_{j-1}\|^2
\end{align*}
Then we conclude that 
\begin{align*}
    &\frac{\eta \ell B}{b/d} \left( 2- 2\eta \ell - 2\frac{b/d}{B} - \frac{b^3 /d^2}{b^3/d^2 - 3 \ell^2 \eta^2 b B - 3d \ell^3 \eta^3 B^2}\right) \E \|\nabla f_\mu(\tilde{x}_j)\|^2 \\
    \le & 2 \ell \E(f_\mu (\tilde{x}_{j-1}) - f_\mu (\tilde{x}_j)) + \frac{3\ell^2 \eta^2 d^2 \mu^2}{2 b}\left(\frac{B\ell^2}{b/d} + \frac{\ell}{\eta}\right) + \frac{2\eta \ell B}{b/d} \left( 1 + \eta \ell+ \frac{b/d}{B} \right) \E\|\hat{e}_j\|^2
\end{align*}
Let $\eta \ell = \gamma \left( \frac{b/d}{B} \right)^{\frac{2}{3}}$ and $b/d \ge 1, \frac{B}{b/d} \ge 8 \ge \frac{8}{b/d}$,
\begin{align*}
    b^3/d^2 - 3 \ell^2 \eta^2 b B - 3d \ell^3 \eta^3 B^2 = & b^3/d^2 \left( 1 - 3\gamma^2 (b/d)^{-1} (\frac{b/d}{B})^{\frac{1}{3}} - 3 \gamma^3 (b/d)^{-1} \right) \\
    \ge &  b^3/d^2 \left( 1 - \frac{3}{2}\gamma^2 - 3\gamma^3  \right)
\end{align*}
Then the above inequality can be simplified as 
\begin{align*}
    & \gamma \left(\frac{B}{b/d}\right)^{\frac{1}{3}} \left(2 - 2\gamma \left( \frac{b/d}{B} \right)^{\frac{2}{3}} - 2  \frac{b/d}{B} - \frac{1}{1 - \frac{3}{2}\gamma^2 - 3\gamma^3} \right) \E \|\nabla f_\mu (\tilde{x}_j)\|^2 \\
    \le & 2 \ell \E\left(f_\mu (\tilde{x}_{j-1}) - f_\mu (\tilde{x}_j) \right) + \frac{3\ell^2 d^2 \mu^2}{2b} \gamma \left( (\frac{b/d}{B})^{\frac{1}{3}} \gamma + (\frac{b/d}{B})^{\frac{2}{3}} \right) \\
    & + 2\gamma \left(\frac{B}{b/d}\right)^{\frac{1}{3}} (1+\gamma (\frac{b/d}{B})^{\frac{2}{3}} + \frac{b/d}{B}) \E \|\hat{e}_j\|^2
\end{align*}
Since $\frac{B}{b/d} \ge 8, \gamma \le \frac{1}{8}$, we have 
\begin{align*}
    2 - 2\gamma \left( \frac{b/d}{B} \right)^{\frac{2}{3}} - 2  \frac{b/d}{B} - \frac{1}{1 - \frac{3}{2}\gamma^2 - 3\gamma^3} \ge & 2 - \frac{\gamma}{2} - \frac{1}{4} - \frac{1}{1 - \frac{3}{2}\gamma^2 - 3\gamma^3} \ge 0.65 \\
    (\frac{b/d}{B})^{\frac{1}{3}} \gamma  + (\frac{b/d}{B})^{\frac{2}{3}}  \le & \frac{\gamma}{2}  + \frac{1}{4}  \le \frac{5}{16}  \\
    1+\gamma (\frac{b/d}{B})^{\frac{2}{3}} + \frac{b/d}{B} \le & 1 + \frac{\gamma}{4} + \frac{1}{8} \le \frac{37}{32}
\end{align*}
Thus we have 
\begin{align*}
    \left( \frac{B}{b/d} \right)^{\frac{1}{3}} \E\|\nabla f_\mu (\tilde{x}_j)\|^2  \le \frac{4 \ell}{\gamma} \E\left(f_\mu (\tilde{x}_{j-1}) - f_\mu (\tilde{x}_j) \right) + \frac{\ell^2 d^2 \mu^2 }{ b} + 4\left( \frac{B}{b/d} \right)^{\frac{1}{3}} \E\|\hat{e}_j\|^2
\end{align*}
Using Lemma~\ref{lemma: ZO-SCSG-Coord-4}, we have 
\begin{align*}
    & \E \|\hat{e}_j\|^2 \\
    = & \E \|v_j - \nabla f_\mu(x_0^j)\|^2 =  \E \|\hat{\nabla}_{coord} f_{\mathcal{I}_j} (x_0^j) \|^2 \\
    \le & 3\E \|\hat{\nabla}_{coord} f_{\mathcal{I}_j} (x_0^j) - \hat{\nabla}_{coord} f(x_0^j)\|^2 + 3\E \|\hat{\nabla}_{coord} f(x_0^j) - \nabla f(x)\|^2 + 3\E\|\nabla f(x_0^j) - \nabla f_\mu(x_0^j) \|^2 \\
    \le & \frac{9 (2\ell^2 d \mu^2 + \sigma^2)}{B} + 3 \ell^2 d \mu^2 + \frac{3 \ell^2 d^2 \mu^2}{4}
\end{align*}
Using Lemma~\ref{lemma: RandGradEst}, we have 
\begin{align*}
    & \E\left(f_\mu (\tilde{x}_{j-1}) - f_\mu (\tilde{x}_j) \right) \le \E \left( f(\tilde{x}_{j-1}) - f(\tilde{x}_j) \right) + \ell \mu^2 \\
    & \E \|\nabla f_\mu (\tilde{x}_j)\|^2 \ge \frac{1}{2} \E\|\nabla f(\tilde{x}_j)\|^2 - \frac{1}{2} \E \|\nabla f(\tilde{x}_j) - \nabla f_\mu (\tilde{x}_j)\|^2 \ge \frac{1}{2} \E\|\nabla f(\tilde{x}_j)\|^2 - \frac{\ell^2 d^2 \mu^2}{8}
\end{align*}
Thus we obtain
\begin{align*}
    & \left( \frac{B}{b/d} \right)^{\frac{1}{3}} \left( \frac{1}{2} \E\|\nabla f(\tilde{x}_j)\|^2 - \frac{\ell^2 d^2 \mu^2}{8} \right) \\
    \le & \frac{4\ell }{\gamma} \left( \E \left( f(\tilde{x}_{j-1}) - f(\tilde{x}_j) \right) + \ell \mu^2 \right) + \frac{\ell^2 d^2 \mu^2 }{ b} + 4\left( \frac{B}{b/d} \right)^{\frac{1}{3}} \left( \frac{9 (2\ell^2 d \mu^2 + \sigma^2)}{B} + 3 \ell^2 d \mu^2 + \frac{3 \ell^2 d^2 \mu^2}{4} \right) \\
    \le & \frac{4\ell }{\gamma} \left( \E \left( f(\tilde{x}_{j-1}) - f(\tilde{x}_j) \right) + \ell \mu^2 \right) + \frac{\ell^2 d^2 \mu^2 }{ b} + 4\left( \frac{B}{b/d} \right)^{\frac{1}{3}} \left( \frac{9 (2\ell^2 d \mu^2 + \sigma^2)}{B} + 4 \ell^2 d^2 \mu^2  \right)
\end{align*}
Finally we get
\begin{align*}
    \left( \frac{B}{b/d} \right)^{\frac{1}{3}} \E\|\nabla f(\tilde{x}_j)\|^2 \le \frac{8\ell}{\gamma} \E \left( f(\tilde{x}_{j-1}) - f(\tilde{x}_j) \right) + \frac{72 \sigma^2}{(b/d)^{\frac{1}{3}} B^{\frac{2}{3}}} + c \left( \frac{B}{b/d} \right)^{\frac{1}{3}} \ell^2 d^2 \mu^2
\end{align*}
where $c$ is a sufficient large constant. Telescope the sum in $j=1, \dots, T$, and using the definition of $\tilde{x}_T^*$, we finally get
\begin{equation*}
    \E \|\tilde{x}_T^*\|^2 \le \frac{\frac{8\ell}{\gamma}}{T \left( \frac{B}{b/d}\right)^{\frac{1}{3}}} \cdot \E\left( f(\tilde{x}_0 ) - f(\tilde{x}_T) \right) + \frac{72 \sigma^2}{B} + c \ell^2 d^2 \mu^2
\end{equation*}

\end{proof}

    

\subsection{Proof of Second-Order Stationary Point (\texorpdfstring{\textbf{Option \uppercase\expandafter{\romannumeral2}}}{Lg})}

\begin{proof}[\textbf{Proof of Theorem~\ref{thm: ZO-SCSG-NCF}}]
Let $N_1$ and $N_2$ be the number of times we reach Line 7 and 9 of Algorithm~\ref{alg: ZO-SCSG-NCF}. From Lemma~\ref{lemma: ZO-SCSG} of ZO-SCSG we know that for one epoch with size $B = \max\{1, \frac{1152 \sigma^2}{\epsilon^2}\} $, mini-batch size $b\ge 1$ and the smoothing parameter $\mu = \frac{\epsilon}{4\sqrt{c}\ell d}$, we have 
\begin{equation*}
    \E \|x_{t+1}\|^2 \le \frac{8 \ell }{\gamma} \left(\frac{b/d}{B}\right)^{\frac{1}{3}} \E \left( f(x_t) - f(x_{t+1}) \right) + \frac{\epsilon^2}{8}
\end{equation*}
Then, if $\|\nabla f(x_{t+1})\| \ge 
\frac{\epsilon}{2}$, we have $x_{t+1} = x_{t+1}$; if $v = \bot$, we set $x_{t+1} = x_{t+1}$ for  if $v \neq \bot$, we have $f(x_{t+1}) - \E f(x_{t+1}) \ge \frac{\delta^3}{12 \rho^2}$ (here the expectation is taken on the randomness of sign of $v$). Thus we have 
\begin{equation*}
    \frac{\gamma B^{\frac{1}{3}}}{ 8 \ell (b/d)^{\frac{1}{3}} } \E \left[\sum_{t=0}^{K-1} \left( \|\nabla f(x_{t+1} )\|^2 - \frac{\epsilon^2}{8} \right) \right] + \frac{\delta^3 }{12 \rho^2 } \E [N_2] \le \Delta_f
\end{equation*}
On one hand, since we have chosen $K$ such that $K \ge \Omega\left( \frac{\ell (b/d)^{\frac{1}{3}} \Delta_f}{\epsilon^2 B^{\frac{1}{3}}}\right) = \Omega \left( \frac{\ell (b/d)^{\frac{1}{3}} \Delta_f}{\epsilon^2 (1+\frac{\sigma^2}{\epsilon^2})^{\frac{1}{3}}} \right)$, then by Markov's inequality, with probability at least $\frac{5}{6}$, it satisfies $\sum_{t=0}^{K-1} \|\nabla f(x_{t+1})\|^2 \le \frac{\epsilon}{4} K$. As a sequence, at least half of the indices $t=0, \dots, K-1$ will satisfy $\|\nabla f(x_{t+1})\| \le \frac{\epsilon}{2}$, which means that $N_1 \ge \frac{K}{2}$.

On the other hand, we have $\frac{\delta^3}{12 \rho^2} \E[N_2] \le \Delta_f + \frac{K \gamma B^{\frac{1}{3}} \epsilon^2}{64 \ell (b/d)^{\frac{1}{3}}}$. Since $K \ge \Omega\left( \frac{\ell (b/d)^{\frac{1}{3}} \Delta_f}{\epsilon^2 B^{\frac{1}{3}}}\right) = \Omega \left( \frac{\ell (b/d)^{\frac{1}{3}} \Delta_f}{\epsilon^2 (1+\frac{\sigma^2}{\epsilon^2})^{\frac{1}{3}}} \right)$, we have $ \E[N_2] \le \frac{K \gamma B^{\frac{1}{3}} \epsilon^2 \rho^2}{\ell \delta^3 (b/d)^{\frac{1}{3}}}$. As long as $B \le \mathcal{O}\left( \frac{\ell^3 \delta^9 b/d}{\epsilon^6 \rho^6} \right)$, or equivalently $b/d \ge \Omega\left( \frac{B\epsilon^6 \rho^6}{\delta^9 \ell^3} \right)$, we have $\E[N_2] \le \frac{K}{12}$. Therefore, with provability at least $\frac{5}{6}$, it satisfies $N_2 \le 
\frac{k}{2}$.

Since $N_1 \ge N_2$, this means with probability at least $\frac{2}{3}$ the algorithm must terminate and output some $x_{t+1}$ in an iteration.

Finally, the per-iteration complexity of Algorithm~\ref{alg: ZO-SCSG-NCF} is dominated by $\tilde{\mathcal{O}}(d\cdot B + \frac{B}{b/d} \cdot b ) = \tilde{\mathcal{O}}(dB)$ function queries for \hyperref[alg: ZO-SCSG]{ZO-SCSG} and $\tilde{\mathcal{O}}(dB)$ function queries for estimating $\|\nabla f(x_{t+1})\|$, as well as $\tilde{\mathcal{O}}(\frac{\ell^2}{\delta^2})$ invoking Algorithm~\ref{alg: ZO-NCF-Online}. Thus the total function query complexity is 
\begin{align*}
    \tilde{\mathcal{O}}\left( K \left( d B+ d \frac{\ell^2}{\delta^2}  \right) \right) &= \tilde{\mathcal{O}}\left(  \left( \frac{\ell (b/d)^{\frac{1}{3}} \Delta_f }{\epsilon^2 B^{\frac{1}{3}}} \right) \left(d B+ d\frac{\ell^2}{\delta^2} \right) \right) \\
    &= \tilde{\mathcal{O}}\left( d \left( \frac{\ell \max\{1, \frac{B^{\frac{1}{3}}\epsilon^2 \rho^2}{\delta^3 \ell}\} \Delta_f }{\epsilon^2 B^{\frac{1}{3}}} \right) \left( B+ \frac{\ell^2}{\delta^2} \right) \right) \\
    &= \tilde{\mathcal{O}}\left( d \left( \frac{\ell (1 + \frac{B^{\frac{1}{3}}\epsilon^2 \rho^2}{\delta^3 \ell}) \Delta_f }{\epsilon^2 B^{\frac{1}{3}}} \right) \left( B+ \frac{\ell^2}{\delta^2} \right) \right) \\
    &= \tilde{\mathcal{O}}\left( d \left( \frac{\ell  \Delta_f }{\epsilon^2 B^{\frac{1}{3}}} + \frac{\rho^2 \Delta_f}{\delta^3} \right) \left( B+ \frac{\ell^2}{\delta^2} \right) \right) \\
    &= \tilde{\mathcal{O}}\left( d \left( \frac{\ell  \Delta_f }{\epsilon^2 \max\{1, \frac{\sigma^2}{\epsilon^2}\}^{\frac{1}{3}}} + \frac{\rho^2 \Delta_f}{\delta^3} \right) \left( \max\{1, \frac{\sigma^2}{\epsilon^2}\} + \frac{\ell^2}{\delta^2} \right) \right) \\
    &= \tilde{\mathcal{O}}\left( d \left( \frac{\ell  \Delta_f }{\epsilon^\frac{4}{3} \sigma^{\frac{2}{3}} } + \frac{\rho^2 \Delta_f}{\delta^3} \right) \left( \frac{\sigma^2 }{\epsilon^2} + \frac{\ell^2}{\delta^2} \right) + d \frac{\ell \Delta_f}{\epsilon^2} \frac{\ell^2}{\delta^2}  \right)
\end{align*}

\end{proof}

\section{Applying Zeroth-Order Negative Curvature Finding to ZO-SPIDER}

In this section, we apply ZO-NCF-Online to ZO-SPIDER to turn it into a local minima finding algorithm and propose ZO-SPIDER-NCF in Algorithm~\ref{alg: ZO-SPIDER-NCF}. As a by-product, we also propose a zeroth-order variant of the SPIDER method in Appendix G that  can converge to an $\epsilon$-approximate FOSP with high probability rather than expectation. Using the same technique as in SPIDER-SFO$^{+}$ \cite{fang2018spider}, that is, instead of moving in a large single step with size $\delta/\rho$ along the approximate negative curvature direction as in ZO-SGD-NCF and ZO-SCSG-NCF, we can split it into $\delta/(\rho \eta)$ equal length mini-steps with size $\eta$. As a result, we can maintain the SPIDER estimates and improve the so-called  non-improvable coupling term $\frac{1}{\delta^3 \epsilon^2}$ by a fact of $\delta$.

\begin{algorithm}[htb]
	\caption{ZO-SPIDER-NCF}
	\label{alg: ZO-SPIDER-NCF}
	\renewcommand{\algorithmicrequire}{\textbf{Input:}} 
	\renewcommand{\algorithmicensure}{\textbf{Output:}}
	\begin{algorithmic}[1]
		\Require Function $f$, starting point $x_0$, $\epsilon>0$ and $\delta>0$.
		\For{$j=0$ to $J$}
	    \State $w_1 \gets$ \hyperref[alg: ZO-NCF-Online]{ZO-NCF-Online} $(f, x_k, \delta,\frac{1}{16J})$
	    \State Randomly flip a sign, and set $w_2 = \pm \eta w_1$
	    \For{$k$ to $k+ \mathscr{K}$}
	        \If{k \text{mod} q = 0}
	        \State Sample $\mathcal{S}_1$ from $[n]$ without replacement, $v_k = \hat{\nabla }_{coord} f_{\mathcal{S}_1} (x_k)$
	        \Else
	        \State Sample $\mathcal{S}_2$ from $[n]$ with replacement, 
	        \State  $v_k = \hat{\nabla}_{coord} f_{\mathcal{S}_2}(x_k) - \hat{\nabla}_{coord} f_{\mathcal{S}_2}(x_{k-1}) + v_{k-1}$
	        \EndIf
	        \If{$w_1 \neq \bot$}
	        \State $x_{k+1} = x_k - w_2$
	        \Else 
	            \If{$\|v_k\| \le 2 \tilde{\epsilon}$}
	            \State \Return $x_k$
	            \EndIf
	            \State $x_{k+1} = x_k - \eta (v_k /\|v_k\|)$
	        \EndIf
	        
	    \EndFor
	    
		\EndFor
	\end{algorithmic}
\end{algorithm}

\begin{theorem}
\label{thm: ZO-SPIDER-NCF-full}
Under  Assumption~\ref{assum: basic-high-probability}, 
if we set $\tilde{\epsilon} =  10 \epsilon \log (128(K_0+1)) ,  |\mathcal{S}_1| = \frac{16\sigma^2}{\epsilon^2},  |\mathcal{S}_2| = \frac{16 \sigma}{\epsilon n_0},  \eta = \frac{\epsilon}{\ell n_0},  q = \frac{\sigma n_0}{\epsilon},
\mu = \left( \frac{\epsilon \cdot \tilde{\epsilon}}{8q^2 \rho^2 d}\right)^{\frac{1}{4}} = \tilde{O}(\frac{\epsilon}{d^{1/4}}),  \mathscr{K} = \frac{\delta \ell n_0}{\rho \epsilon}$,
where $n_0 \in [1, \frac{2 \sigma}{\epsilon}]$, then with probability at least $\frac{3}{4}$, Algorithm~\ref{alg: ZO-SPIDER-NCF} outputs $x_k$ with $j\le J = 8 \left(\left\lfloor \max \left( \frac{12 \rho^2 \Delta_f}{\delta^3}, \frac{4 \rho \Delta_f}{\delta \epsilon} \right) \right\rfloor + 1 \right), k\le K_0 = J  \mathscr{K}$ satisfying $\|\nabla f(x_k)\| \le 3 \tilde{\epsilon}, \lambda_{min} (\nabla^2 f(x_k)) \ge -2\delta$ with $\tilde{\epsilon} = 10 \epsilon \log(128 (K_0+1)) =\tilde{\mathcal{O}}(\epsilon)$. The total function query complexity is bounded by 
\begin{equation*}
    \tilde{\mathcal{O}} \left( d \left( \frac{\sigma \ell \Delta_f}{\epsilon^3} + \frac{\sigma \ell \rho \Delta_f }{\epsilon^2 \delta^2} + \frac{\ell^2 \rho \Delta_f}{\delta^3 \epsilon} + \frac{\ell^2 \rho^2 \Delta_f}{\delta^5} + \frac{\sigma^2}{\epsilon^2} + \frac{\sigma \delta \ell}{\rho \epsilon^2} + \frac{\ell^2 }{\delta^2} \right) \right).
\end{equation*}

\end{theorem}

\begin{remark}
We can boost the confidence of Theorem~\ref{thm: ZO-SPIDER-NCF} from $3/4$ to $1-p$ by running $\log 1/p$ copies of \hyperref[alg: ZO-SPIDER-NCF]{ZO-SPIDER-NCF}
\end{remark}

\begin{algorithm}[htb]
	\caption{ZO-SPIDER-Coord (For convergence rates in high probability)}
	\label{alg: ZO-SPIDER}
	\renewcommand{\algorithmicrequire}{\textbf{Input:}} 
	\renewcommand{\algorithmicensure}{\textbf{Output:}}
	\begin{algorithmic}[1]
		\Require function $f$, starting point $x_0$, $\epsilon>0$.
	    
	    \For{$k = 0$ to $K$}
	        \If{k \text{mod} q = 0}
	        \State Sample $\mathcal{S}_1$ from $[n]$, $v_k = \hat{\nabla }_{coord} f_{\mathcal{S}_1} (x^k)$
	        \Else
	        \State  Sample $\mathcal{S}_2$ from $[n]$ with replacement,
	        \State $v_k = \hat{\nabla}_{coord} f_{\mathcal{S}_2}(x^k) - \hat{\nabla}_{coord} f_{\mathcal{S}_2}(x^{k-1}) + v_{k-1}$
	        \EndIf

	        \If{$\|v_k\| \le 2 \tilde{\epsilon}$}
	        \State \Return $x_k$
	        \EndIf
	        \State $x_{k+1} = x_k - \eta (v_k /\|v_k\|)$

	    \EndFor
	    
        \Ensure $x_K$ \Comment{this line is not reached with high probability}	   
	\end{algorithmic}
\end{algorithm}

\begin{theorem}
\label{thm: ZO-SPIDER}
Under the settings of Algorithm~\ref{alg: ZO-SPIDER}, if we set
\begin{align*}
    \tilde{\epsilon} = & 10 \epsilon \log (4(K_0+1)/p) , \quad |\mathcal{S}_1| = \frac{16\sigma^2}{\epsilon^2}, \quad |\mathcal{S}_2| = \frac{16 \sigma}{\epsilon n_0}, \quad \eta = \frac{\epsilon}{\ell n_0}, \quad q = \frac{\sigma n_0}{\epsilon},\\
    \mu = & \left( \frac{\epsilon \cdot \tilde{\epsilon}}{8q^2 \rho^2 d}\right)^{\frac{1}{4}} = \tilde{O}(\frac{\epsilon}{d^{1/4}})
\end{align*}
where $n_0 \in [1, \frac{2 \sigma}{\epsilon}]$. Then under Assumption~\ref{assum: basic-high-probability}, with probability at least $1-p$, Algorithm~\ref{alg: ZO-SPIDER} terminates before $K_0 = \left \lfloor \frac{4 \ell n_0}{\epsilon^2} \right \rfloor + 2 $ iterations and outputs $x_{\mathcal{K}}$ satisfying 
\begin{equation*}
    \|v_{\mathcal{K}}\| \le 2 \tilde{\epsilon}, \quad \|\nabla f (x_{\mathcal{K}})\| \le 3 \tilde{\epsilon}.
\end{equation*}
The function query complexity is $\mathcal{O} \left( d \left( \frac{ 128 \ell \sigma \Delta_f }{\epsilon^3} + \frac{16 \sigma^2}{\epsilon^2} + \frac{32 \sigma}{\epsilon n_0} \right) \right)$.

\end{theorem}

\begin{lemma}[Proposition 2 in \cite{fang2018spider}]
\label{lemma: concentration inequality}
Let $\epsilon_{1:K}$ be a vector-valued martingale difference sequence with respect to $\mathcal{F}_k$, \emph{i.e.}, for each $k=1,\dots,K, \E[\epsilon_k | \mathcal{F}_{k-1}] =0$ and $\|\epsilon_k\|^2 \le B_k^2$. We have
\begin{equation*}
    \textrm{Pr} \left( \left\| \sum_{k=1}^K \epsilon_k \ge \lambda \right\| \right) \le \exp \left( -\frac{\lambda^2}{4\sum_{k=1}^K} B_k^2 \right),
\end{equation*}
where $\lambda$ is an arbitrary real positive value.

\end{lemma}

\subsection{Proof of high probability results for First-Order Stationary Point of Theorem~\ref{thm: ZO-SPIDER} }

Define $\mathcal{K}$ to be the time when Algorithm~\ref{alg: ZO-SPIDER} stops. We have $\mathcal{K} = 0$ if $\|v_0\|\le 2\epsilon$, and $\mathcal{K} = \inf\{k\ge 0: \|v_k\| \le 2\epsilon\} + 1$ if $\|v_0\| \ge 2\epsilon$.

\begin{lemma}
\label{lemma: ZO-SPIDER-Coord}
Define the event,
\begin{equation*}
    \mathcal{H}_{K_0} = \left( \|v_k - \nabla f(x_k)\|^2 \le \epsilon \cdot \tilde{\epsilon}, \forall k \le \min\{\mathcal{K}, K_0\} \right),
\end{equation*}
then $\mathcal{H}_{K_0}$ will happen with probability at least $1-p$.
\end{lemma}

\begin{proof}
When $k \ge \mathcal{K} $, the algorithm has already stopped. Define virtual update when $x_{k+1} = x_k$, and $v_k$ is generated by Line 3 and Line 6 in Algorithm~\ref{alg: ZO-SPIDER}. 

Define the event $\tilde{\mathcal{H}}_k = ( \|v_k - \nabla f(x_k)\|^2 \ge \epsilon \cdot \tilde{\epsilon}),  k \in [K_0]$. Then if we can prove that for any $k \in [K_0]$, the probability of  $\tilde{\mathcal{H}}_k$ occurring  is no more than $p/(K_0 + 1)$, \emph{i.e.}, $\textrm{Pr}(\tilde{\mathcal{H}}_k ) \le \frac{p}{K_0 +1}$. Then we have 
\begin{equation*}
    \textrm{Pr}(\mathcal{H}_{K_0}) \ge 1- \textrm{Pr}(\bigcup_{k=0}^{K_0} \tilde{\mathcal{H}}_k) \ge 1- \sum_{k=0}^{K_0} \textrm{Pr} (\tilde{\mathcal{H}}_k) \ge 1-p
\end{equation*}

Now, we prove that $\textrm{Pr}(\tilde{\mathcal{H}}_k ) \le \frac{p}{K_0 +1}, \forall k\in [K_0]$.

First, we have 
\begin{align*}
    \|v_k - \nabla f(x_k)\|^2 \le & 2 \|v_k - \hat{\nabla}_{coord} f(x_k)\|^2 + 2 \| \hat{\nabla}_{coord} f(x_k) - \nabla f(x_k) \|^2 \\
    \le & 2 \|v_k - \hat{\nabla}_{coord} f(x_k)\|^2 + \frac{\rho^2 d \mu^4}{18}
\end{align*}

Denote by $\xi_k$ the randomness in maintaining SPIDER $v_k$ at iteration $k$, and $\mathcal{F}_k = \sigma\{\xi_0,\dots,\xi_k\}$, where $\sigma(\cdot)$, where $\sigma\{\cdot\}$ denotes the sigma field. We know that $x_k$ and $v_k$ are measurable on $\mathcal{F}_{k-1}$.

\ding{182} Then given $\mathcal{F}_{k-1}$, if $k = \left \lfloor k/q \right \rfloor q$ , we define
\begin{equation*}
    \hat{\epsilon}_{k,i} = \frac{1}{|\mathcal{S}_1|} \left( \hat{\nabla}_{coord} f_{\mathcal{S}_1(i)}(x_k) - \hat{\nabla}_{coord} f(x_k) \right), \quad 
    \epsilon_{k,i} = \frac{1}{|\mathcal{S}_1|} \left( \nabla f_{\mathcal{S}_1(i)} (x_k) - \nabla f(x_k) \right)
\end{equation*}
where $\mathcal{S}_1(i)$ denotes the $i$-th random component function seleted at iteration $k$ and $1\le i\le |\mathcal{S}_1|$. We have
\begin{equation*}
    \E [\hat{\epsilon}_{k,i} | \mathcal{F}_{k-1}] = 0, \quad \E [\epsilon_{k,i} | \mathcal{F}_{k-1}] = 0, \quad \|\epsilon_{k,i}\| \overset{\textrm{Assumption~\ref{assum: basic-high-probability}}}{\le} \frac{\sigma}{|\mathcal{S}_1|}
\end{equation*}
and 
\begin{align*}
    &\|v_k - \hat{\nabla}_{coord} f(x_k)\|^2 \\
    = &\|\sum_{i=1}^{|\mathcal{S}_1|}  \hat{\epsilon}_{k,i}\|^2 \\
    \le & 2 \|\sum_{i=1}^{|\mathcal{S}_1|} \epsilon_{k,i}\|^2 + 2\|\sum_{i=1}^{|\mathcal{S}_1|} (\epsilon_{k,i}- \hat{\epsilon}_{k,i}) \|^2 \\
    = & 2 \|\sum_{i=1}^{|\mathcal{S}_1|} \epsilon_{k,i}\|^2 + 2 \|\hat{\nabla}_{coord} f_{\mathcal{S}_1} (x_k) - \hat{\nabla}_{coord} f(x_k) + \nabla f_{\mathcal{S}_1}(x_k) - \nabla f(x_k)\|^2 \\
    \le & 2 \|\sum_{i=1}^{|\mathcal{S}_1|} \epsilon_{k,i}\|^2 + 2 \left(2\|\hat{\nabla}_{coord} f_{\mathcal{S}_1} (x_k) - \nabla f_{\mathcal{S}_1 (x_k)}\|^2 + 2\|\hat{\nabla}_{coord} f(x_k) - \nabla f(x_k)\|^2 \right) \\
    \le & 2 \|\sum_{i=1}^{|\mathcal{S}_1|} \epsilon_{k,i}\|^2 + \frac{2}{9}\rho^2 d \mu^4
\end{align*}
Then we have
\begin{align*}
    \|v_k - \nabla f(x_k)\|^2 \le & 2 \|v_k - \hat{\nabla}_{coord} f(x_k)\|^2 + \frac{\rho^2 d \mu^4}{18} \\
    \le & 4\|\sum_{i=1}^{|\mathcal{S}_1|} \epsilon_{k,i}\|^2 + \frac{1}{2}\rho^2 d \mu^4
\end{align*}
Then we have
\begin{align*}
    &\textrm{Pr} \left( \|v_k - \nabla f(x_k)\|^2 \ge \epsilon \cdot \tilde{\epsilon} | \mathcal{F}_{k-1} \right)\\
    \le & \textrm{Pr} \left( 4 \|\sum_{i=1}^{|\mathcal{S}_1|} \epsilon_{k,i}\|^2 + \frac{1}{2}\rho^2 d \mu^4 \ge \epsilon \cdot \tilde{\epsilon} | \mathcal{F}_{k-1} \right) 
    =  \textrm{Pr} \left(  \|\sum_{i=1}^{|\mathcal{S}_1|} \epsilon_{k,i}\|^2  \ge \frac{\epsilon \cdot \tilde{\epsilon}- \frac{1}{2}\rho^2 d \mu^4}{4} | \mathcal{F}_{k-1} \right) \\
    \le & 4 \exp\left( -\frac{\frac{\epsilon \cdot \tilde{\epsilon}- \frac{1}{2}\rho^2 d \mu^4}{4}}{4 |\mathcal{S}_1| \frac{\sigma^2 }{|\mathcal{S}_1|^2}} \right) = 4 \exp\left( - \frac{\epsilon \cdot \tilde{\epsilon}- \frac{1}{2}\rho^2 d \mu^4}{16 \frac{\sigma^2}{|\mathcal{S}_1|}}\right)
\end{align*}

\ding{183} When $k \neq \left \lfloor k/q \right \rfloor q$, set $k_0 = \left \lfloor k/q \right \rfloor q$ and define 
\begin{align*}
    \epsilon_{j,i} &= \frac{1}{|\mathcal{S}_2|} \left( \nabla f_{\mathcal{S}_2 (i)} (x_j) - \nabla f_{\mathcal{S}_2(i)} (x_{j-1}) - \nabla f(x_j) + \nabla f(x_{j-1}) \right) \\
    \hat{\epsilon}_{j,i} &= \frac{1}{|\mathcal{S}_2|} \left( \hat{\nabla}_{coord} f_{\mathcal{S}_2 (i)} (x_j) - \hat{\nabla}_{coord} f_{\mathcal{S}_2(i)} (x_{j-1}) - \hat{\nabla}_{coord} f(x_j) + \hat{\nabla}_{coord} f(x_{j-1}) \right)
\end{align*}
where $\mathcal{S}_2(i)$ denotes the $i$-th random component function selected at iteration $k$ and $1 \le i\le |\mathcal{S}_2|, k_0 \le j \le k$. We have 
\begin{equation*}
    \E[\epsilon_{j,i} | \mathcal{F}_{j-1}] = 0, \quad \E[\hat{\epsilon}_{j,i} | \mathcal{F}_{j-1}] = 0
\end{equation*}
From the update rule if $k < \mathcal{K}$, we have $\|x_{k+1} - x_k\| = \|\eta v_k/\|v_k\|\| = \eta = \frac{\epsilon}{\ell n_0}$, if $k>\mathcal{K}$, we have $\|x_{k+1} - x_k\| = 0 \le \frac{\epsilon}{\ell n_0}$. We have 
\begin{align*}
    & \|\epsilon_{j,i}\|  \\
    \le & \frac{1}{|\mathcal{S}_2|} \left( \|\nabla f_{\mathcal{S}_2 (i)} (x_j)- \nabla f_{\mathcal{S}_2(i)} (x_{j-1}) \| + \|\nabla f(x_j) - \nabla f(x_j)\| \right) \\
    \le &  \frac{2\ell}{|\mathcal{S}_2|} \|x_j - x_{j-1}\| \le \frac{2\epsilon}{|\mathcal{S}_2| n_0}
\end{align*}
for all $k_0 < k \le k$ and $1\le i \le |\mathcal{S}_2|$, and
\begin{align*}
    &\|v_k - \hat{\nabla}_{coord} f(x_k)\|^2 \\
    = & \left\|\hat{\nabla}_{coord} f_{\mathcal{S}_2} (x_k) - \hat{\nabla}_{coord} f_{\mathcal{S}_2} (x_{k-1}) - \hat{\nabla}_{coord} f(x_k) - \hat{\nabla}_{coord} f(x_{k-1}) + (v_k - \hat{\nabla}_{coord} f(x_{k-1}))\right\|^2 \\
    = & \left\| \sum_{i=1}^{|\mathcal{S}_2|}\hat{\epsilon}_{k,i} + v_{k-1} - \hat{\nabla}_{coord} f(x_{k-1}) \right\|^2  \\
    = & \left\| \sum_{j=k_0+1}^k \sum_{i=1}^{|\mathcal{S}_2|}\hat{\epsilon}_{j,i} + v_{k_0} - \hat{\nabla}_{coord} f(x_{k_0})\right\|^2 = \left\|\sum_{j=k_0+1}^k \sum_{i=1}^{|\mathcal{S}_2|}\hat{\epsilon}_{j,i} + \sum_{i=1}^{|\mathcal{S}_1|} \hat{\epsilon}_{k_0, i}\right\|^2 \\
    \le & 2\left\| \sum_{j=k_0+1}^k \sum_{i=1}^{|\mathcal{S}_2|} \epsilon_{j,i} + \sum_{i=1}^{|\mathcal{S}_1|} \epsilon_{k_0, i} \right\|^2 + 2 \left\| \sum_{j=k_0+1}^k \sum_{i=1}^{|\mathcal{S}_2|}( \epsilon_{j,i} -\hat{\epsilon}_{j,i}) + \sum_{i=1}^{|\mathcal{S}_1|} (\epsilon_{k_0, i} - \hat{\epsilon}_{k_0, i}) \right\|^2 \\
    \le & 2\left\| \sum_{j=k_0+1}^k \sum_{i=1}^{|\mathcal{S}_2|} \epsilon_{j,i} + \sum_{i=1}^{|\mathcal{S}_1|} \epsilon_{k_0, i} \right\|^2 + 2 \left( 2 \left\| \sum_{j=k_0+1}^k \sum_{i=1}^{|\mathcal{S}_2|}( \epsilon_{j,i} -\hat{\epsilon}_{j,i}) \right\|^2 + 2 \left\| \sum_{i=1}^{|\mathcal{S}_1|} (\epsilon_{k_0, i} - \hat{\epsilon}_{k_0, i}) \right\|^2 \right) \\
    \le & 2\left\| \sum_{j=k_0+1}^k \sum_{i=1}^{|\mathcal{S}_2|} \epsilon_{j,i} + \sum_{i=1}^{|\mathcal{S}_1|} \epsilon_{k_0, i} \right\|^2 + 2 \left( 2 \left\| \sum_{j=k_0+1}^k \sum_{i=1}^{|\mathcal{S}_2|}( \epsilon_{j,i} -\hat{\epsilon}_{j,i}) \right\|^2 + \frac{2}{9} \rho^2 d \mu^4 \right)
\end{align*}
The second term can be bounded by
\begin{align*}
    & \left\| \sum_{j=k_0+1}^k \sum_{i=1}^{|\mathcal{S}_2|}( \epsilon_{j,i} -\hat{\epsilon}_{j,i}) \right\|^2 \\
    \le & (k-k_0) \sum_{j=k_0+1}^k \left\| \sum_{i=1}^{\mathcal{S}_2} ( \epsilon_{j,i} -\hat{\epsilon}_{j,i}) \right\|^2 \\
    = & (k-k_0) \sum_{j=k_0+1}^k 4 \left( \left\| \nabla f_{\mathcal{S}_2} (x_j) - \hat{\nabla}_{coord} f_{\mathcal{S}_2} (x_{j})\right\|^2 + \left\| \nabla f_{\mathcal{S}_2} (x_{j-1}) - \hat{\nabla}_{coord} f_{\mathcal{S}_2} (x_{j-1}) \right\|^2 \right. \\
    & \left. + \left\| \nabla f(x_j) - \hat{\nabla}_{coord} f(x_j) \right\|^2 + \left\| \nabla f(x_{j-1}) - \hat{\nabla}_{coord} f(x_{j-1}) \right\|^2 \right) \\
    \le & (k-k_0) \sum_{j=k_0+1}^k 4 \cdot  4 \frac{\rho^2 d \mu^4}{36} = (k-k_0)^2\frac{4}{9}\rho^2 d \mu^4
\end{align*}
Thus we have 
\begin{align*}
    &\|v_k - \nabla f(x_k)\|^2 \\
    \le & 2 \|v_k - \hat{\nabla}_{coord} f(x_k)\|^2 + \frac{\rho^2 d \mu^4}{18} \\
    \le & 4\left\| \sum_{j=k_0+1}^k \sum_{i=1}^{|\mathcal{S}_2|} \epsilon_{j,i} + \sum_{i=1}^{|\mathcal{S}_1|} \epsilon_{k_0, i} \right\|^2 + 4 \left( 2 (k-k_0)^2 \frac{4}{9} \rho^2 d \mu^4 + \frac{2}{9} \rho^2 d \mu^4 \right) + \frac{\rho^2 d \mu^4}{18} \\
    \le & 4 \left\| \sum_{j=k_0+1}^k \sum_{i=1}^{|\mathcal{S}_2|} \epsilon_{j,i} + \sum_{i=1}^{|\mathcal{S}_1|} \epsilon_{k_0, i} \right\|^2 + 4 (k-k_0)^2 \rho^2 d \mu^4
\end{align*}
Using Lemma~\ref{lemma: concentration inequality}, we have
\begin{align*}
    &\textrm{Pr} \left( \|v_k - \nabla f(x_k)\|^2 \ge \epsilon \cdot \tilde{\epsilon} \left|\right. \mathcal{F}_{k_0 - 1} \right) \\
    \le & \textrm{Pr} \left( 4 \left\| \sum_{j=k_0+1}^k \sum_{i=1}^{|\mathcal{S}_2|} \epsilon_{j,i} + \sum_{i=1}^{|\mathcal{S}_1|} \epsilon_{k_0, i} \right\|^2 + 4(k-k_0)^2 \rho^2 d \mu^4 \ge \epsilon \cdot \tilde{\epsilon} \left|\right. \mathcal{F}_{k_0-1} \right)  \\
    = & \textrm{Pr} \left( \left\| \sum_{j=k_0+1}^k \sum_{i=1}^{|\mathcal{S}_2|} \epsilon_{j,i} + \sum_{i=1}^{|\mathcal{S}_1|} \epsilon_{k_0, i} \right\|^2 \ge \frac{\epsilon \cdot \tilde{\epsilon}}{4} - (k-k_0)^2 \rho^2 d \mu^4 \left|\right. \mathcal{F}_{k_0 -1}\right) \\
    \le & 4 \exp \left( -\frac{\frac{\epsilon \cdot \tilde{\epsilon}}{4} - (k-k_0)^2 \rho^2 d \mu^4}{4 |\mathcal{S}_1| \frac{\sigma^2}{|\mathcal{S}_1|^2} + 4 |\mathcal{S}_2| (k-k_0) \frac{4\epsilon^2}{|\mathcal{S}_2|^2 n_0^2}}  \right) \\
    \le & 4\exp \left( -\frac{\frac{\epsilon \cdot \tilde{\epsilon}}{4} - q^2 \rho^2 d \mu^4}{4 \frac{\sigma^2}{|\mathcal{S}_1|} + 4 q \frac{4\epsilon^2}{|\mathcal{S}_2| n_0^2}} \right) = 4 \exp \left( -\frac{\epsilon \cdot \tilde{\epsilon} - 4q^2 \rho^2 d \mu^4}{ 16 \frac{\sigma^2}{|\mathcal{S}_1|} + 16 q \frac{4\epsilon^2}{|\mathcal{S}_2| n_0^2}} \right)
\end{align*}

If we set 
\begin{equation*}
    \tilde{\epsilon} = 10 \epsilon \log \frac{4(K_0+1) }{p} ,  |\mathcal{S}_1| = \frac{16\sigma^2}{\epsilon^2},  |\mathcal{S}_2| = \frac{16 \sigma}{\epsilon n_0}, q = \frac{\sigma n_0}{\epsilon}, \mu = \left( \frac{\epsilon \cdot \tilde{\epsilon}}{8q^2 \rho^2 d}\right)^{\frac{1}{4}} = \tilde{O}(\frac{\epsilon}{d^{1/4}})
\end{equation*}
We will get
\begin{itemize}
    \item When $k \neq \left \lfloor k/q \right \rfloor q$, 
    \begin{align*}
        \textrm{Pr} \left( \|v_k - \nabla f(x_k)\|^2 \ge \epsilon \cdot \tilde{\epsilon} \left|\right. \mathcal{F}_{k_0 - 1} \right) \le 4 \exp \left( -\frac{\epsilon \cdot \tilde{\epsilon} - 4q^2 \rho^2 d \mu^4}{ 16 \frac{\sigma^2}{|\mathcal{S}_1|} + 16 q \frac{4\epsilon^2}{|\mathcal{S}_2| n_0^2}} \right) \le \frac{p}{K_0 + 1}
    \end{align*}
    \item When $k = \left \lfloor k/q \right \rfloor q$,
    \begin{align*}
        \textrm{Pr} \left( \|v_k - \nabla f(x_k)\|^2 \ge \epsilon \cdot \tilde{\epsilon} \left|\right. \mathcal{F}_{k_0 - 1} \right) \le & 4 \exp\left( - \frac{\epsilon \cdot \tilde{\epsilon}- \frac{1}{2}\rho^2 d \mu^4}{16 \frac{\sigma^2}{|\mathcal{S}_1|}}\right)  \\
        \le & 4 \exp \left( -\frac{\epsilon \cdot \tilde{\epsilon} - 4q^2 \rho^2 d \mu^4}{ 16 \frac{\sigma^2}{|\mathcal{S}_1|} + 16 q \frac{4\epsilon^2}{|\mathcal{S}_2| n_0^2}} \right) \le \frac{p}{K_0 + 1}
    \end{align*}
\end{itemize}
This completes the whole proof.

\end{proof}

\begin{lemma}
\label{lemma: descent lemma of zo-spider}
    Under assumption~\ref{assum: basic-high-probability}, we have that on $\mathcal{H}_{K_0} \cap (\mathcal{K} > K_0)$, for all $0 \le k \le K_0$,
    \begin{equation*}
        f(x_{k+1}) - f(x_k) \le -\frac{\epsilon \cdot \tilde{\epsilon}}{4\ell n_0}.
    \end{equation*}
    and 
    \begin{equation*}
        f(x_{K_0 +1} - f(x_0)) \le - \frac{\epsilon \cdot \tilde{\epsilon}}{4\ell n_0}K_0 
    \end{equation*}
\end{lemma}
    
\begin{proof}
Let $\eta_k = \eta/\|v_k\|$ and since $f$ has $\ell$-Lipschitz continuous gradient, we have 
\begin{align}
\label{eq: descent}
    f(x_{k+1}) & \le  f(x_k) - \<\nabla f(x_k), x_{k+1} - x_k\> + \frac{\ell}{2} \|x_{k+1} - x_k\|^2 \notag\\
    & = f(x_k) - \eta_k \<\nabla f(x_k), v_k\> + \frac{\ell \eta_k^2}{2} \|v_k\|^2 \notag\\
    & = f(x_k) - \eta_k \<\nabla f(x_k) - v_k, v_k\> - \eta_k\|v_k\|^2 + \frac{\ell \eta_k^2}{2} \|v_k\|^2 \notag\\
    & \le f(x_k) + \frac{\eta_k}{2} \|v_k - \nabla f(x_k)\|^2- \eta_k \left( \frac{1}{2} - \frac{\eta_k \ell}{2} \right)\|v_k\|^2
\end{align}
where the last inequality uses the the Cauchy-Schwarz inequality.    Because we are on the event $\mathcal{H}_{K_0} \cap (\mathcal{K} > K_0)$, so $\mathcal{K}-1 \ge K_0$, then for all $0\le k \le K_0$, we have $\|v_k\|\ge 2 \tilde{\epsilon}$, thus 
\begin{equation*}
    \eta_k = \frac{\epsilon}{\ell n_0} \frac{1}{\|v_k\|} \le \frac{\epsilon}{\ell n_0} \frac{1}{2\tilde{\epsilon}} \le \frac{1}{2\ell n_0} \le \frac{1}{2\ell}
\end{equation*}
we have
\begin{equation*}
    \eta_k \left( \frac{1}{2} - \frac{\eta_k \ell}{2} \right)\|v_k\|^2 \ge \eta_k \frac{1}{4} \|v_k\|^2 = \frac{1}{4} \frac{\epsilon}{\ell n_0} \frac{1}{\|v_k\|} \|v_k\|^2 \ge \frac{\epsilon \cdot \tilde{\epsilon}}{2 \ell n_0}
\end{equation*}
and when $\mathcal{H}_{K_0}$ happens, we also have
\begin{equation*}
    \frac{\eta_k}{2} \|v_k - \nabla f(x_k)\|^2 \le \frac{\eta_k}{2} \epsilon \cdot \tilde{\epsilon} \overset{\textrm{$\eta_k \le \frac{1}{2 \ell n_0}$}}{\le} \frac{\epsilon \cdot \tilde{\epsilon}}{4\ell n_0}
\end{equation*}
Hence 
\begin{equation*}
    f(x_{k+1}) \le f(x_k) + \frac{\epsilon \cdot \tilde{\epsilon}}{4\ell n_0} - \frac{\epsilon \cdot \tilde{\epsilon}}{2 \ell n_0} = f(x_k) - \frac{\epsilon \cdot \tilde{\epsilon}}{4\ell n_0}
\end{equation*}
By telescoping the above the inequality, we have
\begin{equation*}
    f(x_{K_0 + 1}) - f(x_0) \le -\frac{\epsilon \cdot \tilde{\epsilon}}{4\ell n_0} K_0.
\end{equation*}

\end{proof}

\begin{proof}[\textbf{Proof of Theorem~\ref{thm: ZO-SPIDER}} ]

If $\mathcal{K} \le K_0 \cap \mathcal{H}_{K_0}$, we have 
$\|v_{\mathcal{K}}\| \le 2 \tilde{\epsilon}$. Because $\|v_{\mathcal{K}} - \nabla f(x_{\mathcal{K}})\| \le \sqrt{\epsilon \cdot \tilde{\epsilon}} \le \tilde{\epsilon}$ if $\mathcal{H}_{K_0}$ occurs, so $\|\nabla f(x_{\mathcal{K}})\| \le 2 \tilde{\epsilon}$.
    
If $\mathcal{K} > K_0 \cap \mathcal{H}_{K_0}$, we have
\begin{align*}
    -\Delta_f \le & f^* - f(x_0) \le f(x_{K_0 +1}) - f(x_0) \le -\frac{\epsilon \cdot \tilde{\epsilon}}{4\ell n_0} K_0 \le  -\frac{\epsilon \cdot \tilde{\epsilon}}{4\ell n_0} \left(\frac{4 \ell \Delta_f n_0}{\epsilon^2} + 1 \right) \\
    \le & -\frac{\epsilon \cdot \tilde{\epsilon}}{4\ell n_0} \left(\frac{4 \ell \Delta_f n_0}{\epsilon \cdot \tilde{\epsilon}} + 1 \right) = - \Delta_f - \frac{\epsilon \cdot \tilde{\epsilon}}{4\ell n_0}
\end{align*}
 which is contradict with the fact that $-\Delta_f > - \Delta_f - \frac{\epsilon \cdot \tilde{\epsilon}}{4\ell n_0}$. This means that when $\mathcal{H}_{K_0}$ happens, then the algorithm must terminate before $K_0$ iterations.
 
 Therefore, the total function query complexity can be bounded by:
 \begin{align*}
     d \left(\left \lceil K_0/q \right \rceil |\mathcal{S}_1| + K_0 |\mathcal{S}_2| \right) \le &  d \left( \left( K_0/q +1 \right) |\mathcal{S}_1| + K_0 |\mathcal{S}_2| \right) \\
     \stackrel{\textrm{\ding{172}}}{ \le} & d \left(|\mathcal{S}_1| + 2 K_0 |\mathcal{S}_2|  \right) \\
     \le & d \left( \frac{16\sigma^2}{\epsilon^2} + 2 \left( \frac{4\ell \Delta_f n_0}{\epsilon^2} + 1 \right) \frac{16 \sigma}{\epsilon n_0} \right) \\
     = & d \left( \frac{ 128 \ell \sigma \Delta_f }{\epsilon^3} + \frac{16 \sigma^2}{\epsilon^2} + \frac{32 \sigma}{\epsilon n_0} \right)
 \end{align*}
 where \ding{172} is because $|\mathcal{S}_1| = q|\mathcal{S}_2|$.
    
\end{proof}

\subsection{Proof of high probability results for Second-Order Stationary Point of  Theorem~\ref{thm: ZO-SPIDER-NCF} }

From Algorithm~\ref{alg: ZO-SPIDER-NCF}, we know that all randomness in iteration $k$ come from three parts: $\left.1\right)$ maintaining SPIDER $v_k$ in Line 5-10; $\left.2\right)$ to conduct the zeroth-order negative-curvature search in Line 2; choosing a random sign of $w_2$ in Line 3. Denote by $\xi_k^1; \xi_k^2; \xi_k^3$ the randomness of from the three parts, respectively. Let $\mathcal{F}_k$ be the filtration involving the full information of $x_{0:k}, v_{0:k}$, \emph{i.e.}, $\mathcal{F}_k = \sigma \{\xi_{0;k}^1, \xi_{0;k}^2, \xi_{0;k-1}^3 \}$. So the randomness in iteration $k$ given $\mathcal{F}_k$ only comes from $\xi_k^3$.

Denote the random index
\begin{equation*}
    \mathcal{I}_k = 
    \begin{cases}
    1 & \text{if} \quad w_1 = \bot \\
    2 & \text{if} \quad w_1 \neq \bot
    \end{cases}
\end{equation*}
then we know that $\mathcal{I}_k$ is measurable on $\mathcal{F}_{\left \lfloor k/\mathscr{K} \right \rfloor\mathscr{K}}$ and also on $\mathcal{F}_k$. When the event $(\mathcal{I}_k=1 \bigcap \|v_k\| \le 2\tilde{\epsilon})$ happens, then the algorithm will be stopped. In this case, we define a virtual update $x_{k+1} = x_k$ in Line 13 and Line 18.

Let $\mathcal{H}_k^1$ denotes the event that algorithm has not stopped before $k$, \emph{i.e.}, 
\begin{equation*}
    \mathcal{H}_k^1 = \bigcap_{i=1}^{k} \left( \left( \|v_i\| \ge 2\tilde{\epsilon} \bigcap \mathcal{I}_i =1 \right) \bigcup \mathcal{I}_i =2 \right)
\end{equation*}
Let $\mathcal{H}_{\mathscr{K}j}^2$ denotes the event that the Zeroth-Order Curvature Finding in iteration $\mathscr{K} j$ runs successfully.

Let $\mathcal{H}_k^3$ denotes the event that
\begin{equation*}
    \mathcal{H}_k^3 = \left( \bigcap_{i=0}^k (\|v_i - \nabla f(x_i)\| \le \epsilon \cdot \tilde{\epsilon}) \right) \bigcap \left( \bigcap_{j=0}^{\left\lfloor k/ \mathscr{K} \right \rfloor} \mathcal{H}_{\mathscr{K} j}^2 \right)
\end{equation*}
Then we have $\mathcal{H}_3^k \in \mathcal{F}_k$, and $\mathcal{H}_1^3 \supseteq \mathcal{H}_2^3 \supseteq \cdots \supseteq \mathcal{H}_k^3 $.

\begin{lemma}
\label{lemma: event-3}
With the setting of Theorem~\ref{thm: ZO-SPIDER-NCF}, and under the Assumption~\ref{assum: basic-high-probability}, we have 
\begin{equation*}
    \textrm{Pr} \left( \mathcal{H}_{K_0}^3 \right) \ge \frac{15}{16}
\end{equation*}
\end{lemma}

\begin{proof}
Denote the event $\tilde{H}_k = (\|v_k - \nabla f(x_k)\| \le \epsilon \cdot \tilde{\epsilon}), 0\le k \le K_0$. If we can prove that $\textrm{Pr}(\tilde{H}_k) \ge 1 - \frac{1}{32 (K_0 + 1)}$, we have $\textrm{Pr} \left( \bigcap_{i=0}^{K_0} \tilde{H}_i \right) \ge 1- \frac{1}{32}$. On the other hand, from Theorem~\ref{thm: ZO-NCF-Online}, we know that each time $\textrm{Pr} \left(\mathcal{H}_{\mathscr{K} j}^2  \right) \ge 1-\frac{1}{32 J}$, so $\textrm{Pr} \left( \bigcap_{j=0}^{\left\lfloor K_0/ \mathscr{K} \right \rfloor} \mathcal{H}_{\mathscr{K} j}^2 \right) \ge 1- \frac{1}{32}$. So we have $\textrm{Pr} ( \mathcal{H}_{K_0}^3) \ge \frac{15}{16}$.

Now, we prove $\textrm{Pr}(\tilde{H}_k) \ge 1 - \frac{1}{32 (K_0 + 1)}$.

Consider the filtration of full information of $x_{0:k}$, $\mathcal{F}_k^2 = \sigma \{\xi_{0;k-1}^1, \xi_{0;k}^2, \xi_{0;k-1}^3 \}$. We know $x_k$ is measurable on $\mathcal{F}_k^2$. Given $\mathcal{F}_k^2$, we have 
\begin{itemize}
    \item when $k = \left \lfloor k/q \right \rfloor q$, 
    \begin{equation*}
        \E_i \left[ \hat{\nabla}_{coord} f_i (x_k) - \hat{\nabla}_{coord} f(x_k) \left.\right| \mathcal{F}_k^2 \right] =0, \quad \E_i \left[ \nabla f_i(x_k) - \nabla f(x_k) \left.\right| \mathcal{F}_k^2 \right]=0
    \end{equation*}
    \item when $k \neq \left \lfloor k/q \right \rfloor q$,
    \begin{align*}
        & \E_i \left[ \hat{\nabla}_{coord} f_i (x_k) - \hat{\nabla}_{coord} f_i (x_{k-1}) + \left( \hat{\nabla}_{coord} f (x_k) - \hat{\nabla}_{coord} f (x_{k-1}) \right) \right] =  0 \\
        & \E_i \left[ \nabla f_i (x_k) - \nabla f_i (x_{k-1}) + \left( \nabla f (x_k) - \nabla f (x_{k-1}) \right) \right] =  0
    \end{align*}
\end{itemize}

Because $x_k$ is generated by one of the three ways:
\begin{itemize}
    \item when $w_1 = \bot$, we have $\|x_k - x_{k-1}\| = \|\eta (v_{k-1}/ \|v_{k-1}\|)\| = \eta = \frac{\epsilon}{\ell n_0}$;
    \item when $w_1 \neq \bot$, we have $\|x_k - x_{k-1}\| = \|\eta w_1\| = \eta = \frac{\epsilon}{\ell n_0}$;
    \item when Algorithm~\ref{alg: ZO-SPIDER-NCF} has already stopped, we have $\|x_{k+1} - x_k\|=0 \le \frac{\epsilon}{\ell n_0}$
\end{itemize}
So $v_k - \nabla f(x_k)$ is martingale. If we set $|\mathcal{S}_1|, |\mathcal{S}_1|, \eta, \mu$ and $\tilde{\epsilon} =  10 \epsilon \log(128 (K_0+1))$ as the same in Lemma~\ref{lemma: ZO-SPIDER-Coord}. Then using the same technique of  Lemma~\ref{lemma: ZO-SPIDER-Coord} with $p = \frac{1}{32}$, we have $\textrm{Pr}(\tilde{H}_k) \ge 1 - \frac{1}{16 (K_0 + 1)}, 0\le k \le K_0$.

\end{proof}

\begin{lemma}
\label{lemma: sceond-order-stationary-point-high-probability}
If $\left( \mathcal{H}_{K_0}^1 \right)^c \cap \mathcal{H}_{K_0}^3$ happens, Algorithm~\ref{alg: ZO-SPIDER-NCF} outputs $x_k$ satisfying
\begin{equation*}
    \|\nabla f(x_k)\| \le 3 \tilde{\epsilon}, \quad \lambda_{min} (\nabla^2 f(x_k)) \ge -2\delta,
\end{equation*}
before $K_0$ iterations.
\end{lemma}

\begin{proof}
If $\left( \mathcal{H}_{K_0}^1 \right)^c$ happens, then Algorithm~\ref{alg: ZO-SPIDER-NCF} has already stopped before iteration $K_0$ and output $x_k$ with $\|v_k\|\le 2\tilde{\epsilon}$. If $\mathcal{H}_{K_0}^3$ happens, we have
\begin{itemize}
    \item $\|\nabla f(x_k) - v_k\|\le \sqrt{\epsilon \cdot \tilde{\epsilon}} \le \tilde{\epsilon}$. So we have 
\begin{equation*}
    \|\nabla f(x_k)\| \le \|v_k\| + \|\nabla f(x_k) - v_k\| \le 3\tilde{\epsilon}.
\end{equation*}
    \item Each time  the Zeroth-Order Curvature Finding runs successfully, so from Theorem~\ref{thm: ZO-NCF-Online}, we have 
    \begin{equation*}
        \lambda_{min} (\nabla^2 f(x_{k_0})) \ge - \delta,
    \end{equation*}
    where $k_0 = \left \lfloor k /\mathscr{K}  \right\rfloor \mathscr{K}$. 
\end{itemize}
From Assumption~\ref{assum: basic-high-probability}, we have 
\begin{equation*}
    \|\nabla^2 f(x) - \nabla^2 f(y)\| \le \left\| \frac{1}{n} \sum_{i=1}^n \left( \nabla^2 f_i (x) - \nabla^2 f_i (y) \right) \right\| \le \frac{1}{n} \sum_{i=1}^n \|\nabla^2 f_i(x) - \nabla^2 f_i (y)\| \le \rho \|x-y\|.
\end{equation*}
This means that $f$ has $\rho$-Lipschitz Hessian. We have \begin{align*}
    \left\| \nabla^2 f(x_k) - \nabla^2 f(x_{k_0}) \right\| \le & \left\| \sum_{i=k_0}^{k-1} (\nabla^2 f(x_{i+1}) - \nabla^2 f(x_i)) \right\| \\
    \le & \sum_{i=k_0}^{k-1} \rho \|x_{i+1} - x_i\| \le \mathscr{K} \frac{\epsilon}{\ell n_0} \overset{\textrm{$\mathscr{K} = \frac{\delta \ell n_0}{\rho \epsilon}$}}{=} \delta
\end{align*}
Finally we get $\lambda_{min} (\nabla^2 f(x_k)) \ge -2\delta$.

\end{proof}

\begin{proof}[\textbf{Proof of Theorem~\ref{thm: ZO-SPIDER-NCF}}]
Denote $\mathcal{H}_k^4 = \mathcal{H}_k^1 \cap \mathcal{H}_k^3$. For all iteration $\mathcal{K}$ with $\textrm{mod}(\mathcal{K}, \mathscr{K}) = 0$, given $\mathcal{F}_{\mathcal{K}}$, we consider:

\ding{182} When $\mathcal{H}_k^4 \bigcap (\mathcal{I}_{\mathcal{K}}=2)$ happens. From Lemma~\ref{lemma: Lipschitz} and the fact that $f$ is $\rho$-Hessian Lipschitz, we have 
\begin{align*}
    & f(x_{\mathcal{K}+\mathscr{
    K}}) \\
    \le & f(x_{\mathcal{K}}) + \<\nabla f(x_{\mathcal{K}}), x_{\mathcal{K}+\mathscr{
    K}} - x_{\mathcal{K}} \> + \frac{1}{2} (x_{\mathcal{K}+\mathscr{
    K}} - x_{\mathcal{K}})^\T \nabla^2 f(x_{\mathcal{K}}) (x_{\mathcal{K}+\mathscr{
    K}} - x_{\mathcal{K}}) + \frac{\rho}{6} \|x_{\mathcal{K}+\mathscr{
    K}} - x_{\mathcal{K}}\|^3.
\end{align*}
From Theorem~\ref{thm: ZO-NCF-Online}, we have $w_1^\T \nabla^2 f(x_{\mathcal{K}}) w_1 \le -\frac{\delta}{2}$. Take expectation on the random number of the sign we have 
\begin{equation*}
    \E \<\nabla f(x_{\mathcal{K}}), x_{\mathcal{K}+\mathscr{
    K}} - x_{\mathcal{K}} \> = 0
\end{equation*}
Thus we have 
\begin{align*}
    f(x_{\mathcal{K}+\mathscr{
    K}}) \le & f(x_{\mathcal{K}}) + \frac{1}{2} (\eta \cdot \mathscr{K} w_1)^\T \nabla^2 f(x_{\mathcal{K}}) (\eta \cdot \mathscr{K} w_1) + \frac{\rho}{6} \|\eta \cdot \mathscr{K} w_1\|^3 \\
    \overset{\textrm{$\eta = \frac{\epsilon}{\ell n_0}, \mathscr{K} = \frac{\delta \ell n_0}{\rho \epsilon}$}}{=} & f(x_{\mathcal{K}}) - \frac{\delta^3}{4\rho^2} + \frac{\delta^3}{6 \rho^2} = f(x_{\mathcal{K}}) - \frac{\delta^3}{12 \rho^2}
\end{align*}
Then we analysis the difference of $\left(f(x_{\mathcal{K}} - f^*)\right) \mathbb{I}_{\mathcal{H}_{\mathcal{K}}^4}$, where $\mathbb{I}_{\mathcal{H}_{\mathcal{K}}^4}$ is the indication function of $\mathcal{H}_{\mathcal{K}}^4$, then we have
\begin{align}
\label{eq: descent-second}
    & \E \left[ \left( f(x_{\mathcal{K}+ \mathscr{K}} - f^*) \right)\mathbb{I}_{\mathcal{H}_{\mathcal{K}+ \mathscr{K}}^4} \left.\right| \mathcal{F}_{\mathcal{K}} -
    \left( f(x_{\mathcal{K}} - f^*) \right)\mathbb{I}_{\mathcal{H}_{\mathcal{K}}^4} \left.\right| \mathcal{F}_{\mathcal{K}} \right] \notag\\
    = &\E \left[ \left( f(x_{\mathcal{K}+ \mathscr{K}} - f^*) \right) \left(\mathbb{I}_{\mathcal{H}_{\mathcal{K}+ \mathscr{K}}^4} - \mathbb{I}_{\mathcal{H}_{\mathcal{K}}^4} \right)\left.\right| \mathcal{F}_{\mathcal{K}} \right] + \E \left[ \left( f(x_{\mathcal{K} + \mathscr{K}}) - f(x_{\mathcal{K}}) \right) \mathbb{I}_{\mathcal{H}_{\mathcal{K}}^4} \left.\right| \mathcal{F}_{\mathcal{K}} \right] \notag\\
    \le & \textrm{Pr}(\mathcal{H}_{\mathcal{K}}^4 \left.\right| \mathcal{F}_{\mathcal{K}}) \E \left[ \left( f(x_{\mathcal{K} + \mathscr{K}}) - f(x_{\mathcal{K}}) \right) \mathbb{I}_{\mathcal{H}_{\mathcal{K}}^4} \left.\right| \mathcal{F}_{\mathcal{K}} \cap \mathcal{H}_{\mathcal{K}}^4  \right] \notag\\
    \le & - \textrm{Pr}(\mathcal{H}_{\mathcal{K}}^4 \left.\right| \mathcal{F}_{\mathcal{K}}) \frac{\delta^3}{12 \rho^2}
\end{align}
where the last second inequality uses the fact that $\mathbb{I}_{\mathcal{H}_{\mathcal{K}+ \mathscr{K}}^4} - \mathbb{I}_{\mathcal{H}_{\mathcal{K}}^4} \le 0$. 

\ding{183} When $\mathcal{H}_k^4 \bigcap (\mathcal{I}_{\mathcal{K}}=1)$ happens, then for $\mathcal{K} \le k < \mathcal{K}+\mathscr{K}$, from Eq.\eqref{eq: descent}, we have
\begin{equation*}
    f(x_{k+1}) \le f(x_k) + \frac{\eta_k}{2} \|v_k - \nabla f(x_k)\|^2- \eta_k \left( \frac{1}{2} - \frac{\eta_k \ell}{2} \right)\|v_k\|^2
\end{equation*}
where $\eta_k = \eta/\|v_k\|$. If $\mathcal{H}_k^4$ happens we have $\|v_k\|\ge 2\tilde{\epsilon}$ and $\|v_k - \nabla f(x_k)\|\le \epsilon \cdot \tilde{\epsilon}$, then from the proof of Lemma~\ref{lemma: descent lemma of zo-spider}, we also have
\begin{equation*}
    f(x_{k+1}) \le f(x_k) - \frac{\epsilon \cdot \tilde{\epsilon}}{4 \ell n_0}
\end{equation*}
Taking expectation up to $\mathcal{F}_{\mathcal{K}}$, we have
\begin{equation*}
    \E \left[f(x_{k+1}) - f(x_k) \left.\right| \mathcal{F}_{\mathcal{K}}\cap \mathcal{H}_k^4 \right] \le - \frac{\epsilon \cdot \tilde{\epsilon}}{4 \ell n_0}
\end{equation*}
By analyzing the difference of $\left(f(x_k )-f^*\right) \mathbb{I}_{\mathcal{H}_{k}}^4$, we have
\begin{align*}
    & \E \left[ \left( f(x_{k+1} )- f^* \right)\mathbb{I}_{\mathcal{H}_{k+1}^4} \left.\right| \mathcal{F}_{\mathcal{K}} -
    \left( f(x_{k} )- f^* \right)\mathbb{I}_{\mathcal{H}_{k}^4} \left.\right| \mathcal{F}_{\mathcal{K}} \right] \\
    = &\E \left[ \left( f(x_{k+1} )- f^* \right) \left(\mathbb{I}_{\mathcal{H}_{k+1}^4} - \mathbb{I}_{\mathcal{H}_{k}^4} \right)\left.\right| \mathcal{F}_{\mathcal{K}} \right] + \E \left[ \left( f(x_{k+1}) )- f(x_{k} \right) \mathbb{I}_{\mathcal{H}_{k}^4} \left.\right| \mathcal{F}_{\mathcal{K}} \right] \\
    \le & \textrm{Pr}(\mathcal{H}_{k}^4 \left.\right| \mathcal{F}_{\mathcal{K}}) \E \left[ \left( f(x_{k+1}) - f(x_{k}) \right) \mathbb{I}_{\mathcal{H}_{k}^4} \left.\right| \mathcal{F}_{\mathcal{K}} \cap \mathcal{H}_{k}^4  \right] \\
    \le & - \textrm{Pr}(\mathcal{H}_{k}^4 \left.\right| \mathcal{F}_{\mathcal{K}}) \frac{\epsilon \cdot \tilde{\epsilon}}{4 \ell n_0}
\end{align*}
where the last second inequality uses the fact that $\mathbb{I}_{\mathcal{H}_{\mathcal{K}+ \mathscr{K}}^4} - \mathbb{I}_{\mathcal{H}_{\mathcal{K}}^4} \le 0$. By telescoping the above inequality with $k$ from $\mathcal{K}$ to $\mathcal{K} + \mathscr{K}-1$, we have
\begin{align}
\label{eq: descent-first}
    & \E \left[ \left( f(x_{\mathcal{K}+ \mathscr{K}} )- f^* \right)\mathbb{I}_{\mathcal{H}_{\mathcal{K}+ \mathscr{K}}^4} \left.\right| \mathcal{F}_{\mathcal{K}} -
    \left( f(x_{\mathcal{K}} )- f^* \right)\mathbb{I}_{\mathcal{H}_{\mathcal{K}}^4} \left.\right| \mathcal{F}_{\mathcal{K}} \right] \notag\\
    \le & - \frac{\epsilon \cdot \tilde{\epsilon}}{4 \ell n_0} \sum_{i=\mathcal{K}}^{\mathcal{K}+\mathscr{K}} \textrm{Pr}(\mathcal{H}_{i}^4 \left.\right| \mathcal{F}_{\mathcal{K}}) 
    \le  - \frac{\epsilon \cdot \tilde{\epsilon}}{4 \ell n_0} \mathscr{K} \textrm{Pr}(\mathcal{H}_{\mathcal{K} +\mathscr{K}}^4 \left.\right| \mathcal{F}_{\mathcal{K}}) \notag\\
    \overset{\textrm{$\mathscr{K} = \frac{\delta \ell n_0}{\rho \epsilon}$}}{=} & - \textrm{Pr}(\mathcal{H}_{\mathcal{K} +\mathscr{K}}^4 \left.\right| \mathcal{F}_{\mathcal{K}}) \frac{\delta \tilde{\epsilon}}{4 \rho}
\end{align}
where the second inequality uses the fact that $\mathcal{H}_i^4 \subseteq \mathcal{H}_{\mathcal{K}+\mathscr{K}}^4$, thus $\textrm{Pr}(\mathcal{H}_i^4 \left.\right| \mathcal{F}_{\mathcal{K}}) \ge \textrm{Pr}(\mathcal{H}_{\mathcal{K}+\mathscr{K}}^4 \left.\right| \mathcal{F}_{\mathcal{K}})$.

Combining \eqref{eq: descent-second} and \eqref{eq: descent-first} we have 
\begin{align*}
    & \E \left[ \left( f(x_{\mathcal{K}+ \mathscr{K}}) - f^* \right)\mathbb{I}_{\mathcal{H}_{\mathcal{K}+ \mathscr{K}}^4} \left.\right| \mathcal{F}_{\mathcal{K}} -
    \left( f(x_{\mathcal{K}}) - f^* \right)\mathbb{I}_{\mathcal{H}_{\mathcal{K}}^4} \left.\right| \mathcal{F}_{\mathcal{K}} \right] \\
    \le & \min \left(- \textrm{Pr}(\mathcal{H}_{\mathcal{K}}^4 \left.\right| \mathcal{F}_{\mathcal{K}}) \frac{\delta^3}{12 \rho^2}, - \textrm{Pr}(\mathcal{H}_{\mathcal{K} +\mathscr{K}}^4 \left.\right| \mathcal{F}_{\mathcal{K}}) \frac{\delta \tilde{\epsilon} }{4 \rho}\right) \\
    \le & -\textrm{Pr}(\mathcal{H}_{\mathcal{K} +\mathscr{K}}^4 \left.\right| \mathcal{F}_{\mathcal{K}}) \min\left( \frac{\delta^3}{12 \rho^2}, \frac{\delta \tilde{\epsilon} }{4 \rho} \right)
\end{align*}
By taking full expectation on the above inequality, and telescoping the results with $\mathcal{K} = 0, \mathscr{K}, \dots, (J-1)\mathscr{K}$, we have 
\begin{align*}
    & \E \left[ \left( f(x_{J \mathscr{K}}) - f^* \right)\mathbb{I}_{\mathcal{H}_{J \mathscr{K}}^4}  -
    \left( f(x_0) - f^* \right)\mathbb{I}_{\mathcal{H}_0^4}  \right] \\
    \le & - \min\left( \frac{\delta^3}{12 \rho^2}, \frac{\delta \tilde{\epsilon} }{4 \rho} \right) \sum_{j=0}^{J-1} \textrm{Pr}(\mathcal{H}_{j \mathscr{K}}^4 ) \\
    \le & - \min\left( \frac{\delta^3}{12 \rho^2}, \frac{\delta \tilde{\epsilon}}{4 \rho} \right) J \textrm{Pr}(\mathcal{H}_{J \mathscr{K}}^4 )
\end{align*}
Then we use the fact that $f(x_{J \mathscr{K}}) - f^* \ge 0, f(x_0) - f^* \le \Delta_f$ and choose $J = 8 \left(\left\lfloor \max \left( \frac{12 \rho^2 \Delta_f}{\delta^3}, \frac{4 \rho \Delta_f}{\delta \epsilon} \right) \right\rfloor + 1 \right) \ge \frac{8 \Delta_f}{\min\left( \frac{\delta^3}{12 \rho^2}, \frac{\delta \tilde{\epsilon}}{4 \rho} \right)}$, we have 
\begin{equation*}
    \textrm{Pr} \left( \mathcal{H}_{J \mathscr{K}}^4 \right) \overset{\textrm{$K_0 = J \mathscr{K}$}}{=} \textrm{Pr} \left( \mathcal{H}_{K_0}^4 \right) \le \frac{1}{8}
\end{equation*}
Using the union bound and Lemma~\ref{lemma: event-3}, we have 
\begin{align*}
    \textrm{Pr} \left( \mathcal{H}_{K_0}^1 \right) = & \textrm{Pr}\left( \mathcal{H}_{K_0}^1 \bigcap \mathcal{H}_{K_0}^3 \right) + \textrm{Pr}\left( \mathcal{H}_{K_0}^1 \bigcap \left( \mathcal{H}_{K_0}^3 \right)^c \right) = \textrm{Pr} \left( \mathcal{H}_{K_0}^4 \right) + \textrm{Pr}\left( \mathcal{H}_{K_0}^1 \bigcap \left( \mathcal{H}_{K_0}^3 \right)^c \right) \\
    \le & \textrm{Pr} \left( \mathcal{H}_{K_0}^4 \right) + \textrm{Pr} \left( \left( \mathcal{H}_{K_0}^3 \right)^c \right) = \textrm{Pr} \left( \mathcal{H}_{K_0}^4 \right) + 1 - \textrm{Pr} \left( \mathcal{H}_{K_0}^3  \right) \\
    \le & \frac{1}{8} +1 - \frac{15}{16} = \frac{3}{16}
\end{align*}
Then we have
\begin{align*}
    \textrm{Pr} \left( \left(\mathcal{H}_{K_0}^1\right)^c \bigcap \mathcal{H}_{K_0}^3 \right) =& 1 - \textrm{Pr} \left( \mathcal{H}_{K_0}^1 \bigcup \left( \mathcal{H}_{K_0}^3 \right)^c \right) \ge 1 - \textrm{Pr} \left(\mathcal{H}_{K_0}^1\right) - \textrm{Pr}\left( \left( \mathcal{H}_{K_0}^3 \right)^c \right) \\
    = & \textrm{Pr}\left( \mathcal{H}_{K_0}^3 \right) - \textrm{Pr} \left(\mathcal{H}_{K_0}^1\right) \ge \frac{15}{16} -\frac{3}{16} = \frac{3}{4}
\end{align*}
From Lemma~\ref{lemma: sceond-order-stationary-point-high-probability}, we have with probability at least $\frac{3}{4}$, Algorithm~\ref{alg: ZO-SPIDER-NCF} will be terminated and output $x_k$ before $K_0$ iterations satisfying 
\begin{equation*}
    \|\nabla f(x_k)\| \le 3\tilde{\epsilon}, \quad \lambda_{min} (\nabla^2 f(x_k)) \ge -2\delta,
\end{equation*}

Then we compute the total function query complexity:
\begin{itemize}
    \item On the one hand, with probability at least $3/4$, the algorithm stops with no more than $K_0$ iterations, thus the function query complexity of computing the deterministic coordinate-wise gradient in Line 6 and Line 9 of Algorithm~\ref{alg: ZO-SPIDER-NCF} can be bounded by 
    \begin{align*}
        d \left(\left \lceil K_0/q \right \rceil |\mathcal{S}_1| + K_0 |\mathcal{S}_2| \right) \le &  d \left( \left( K_0/q +1 \right) |\mathcal{S}_1| + K_0 |\mathcal{S}_2| \right) 
        \stackrel{\textrm{$|\mathcal{S}_1| = q|\mathcal{S}_2|$}}{ \le}  d \left(|\mathcal{S}_1| + 2 K_0 |\mathcal{S}_2|  \right) \\
        \le & d \left( |\mathcal{S}_1| +  \mathscr{K} \left( 8\left\lfloor \max \left( \frac{12 \rho^2 \Delta_f}{\delta^3}, \frac{4 \rho \Delta_f}{\delta \epsilon} \right) \right\rfloor + 8 \right)  \cdot 2 |\mathcal{S}_2| \right)
    \end{align*}
    \item On the other hand, with probability at least $3/4$, the algorithm stops with no more than $K_0$ iterations, thus there are at most $K_0/\mathscr{K} = J$ times of zeroth-order negative curvature search. The total function query complexity for zeroth-order negative curvature search can be bounded by 
    \begin{equation*}
        \tilde{O} \left( d \left( 8\left\lfloor \max \left( \frac{12 \rho^2 \Delta_f}{\delta^3}, \frac{4 \rho \Delta_f}{\delta \epsilon} \right) \right\rfloor + 8 \right) \frac{\ell^2}{\delta^2} \right)
    \end{equation*}
    where $\tilde{O}$ hides a polylogarithmic factor of $d$.
\end{itemize}

Then by substituting $\mathscr{K} = \frac{\delta \ell n_0}{\rho \epsilon}, |\mathcal{S}_1| = \frac{16\sigma^2}{\epsilon^2}, |\mathcal{S}_2| = \frac{16 \sigma}{\epsilon n_0} $ the total function query complexity can be bounded by 
\begin{align*}
    & \tilde{\mathcal{O}} \left(d \left( \frac{16 \sigma^2}{\epsilon^2} + \left( 2\frac{16 \sigma}{\epsilon n_0}  \frac{\delta \ell n_0}{\rho \epsilon} + \frac{\ell^2}{\delta^2} \right) \cdot \left( 8 \left( \frac{12 \rho^2 \Delta_f}{\delta^3} + \frac{4 \rho \Delta_f}{\delta \epsilon} \right) +8 \right) \right)\right) \\
    = & \tilde{\mathcal{O}} \left(d  \frac{\sigma^2}{\epsilon^2}  + d \left( \frac{\sigma \delta \ell}{\rho \epsilon^2} + \frac{\ell^2}{\delta^2}\right) \left( \frac{\rho^2 \Delta_f}{\delta^3} + \frac{\rho \Delta_f}{\delta \epsilon} + 1\right)   \right) \\
    = & \tilde{\mathcal{O}} \left( d \left( \frac{\sigma \ell \Delta_f}{\epsilon^3} + \frac{\sigma \ell \rho \Delta_f }{\epsilon^2 \delta^2} + \frac{\ell^2 \rho \Delta_f}{\delta^3 \epsilon} + \frac{\ell^2 \rho^2 \Delta_f}{\delta^5} + \frac{\sigma^2}{\epsilon^2} + \frac{\sigma \delta \ell}{\rho \epsilon^2} + \frac{\ell^2 }{\delta^2} \right) \right)
\end{align*}
above, $\tilde{\mathcal{O}}$ hides the polylogarithmic factor of $d$ and the constant factor.
\end{proof}

\section{Additional Experiments}
In this section, we conduct several experiments to verify the effectiveness of our methods for both deterministic setting and stochastic setting. Specifically, for the deterministic setting, we compare our ZO-GD-NCF against three ZO algorithms for escaping saddle points, which are ZPSGD, PAGD, and RSPI. For the stochastic setting, we compare the three algorithms proposed in the paper, which are ZO-SGD-NCF, ZO-SCSG-NCF, and ZO-SPIDER-NCF. We don't compare our methods against ZO-SCRN because each iteration of ZO-SCRN needs to solve a cubic minimization subproblem: $x_{k+1} = \argmin_{x\in \R^d} g^\T(x-x_k) + \frac{1}{2}(x-x_k)^\T H (x-x_k) + \frac{\alpha}{6} \|x-x_k\|^3$,
where $g$ and $H$ are inexact estimations of the full gradient $\nabla f(x_k)$ and $\nabla^2 f(x_k)$ by ZO oracle, respectively. Although many efficient inexact solvers of the cubic minimization subproblem have been proposed, most of them are second-order or first-order methods \cite{agarwal2017finding,carmon2016gradient,cartis2011adaptive1,cartis2011adaptive2}, which is out of the scope of this paper.

\subsection{Algorithms Description}

\begin{algorithm}[htp]
	\caption{Zero-th order Perturbed Stochastic Gradient Descent (ZPSGD)}\label{algo:PSGD}
	\begin{algorithmic}
		\renewcommand{\algorithmicrequire}{\textbf{Input: }}
		\renewcommand{\algorithmicensure}{\textbf{Output: }}
		\Require $x_0$, learning rate $\eta$, noise radius $r$, mini-batch size $m$.
		\For{$t = 0, 1, \ldots, T$}
		\State sample $(z^{(1)}_t, \cdots, z^{(m)}_t) \sim \mathcal{N}(0,\sigma^2 \I)$
		\State $g_t(x_t)  \leftarrow  \sum_{i=1}^m z^{(i)}_t[f(x_t+z^{(i)}_t)- f(x_t)]/(m\sigma^2)$
		\State $x_{t+1} \leftarrow x_t - \eta (g_t(x_t) + \xi_t), \qquad \xi_t \text{~uniformly~} \sim \mathbb{B}_0(r)$
		\EndFor
		\State \textbf{return} $x_T$
	\end{algorithmic}
\end{algorithm}

\begin{algorithm}[H]
 \renewcommand{\thealgorithm}{}
    \centering
    \caption{Initialization: $(\ell, \rho, \epsilon, c, \delta, \Delta_f)$}\label{algo:init}    
    \begin{algorithmic}[1]
        \State $\chi \leftarrow 3\max\{\log(\frac{d\ell\Delta_f}{c\epsilon^2\delta}), 4\}, 
            ~\eta \leftarrow \frac{c}{\ell},
            ~r \leftarrow \frac{\sqrt{c}}{\chi^2}\cdot\frac{\epsilon}{\ell}, 
            ~g_{\text{thres}} \leftarrow \frac{\sqrt{c}}{\chi^2}\cdot \epsilon,
            ~f_{\text{thres}} \leftarrow \frac{c}{\chi^3} \cdot \sqrt{\frac{\epsilon^3}{\rho}}$ 
        \State $~t_{\text{thres}} \leftarrow \frac{\chi}{c^2}\cdot\frac{\ell}{\sqrt{\rho \epsilon}},
            ~S \gets \frac{\sqrt{c}}{\chi}\frac{\sqrt{\rho \epsilon}}{\rho},
            ~h_{low} \leftarrow \frac{1}{c_h}\min\{g_{\text{thres}}, \frac{r\rho\delta S}{2\sqrt{d}} \}$
    \end{algorithmic}
\end{algorithm}
\vspace{-1cm}
\begin{center}
\begin{minipage}{0.49\textwidth}
\begin{algorithm}[H]
    \centering
    \caption{\text{PAGD}($\mathbf{x}_0$)}\label{algo:PAGD}
    \begin{algorithmic}[1]
    \For{$t = 0, 1, \ldots $} \label{alg:marker}
            \State $\mathbf{z}_t \leftarrow q(\mathbf{x}_t, \frac{g_{\text{thres}}}{4c_h})$
            \If{$\| \mathbf{z}_t \| \ge \frac{3}{4}g_{\text{thres}}$}\label{line:if} 
            \State $\mathbf{x}_{t+1} \leftarrow \mathbf{x}_t - \eta \mathbf{z}_t$ 
            \Else \label{line:else} \State  $\mathbf{x}_{t+1} \leftarrow $ EscapeSaddle ($\mathbf{x}_t $)
                \If{$\mathbf{x}_{t+1} = \mathbf{x}_{t}$} {\textbf{return } $\mathbf{x}_{t}$}
                \EndIf
            \EndIf
            \EndFor
    \end{algorithmic}
\end{algorithm}
\end{minipage}
\hfill
\begin{minipage}{0.49\textwidth}
\begin{algorithm}[H]
    \centering
    \caption{EscapeSaddle ($\hat{\mathbf{x}}$)}\label{algo:EscapeSaddle}
    \begin{algorithmic}[1]
        \State $\boldsymbol{\xi} \sim\operatorname{Unif}(B_\mathbf{0}(r))$
        \State $\tilde{\mathbf{x}}_0 \leftarrow \hat{\mathbf{x}} + \boldsymbol{\xi}$
        \For{$i = 0, 1, \ldots t_{\text{thres}} $}
            \If{ $f(\hat{\mathbf{x}})  - f(\tilde{\mathbf{x}}_i)  \geq f_{\text{thres}}$}
                \State \textbf{return }  $\tilde{\mathbf{x}}_i$ 
            \EndIf
            \State $\tilde{\mathbf{x}}_{i+1} \leftarrow \tilde{\mathbf{x}}_i - \eta q( \tilde{\mathbf{x}}_i, h_{low})$
        \EndFor
        \State \textbf{return } $\hat{\mathbf{x}}$ 
    \end{algorithmic}
\end{algorithm}
\end{minipage}
\end{center}

\begin{minipage}{0.46\textwidth}
\begin{algorithm}[H]
\caption{Random search PI (RSPI).}
\begin{algorithmic}[1]
\State {\bf{Parameters}} $\sigma_1,\sigma_2>0$
\State Initialize $x_0$ at random
\For{$k=0, 2, 4, \cdots 2K$}
\State $s_1 \sim S^{d-1}$ (uniformly)
\State $x_{k+1} = \arg \min \{ f(x_k), f(x_k + \sigma_1 s_1), f(x_k - \sigma_1 s_1) \}$
\State $s_2 = \text{DFPI}(x_k)$ 
\State $x_{k+2} = \arg \min \{ f(x_{k+1}), f(x_{k+1} + \sigma_2 s_2), f(x_{k+1} - \sigma_2 s_2) \}$
\State Optional: Update $\sigma_1$ and $\sigma_2$
\EndFor
\end{algorithmic}
\label{algo:random_search_PI}
\end{algorithm}
\end{minipage}
\hfill
\begin{minipage}{0.53\textwidth}
\begin{algorithm}[H] 
\begin{algorithmic}[1]
\State {{\bf{Parameters}} $c,r,\eta>0 \text{ and } T_\text{DFPI} \in \mathbb{Z}^+$}
\State {\bf{Inputs} : $x \in \R^d$, }
\State $s_2^{(0)} \sim S^{d-1}$ (uniformly) 
\For{$t = 0 \dots T_\text{DFPI}-1$}
\State $g_+ = \sum\limits_{i=1}^d \frac{f(x + r\cdot s^{(t)}_2 + c \cdot  e_i) - f(x + r\cdot s^{(t)}_2 - c \cdot  e_i)}{2c} e_i$ 
\State  $g_- = \sum\limits_{i=1}^d \frac{f(x - r\cdot s^{(t)}_2 + c \cdot e_i) - f( x - r\cdot s^{(t)}_2 - c \cdot  e_i)}{2c} e_i$ 
\State \textbf{Update}: $s^{(t+1)}_2 = s^{(t)}_2 - \eta \frac{g_+ - g_-}{2r}$ 
\State Normalize $s^{(t+1)}_2 = s^{(t+1)}_2/\| s^{(t+1)}_2\|$
\EndFor
\State {\bf Return : $s^{(T_{\textmd{DFPI}})}_2$ }
\end{algorithmic}
\caption{Derivative-Free Power Iteration (DFPI)}
\label{alg:inexact_PI}
\end{algorithm}
\end{minipage}

\subsection{Parameter Settings of the Octopus Function Experiment}

The detailed parameter settings of the octopus function experiment are stated in the following table.

\begin{table}[htb]
\caption{Choices of parameters for the experiment of the octopus function.}
\begin{center}
\begin{tabular}{cc}
\hline
\textbf{Algorithm} & \multicolumn{1}{c}{\textbf{Parameters}} \\ 
\hline
\multicolumn{2}{c}{\textbf{d = 10, 30, 50, 100}}\\ 
\hline
ZPSGD     & $\ell= e, \rho = 1e-4, r=\epsilon, \eta = \frac{1}{2\ell}, m = d$  \\ 
\hline
PAGD      & $\ell= e, \rho = e, \eta = \frac{1}{4\ell}, r = \frac{e}{100}, t_{\textmd{thresh}}=1, g_{\textmd{thresh}} = \frac{e\gamma}{100}$       \\ 
\hline
ZO-GD-NCF & $\ell= e, \rho = e, \eta = \frac{1}{4\ell}, p=0.01$ \\ 
\hline
\multicolumn{2}{c}{\textbf{d = 10}}\\ 
\hline
RSPI      & $\ell= e, \rho = e, \sigma_1 = 1, \sigma_2 = 1.25, \rho_{\sigma_1} = 0.95, T_{\sigma_1} = 20, T_{\textmd{DFPI}} = 20$       \\ 
\hline
\multicolumn{2}{c}{\textbf{d = 30}}\\ 
\hline
RSPI      & $\ell= e, \rho = e, \sigma_1 = 1, \sigma_2 = 1.25, \rho_{\sigma_1} = 0.95, T_{\sigma_1} = 20, T_{\textmd{DFPI}} = 20$       \\ 
\hline
\multicolumn{2}{c}{\textbf{d = 50}}\\ 
\hline
RSPI      & $\ell= e, \rho = e, \sigma_1 = 1, \sigma_2 = 1.25, \rho_{\sigma_1} = 0.85, T_{\sigma_1} = 20, T_{\textmd{DFPI}} = 20$       \\ 
\hline
\multicolumn{2}{c}{\textbf{d = 100}}\\ 
\hline
RSPI      & $\ell= e, \rho = e, \sigma_1 = 0.75, \sigma_2 = 0.5, \rho_{\sigma_1} = 0.9, T_{\sigma_1} = 20, T_{\textmd{DFPI}} = 20$       \\ 
\hline
\end{tabular}
\end{center}
\end{table}

\subsection{Comparison between ZO-GD-NCF and Neon2+GD on Octopus Function}

In order to show that our ZO-GD-NCF method will not significantly increase the iteration complexity compared to Neon2+GD \cite{allen2018neon2}, we compare the iteration performance between ZO-GD-NCF and Neon2+GD on octopus function. The parameters corresponding to the octopus function are set with $\tau = e, L = e, \gamma =1$. We initialize the two algorithms at point $(0, \dots, 0)^\T$.

We set $\epsilon = 1e-4, \delta = \sqrt{\rho \epsilon}$ for all experiments  and report the function value v.s. number of iterations in Figure \ref{fig: zo_ncf_gd-vs-neon2}.

\begin{figure}[!htb]
\centering
\subfigure[d=30]{
\centering
\includegraphics[width=0.4\textwidth]{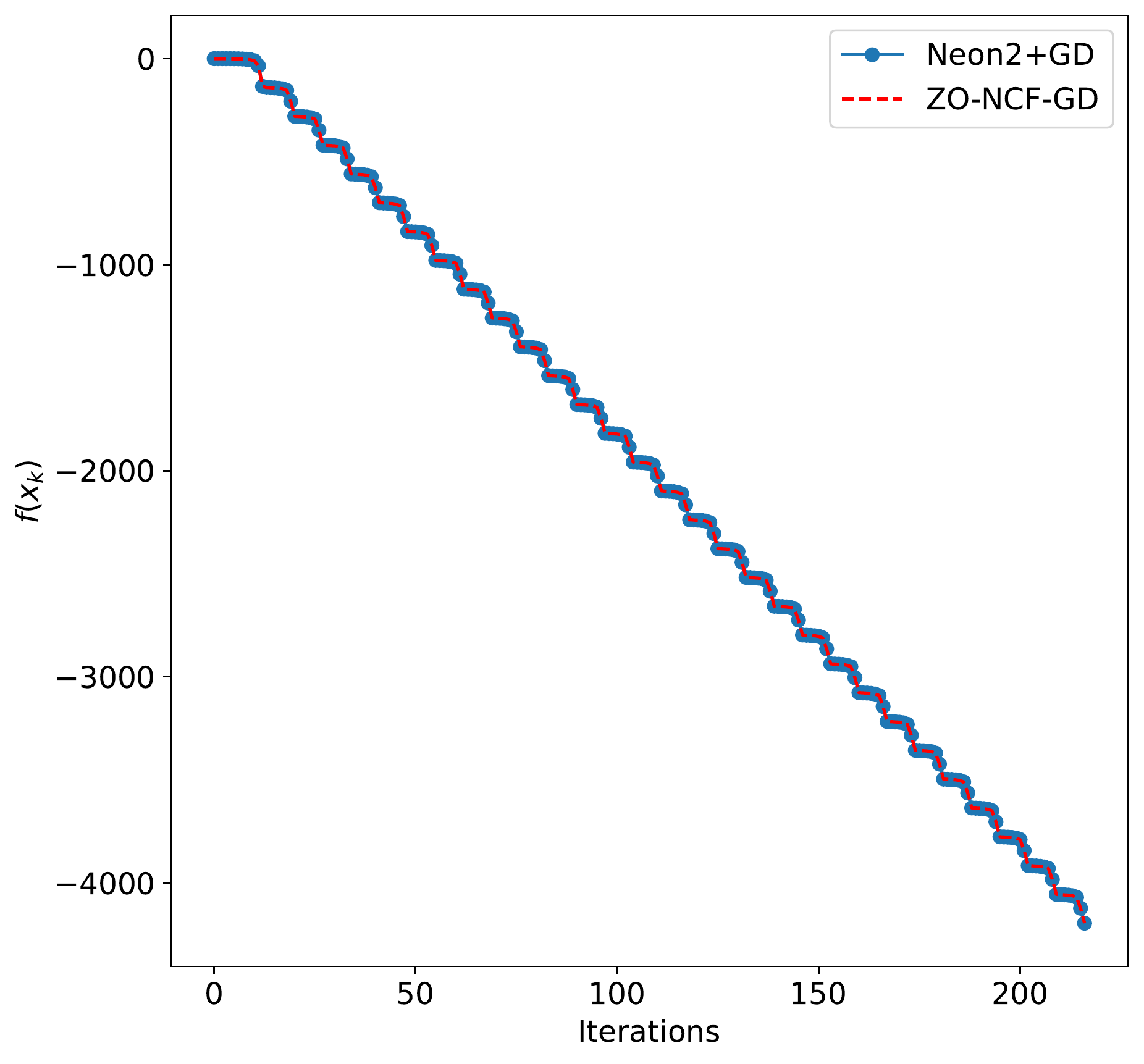}
}
\subfigure[d=50]{
\centering
\includegraphics[width=0.4\textwidth]{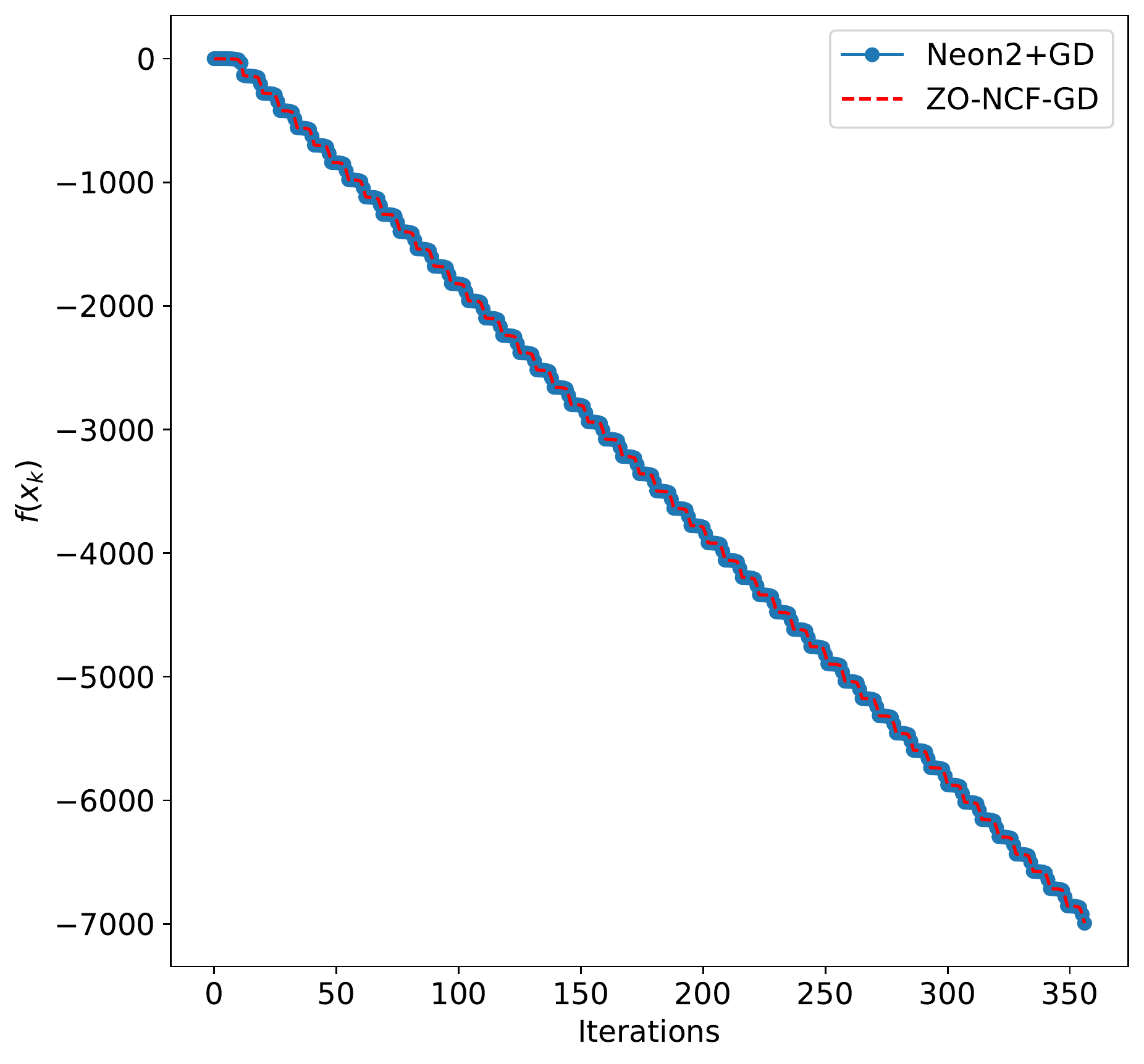}
}

\caption{Comparison of ZO-GD-NCF and Neon2+GD for solving octopus function problem.}
\label{fig: zo_ncf_gd-vs-neon2}
\end{figure}

The results in Figure \ref{fig: zo_ncf_gd-vs-neon2} clearly shows that ZO-GD-NCF have almost the same iteration performance with Neon2+GD. The detailed parameter settings are stated in the following table.

\begin{table}[htb]
\caption{Choices of parameters of ZO-GD-NCF and Neon2+GD.}
\begin{center}
\begin{tabular}{cc}
\hline
\textbf{Algorithm} & \multicolumn{1}{c}{\textbf{Parameters}} \\ 
\hline
ZO-GD-NCF & $\ell= e, \rho = e, \eta = \frac{1}{2\ell}, p=0.01$ \\ 
\hline
Neon2+GD  & $\ell= e, \rho = e,  \eta = \frac{1}{2\ell}, p=0.01$  \\ 
\hline
\end{tabular}
\end{center}
\end{table}

\subsection{Cubic Regularization Problem}
To test performance of the proposed methods for both deterministic setting and stochastic setting. We consider the cubic regularization problem \cite{liu2018adaptive}, which is defined as:
\begin{equation}
\label{eq: cubic regularization problem}
    \min_{w \in \R^d} \frac{1}{2}w^\T A w + b^\T w + \frac{\alpha}{3}\|w\|^3.
\end{equation}
For the deterministic setting, we generate a diagonal $A$ such that 10\% randomly selected diagonal entries are -1 and the rest diagonal entries are uniformly randomly chosen from [1, 2], and set $b$ to a zero vector. For the stochastic setting, we let $A = A' + \E[\textrm{diag}(\xi)]$ and $b = \E[\xi']$, where $A'$ is generated the same way as that in the deterministic setting, $\xi$ are uniformly randomly chosen from [-0.1, 0.1] and $\xi'$ are uniformly randomly chosen from $[-1, 1]$. The parameter $\alpha$ in Eq.~\eqref{eq: cubic regularization problem} is set to $0.5$ for both deterministic setting and stochastic setting. We set $\epsilon =10^{-2}, \delta = \sqrt{\rho \epsilon}$ for all experiments. To test the ability of different algorithms to escape from saddle points, we initialize all algorithms at a saddle point $(0, \dots, 0)^\T$.

\begin{figure}[!htb]
\centering
\subfigure[Deterministic, d=100]{
\centering
\includegraphics[width=0.23\textwidth]{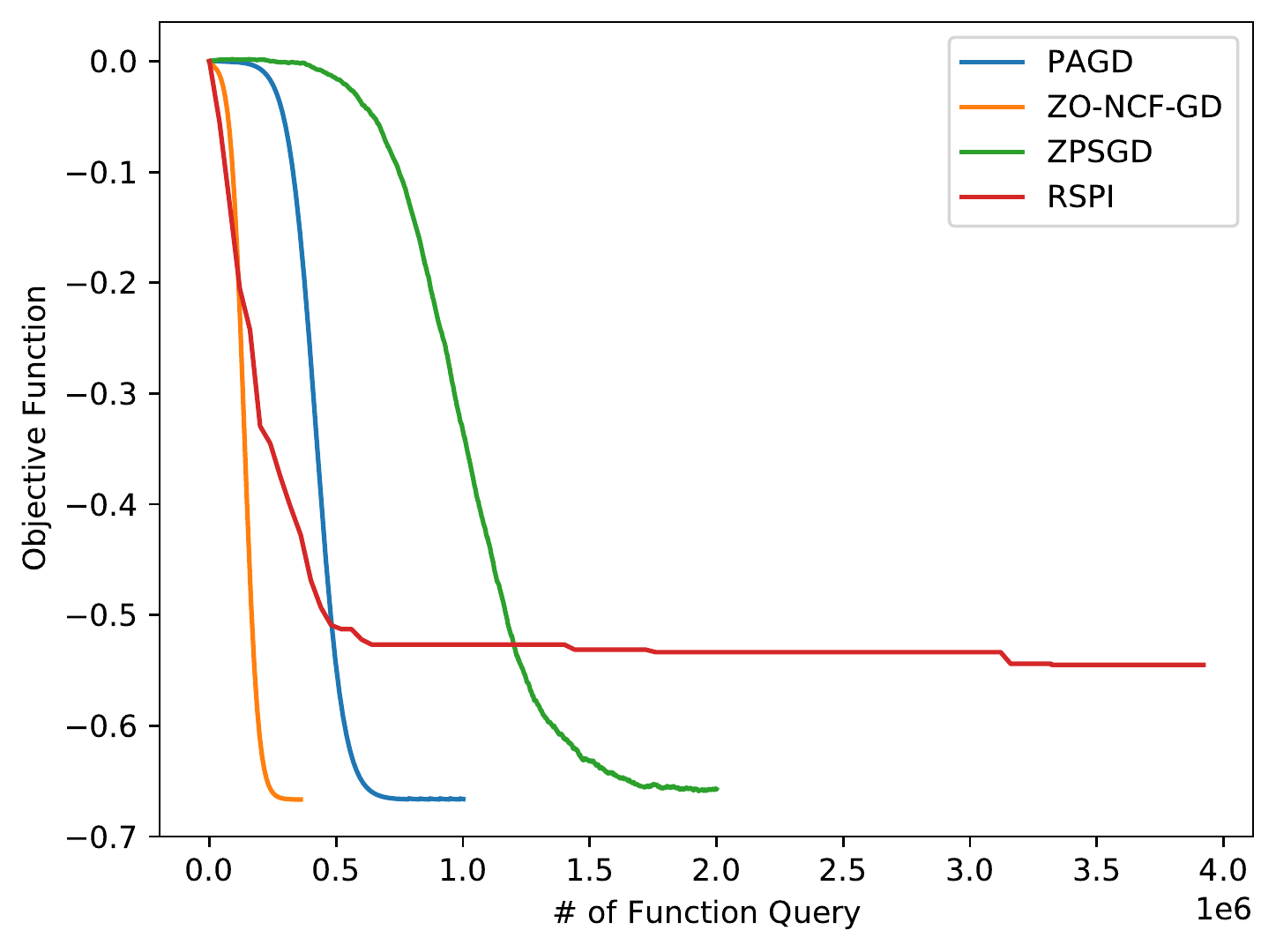}
}
\subfigure[Deterministic, d=200]{
\centering
\includegraphics[width=0.23\textwidth]{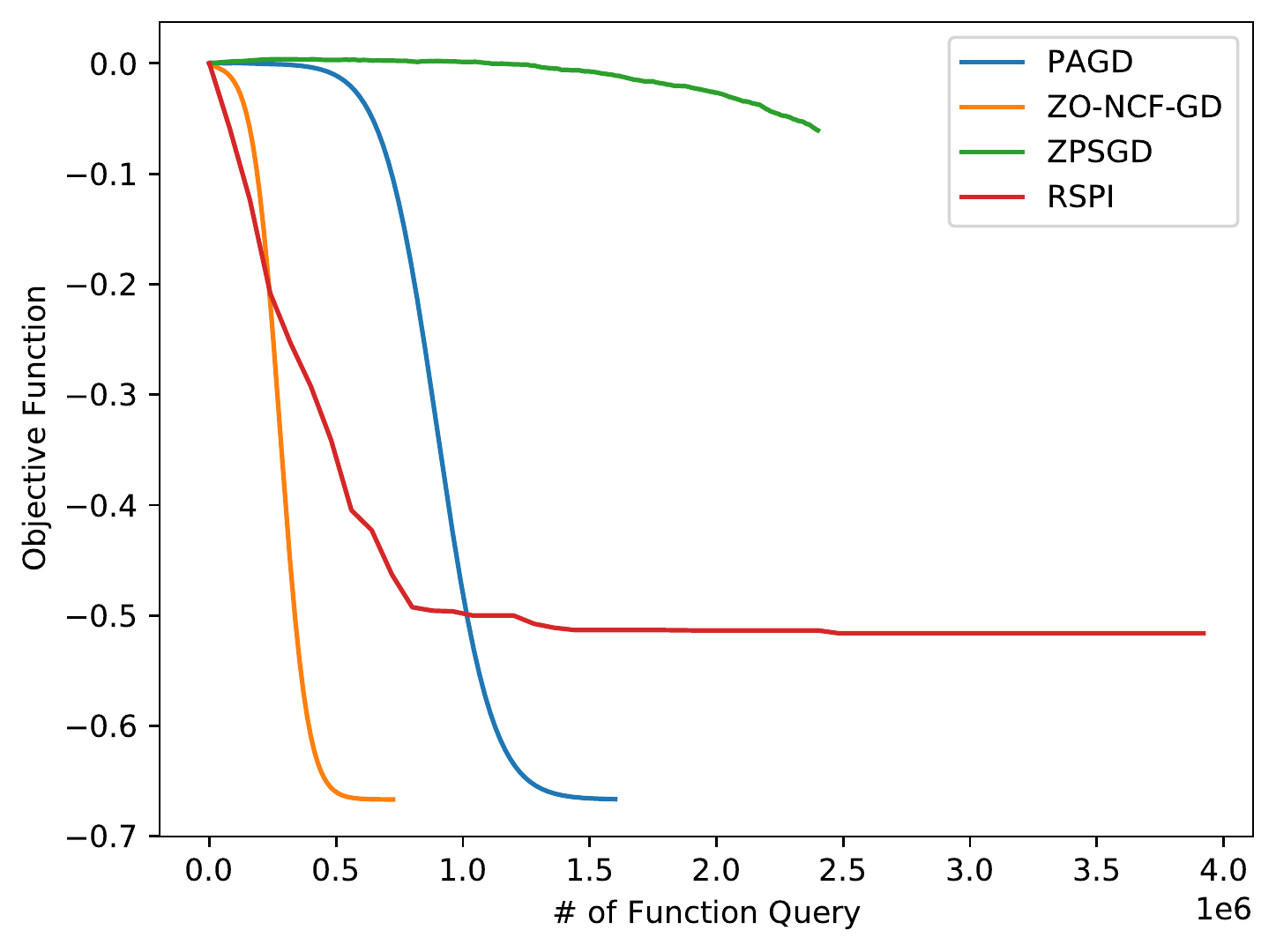}
}
\subfigure[Stochastic, d=20]{
\centering
\includegraphics[width=0.23\textwidth]{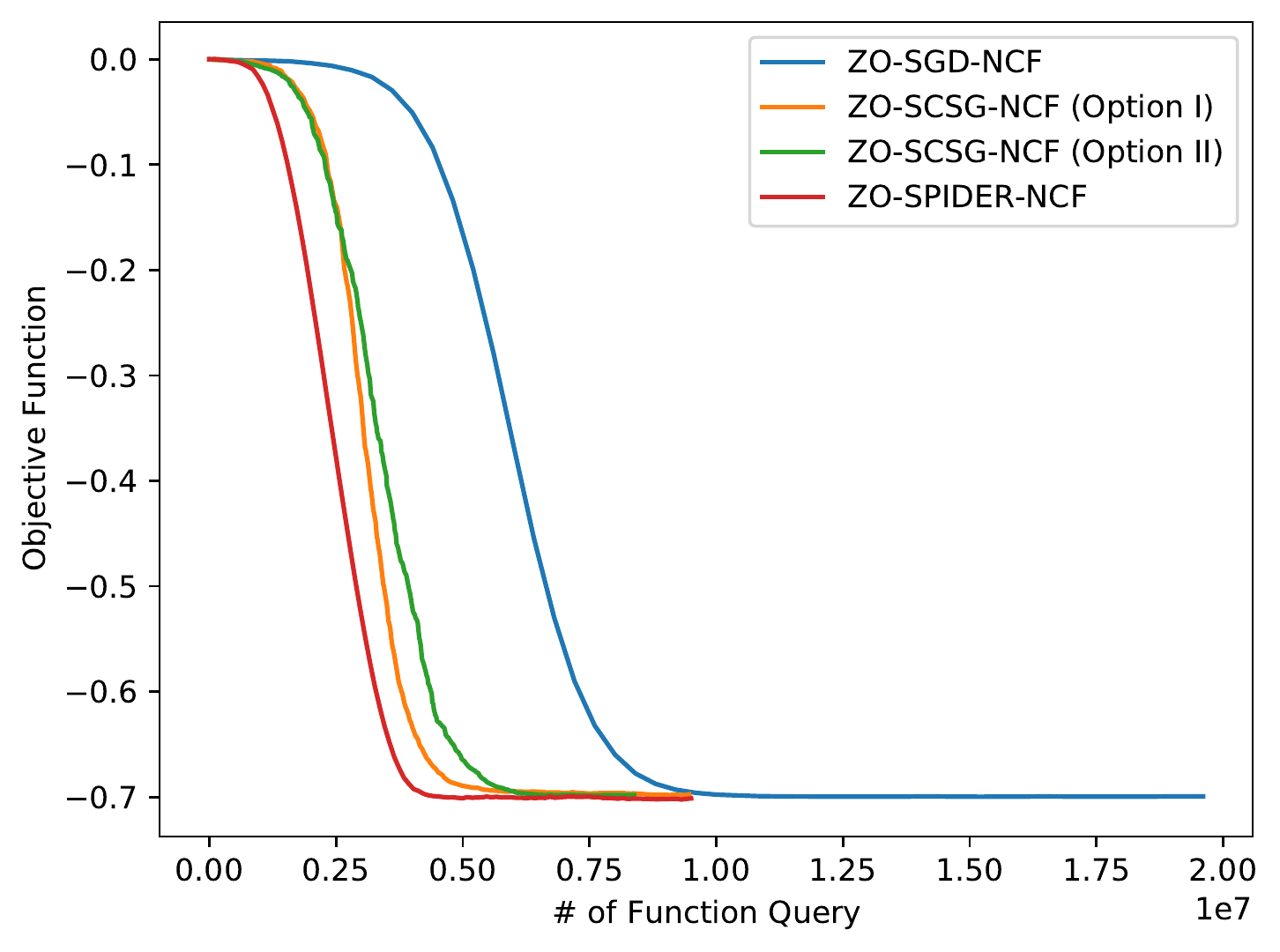}
}
\subfigure[Stochastic, d=100]{
\centering
\includegraphics[width=0.23\textwidth]{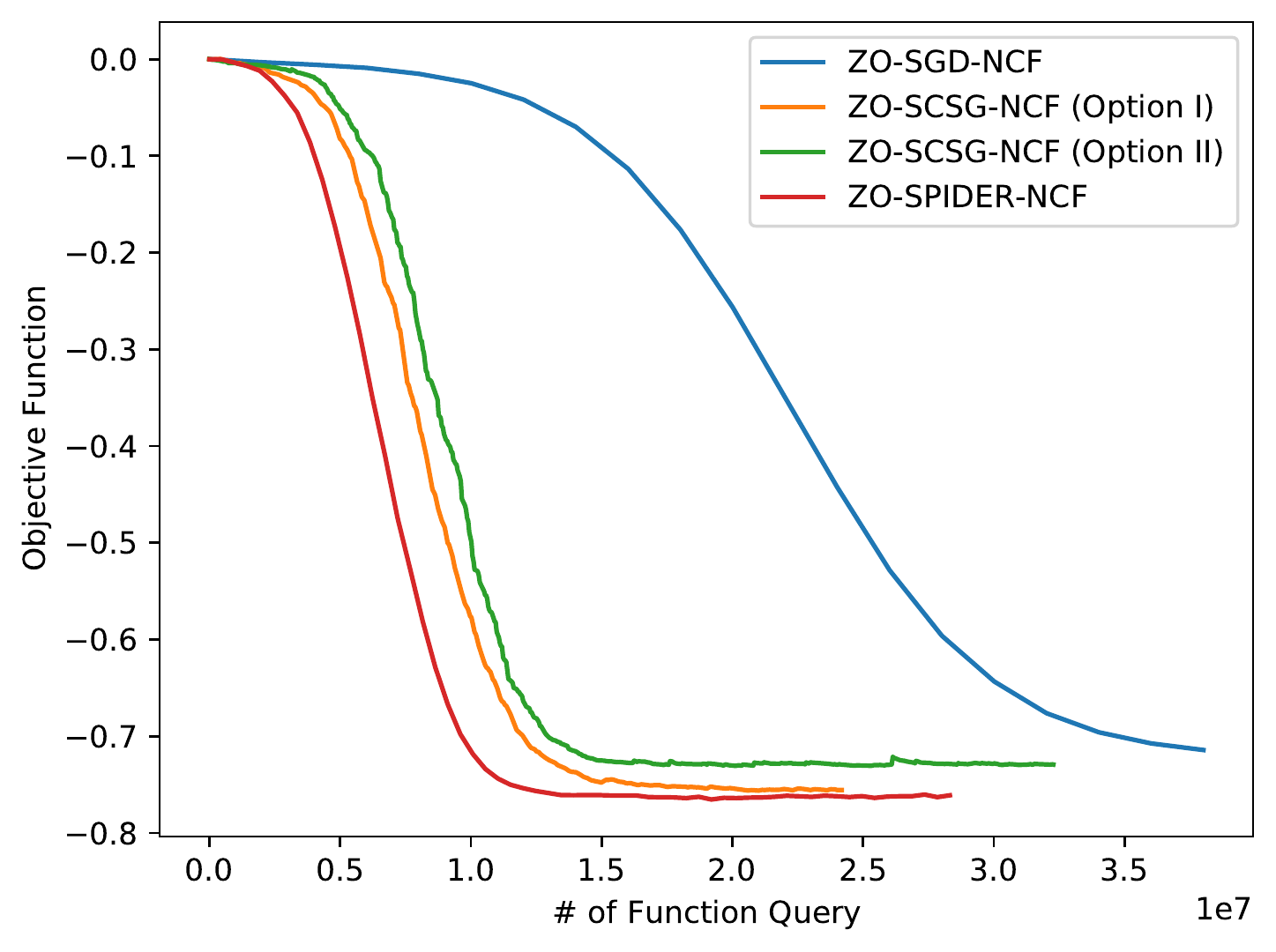}
}

\caption{Comparison of different algorithms for solving cubic regularization problem in deterministic setting and stochastic setting.}
\label{fig: cubic-regularization-problem}
\end{figure}

In deterministic setting, the results in Fig.~\ref{fig: cubic-regularization-problem} (a) and (b) illustrate that,  $\left. 1 \right)$ the negative curvature finding based algorithms (ZO-GD-NCF, RSPI) can escape saddle points more efficient than the random perturbation based algorithms (ZPSGD, PAGD). $\left. 2\right)$ On the other hand, the gradient estimation based algorithms (ZO-GD-NCF, ZPSGD, PAGD) converge faster than the random search based algorithm (RSPI). In the stochastic setting, ZO-SPIDER-NCF converges faster than other three algorithms.

\begin{table}[htb]
\caption{Choices of parameters for the experiment of the cubic regularization problem.}
\begin{center}
\begin{tabular}{ll}
\hline
\textbf{Algorithm} & \multicolumn{1}{c}{\textbf{Parameters}} \\ 
\hline
\multicolumn{2}{c}{\textbf{Deterministic, d = 100}}\\ 
\hline
ZPSGD     & $\ell= 10^2, \rho = 1, r=\epsilon, \eta = \frac{1}{2\ell}, m = d$  \\ 
\hline
PAGD      & $\ell= 10^2, \rho = 1, \eta = \frac{1}{4\ell}, r = 10^{-2}, t_{\textmd{thresh}}=1, g_{\textmd{thresh}} = 10^{-2}$       \\ 
\hline
RSPI      & $\ell= 10^2, \rho = 1, \sigma_1 = 0.4, \sigma_2 = 0.4, \rho_{\sigma_1} = 0.98, T_{\sigma_1} = 10, T_{\textmd{DFPI}} = 100$       \\ 
\hline
ZO-GD-NCF & $\ell= 10^2, \rho = 1, \eta = \frac{1}{4\ell}, p=0.01$ \\ 
\hline
\multicolumn{2}{c}{\textbf{Deterministic, d = 200}}\\ 
\hline
ZPSGD     & $\ell= 10^2, \rho = 1, r=\epsilon, \eta = \frac{1}{2\ell}, m = d$  \\ 
\hline
PAGD      & $\ell= 10^2, \rho = 1, \eta = \frac{1}{4\ell}, r = 10^{-2}, t_{\textmd{thresh}}=1, g_{\textmd{thresh}} = 10^{-2}$       \\ 
\hline
RSPI      & $\ell= 10^2, \rho = 1, \sigma_1 = 0.4, \sigma_2 = 0.4, \rho_{\sigma_1} = 0.98, T_{\sigma_1} = 10, T_{\textmd{DFPI}} = 100$       \\ 
\hline
ZO-GD-NCF & $\ell= 10^2, \rho = 1, \eta = \frac{1}{4\ell}, p=0.01$ \\ 
\hline
\multicolumn{2}{c}{\textbf{Stochastic, d = 20, d=100}}\\ 
\hline
ZO-SGD-NCF & $\ell = 10^2, \rho = 1, \eta = \frac{1}{3\ell}, |S|=128, p = 0.01$  \\
\hline
ZO-SCSG-NCF &\textbf{Option \uppercase\expandafter{\romannumeral1} }: $\ell = 10^2, \rho=1, \eta = \frac{1}{4\ell}, B =128, b=10, p=0.01$  \\
\hline
ZO-SCSG-NCF &\textbf{Option \uppercase\expandafter{\romannumeral2} }: $\ell = 10^2, \rho=1, \eta = \frac{1}{10\ell}, B =128, b=10, p=0.01$  \\
\hline
ZO-SPIDER-NCF & $\ell =10^2, \rho=1, \eta = \frac{1}{15\ell}, |\mathcal{S}_1| = 128, |\mathcal{S}_2| = 10, p=0.01 $ \\
\hline
\end{tabular}
\end{center}
\end{table}

\subsection{Regularized Non-Linear Least-Square Problem}
We next consider the regularized non-linear least-square problem \cite{liu2018adaptive}, which is defined as:
\begin{equation}
\label{eq: regularized non-linear least-square}
    \min_{w\in \R^d} \frac{1}{n} \sum_{i=1}^n (y_i - \sigma(w^\T x_i))^2 + \sum_{i=1}^d \frac{\lambda w_i^2}{1 + \alpha w_i^2},
\end{equation}
where $x_i \in \R^d$, $y_i \in \{0, 1\}$, $\sigma(s) = \frac{1}{1+\exp(-s)}$, and the second term is a non-convex regularizer. We use the w1a data (n=2477, d=300) which can be downloaded from the LIBSVM website \cite{CC01a}. We set $\lambda=1, \alpha =1$ in Eq.\eqref{eq: regularized non-linear least-square} and $\epsilon=1e-2, \delta = \sqrt{\rho \epsilon}$ for both deterministic setting and stochastic setting. We report the objective function value v.s. the number of function queries in Figure~\ref{fig: regularized non-linear least-square problem} and can draw similar conclusions to the previous experiment.

\begin{figure}[!htb]
\centering
\subfigure{
\centering
\includegraphics[width=0.4\textwidth]{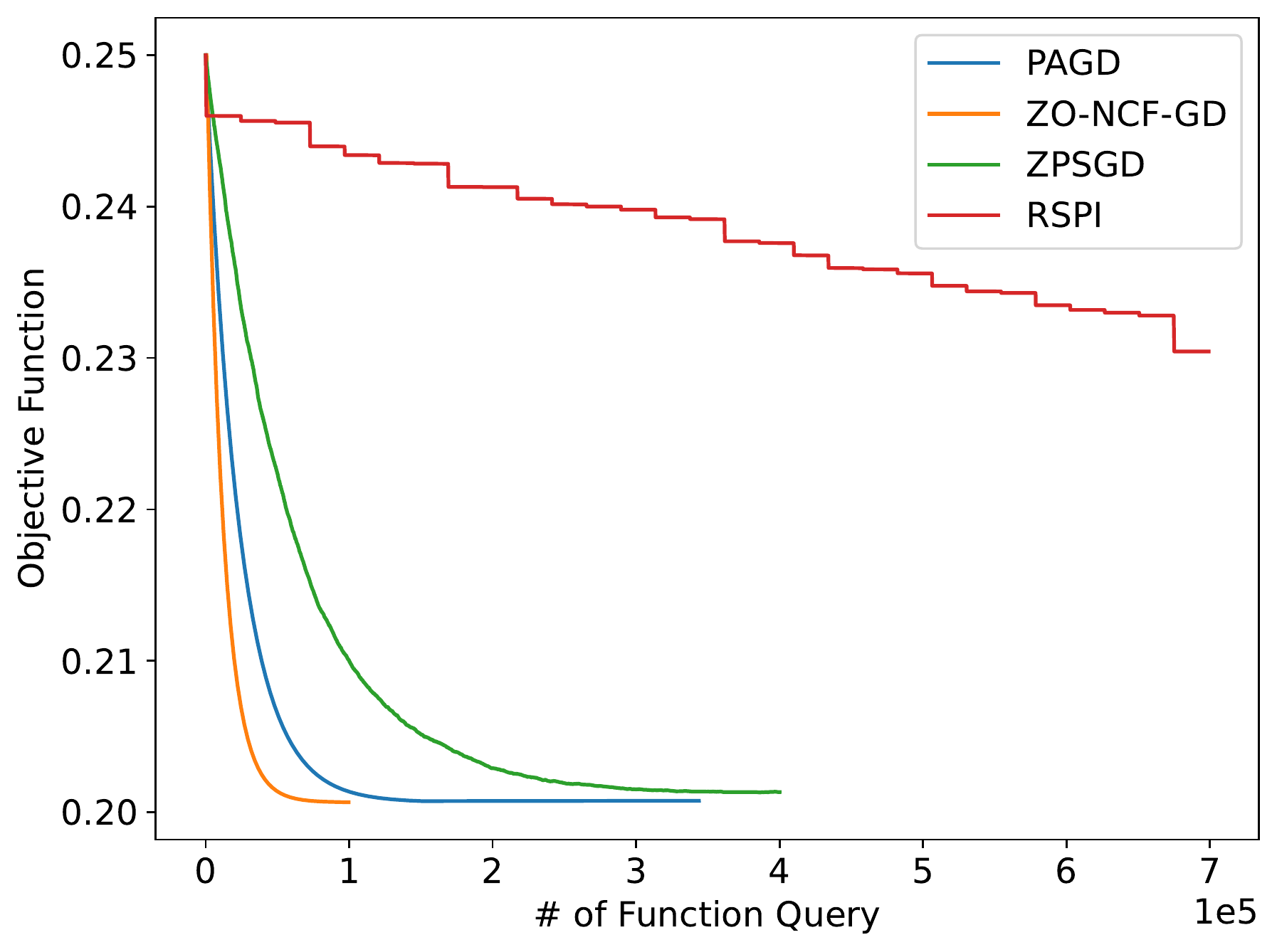}
}
\subfigure{
\centering
\includegraphics[width=0.4\textwidth]{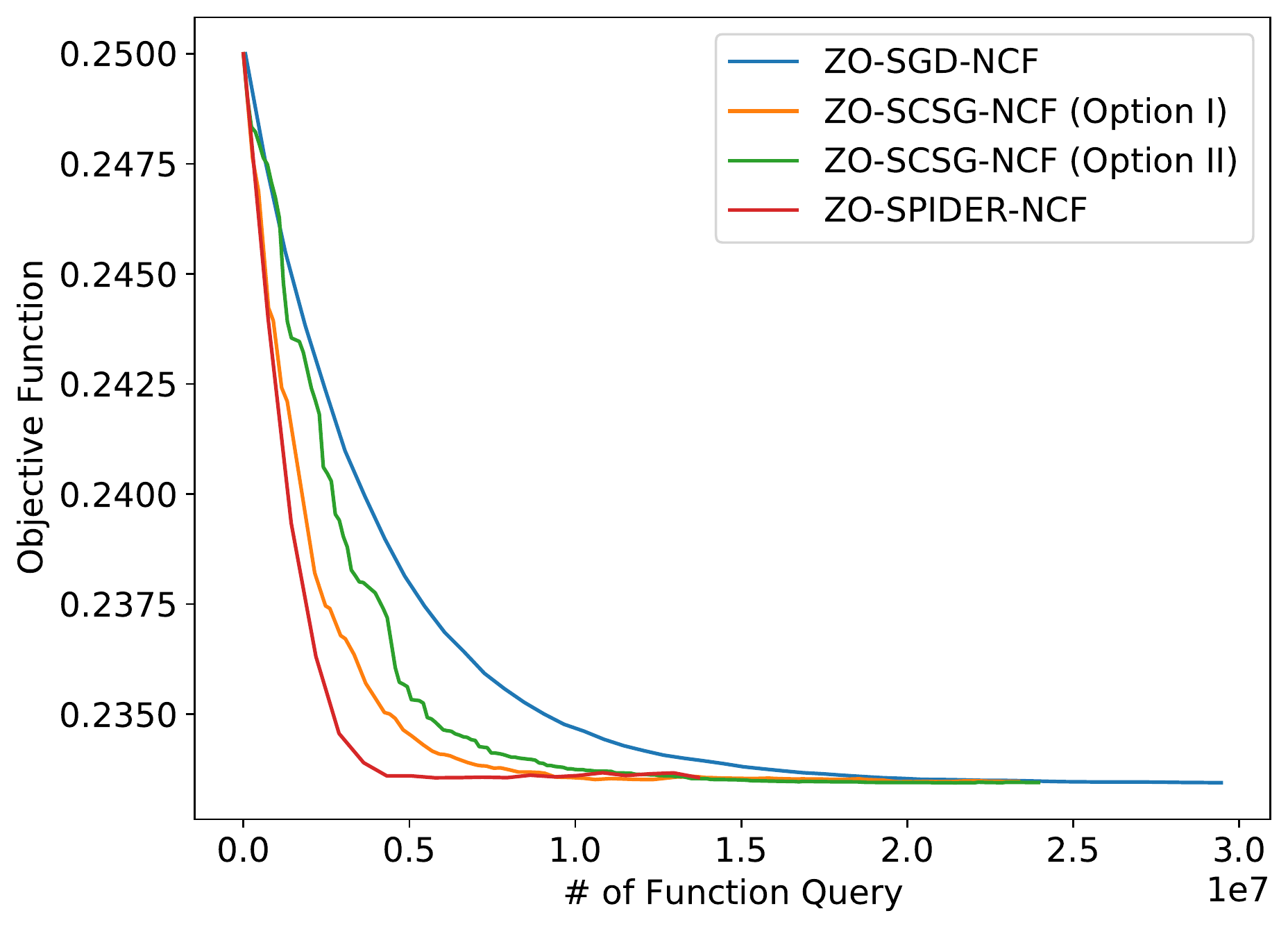}
}

\caption{Comparison of different algorithms for solving regularized non-linear least-square problem in deterministic setting and stochastic setting}
\label{fig: regularized non-linear least-square problem}
\end{figure}

\begin{table}[htb]
\caption{Choices of parameters for the experiment of the regularized non-linear least-square problem.}
\begin{center}
\begin{tabular}{ll}
\hline
\textbf{Algorithm} & \multicolumn{1}{c}{\textbf{Parameters}} \\ 
\hline
\multicolumn{2}{c}{\textbf{Deterministic}}\\ 
\hline
ZPSGD     & $\ell= 10^2, \rho = 1, r=\epsilon, \eta = \frac{1}{2\ell}, m = d$  \\ 
\hline
PAGD      & $\ell= 10^2, \rho = 1, \eta = \frac{1}{4\ell}, r = 10^{-2}, t_{\textmd{thresh}}=1, g_{\textmd{thresh}} = 10^{-2}$       \\ 
\hline
RSPI      & $\ell= 10^2, \rho = 1, \sigma_1 = 0.4, \sigma_2 = 0.4, \rho_{\sigma_1} = 0.98, T_{\sigma_1} = 10, T_{\textmd{DFPI}} = 100$       \\ 
\hline
ZO-GD-NCF & $\ell= 10^2, \rho = 1, \eta = \frac{1}{4\ell}, p=0.01$ \\ 
\hline
\multicolumn{2}{c}{\textbf{Stochastic}}\\ 
\hline
ZO-SGD-NCF & $\ell = 10^2, \rho = 1, \eta = \frac{1}{3\ell}, |S|=128, p = 0.01$  \\
\hline
ZO-SCSG-NCF &\textbf{Option \uppercase\expandafter{\romannumeral1} }: $\ell = 10^2, \rho=1, \eta = \frac{1}{4\ell}, B =128, b=10, p=0.01$  \\
\hline
ZO-SCSG-NCF &\textbf{Option \uppercase\expandafter{\romannumeral2} }: $\ell = 10^2, \rho=1, \eta = \frac{1}{10\ell}, B =128, b=10, p=0.01$  \\
\hline
ZO-SPIDER-NCF & $\ell =10^2, \rho=1, \eta = \frac{1}{15\ell}, |\mathcal{S}_1| = 128, |\mathcal{S}_2| = 10, p=0.01 $ \\
\hline
\end{tabular}
\end{center}
\end{table}

\end{document}